\documentclass[11pt]{article}

 \usepackage{url}            
 \usepackage{booktabs}       
 \usepackage{amsfonts}       
 \usepackage{nicefrac}       
 \usepackage{microtype}      

\usepackage{fncylab,enumerate,wrapfig,placeins}
\usepackage{bbm}
\usepackage{graphicx}
\usepackage{enumitem}  
\usepackage{amsmath,amsthm,amsfonts,amssymb}  
\usepackage{pifont}

\usepackage{geometry}

\pdfoutput=1  

 \usepackage{url}            
 \usepackage{booktabs}       
 \usepackage{amsfonts}       
 \usepackage{nicefrac}       
 \usepackage{microtype}      
 
  \usepackage{pifont}

\usepackage{graphicx}

\newtheorem{theorem}{Theorem}
\newtheorem{proposition}{Proposition}
\newtheorem{lemma}{Lemma}
\newtheorem{corollary}{Corollary}
 
\usepackage{enumerate,wrapfig,placeins}
\usepackage{bbm}
\usepackage{enumitem}  

\usepackage{booktabs}
\usepackage{multirow}
\usepackage{algorithm}

\DeclareFontFamily{U}{mathx}{}
\DeclareFontShape{U}{mathx}{m}{n}{ <-> mathx10 }{}
\DeclareSymbolFont{mathx}{U}{mathx}{m}{n}
\DeclareFontSubstitution{U}{mathx}{m}{n}

\DeclareMathAccent{\widebar}{0}{mathx}{"73}

\usepackage[usenames,dvipsnames]{color}

\RequirePackage[colorlinks,citecolor=BlueViolet,urlcolor=BlueViolet]{hyperref}

\usepackage[nameinlink,noabbrev]{cleveref}

\definecolor{darkred}{RGB}{153,0,0} 

\hypersetup{colorlinks,linkcolor=Blue}

\Crefname{corollary}{Corollary}{Corollaries}
\Crefname{eqnarray}{eq.}{eqs.}
\Crefname{equation}{eq.}{eqs.}
\Crefname{figure}{Fig.}{Figs.}
\Crefname{tabular}{Tab.}{Tabs.}
\Crefname{table}{Tab.}{Tabs.}
\Crefname{lemma}{Lemma}{Lemmas}
\Crefname{proposition}{Prop.}{Propositions}
\Crefname{theorem}{Thm.}{Thms.}
\Crefname{definition}{Def.}{Defs.} 
\Crefname{section}{Section}{Sections}
\Crefname{assumption}{Assumption}{Assumptions}
\Crefname{exmp}{Example}{Examples}
\Crefname{exercise}{Exercise}{Exercises}

 \usepackage{textgreek,upgreek,bm}

\newcommand{\mydash}{{\hbox{\rule{1ex}{.2ex}}}}
 
\def\bdd#1{b_{\text{\rm\tiny\ref{#1}}}}
\def\bdds#1{\clE^{\text{\rm\tiny\ref{#1}}}}
\def\bdde#1{\varrho^{\text{\rm\tiny\ref{#1}}}}

\def\Obj{\Upgamma}  

\def\bexp{b_{\text\rm\sf e}}
\def\rhoexp{\varrho_{\text\rm\sf e}}
\def\CEdist{W}
\def\Trelax{T_{\text{\sf r}}}
\def\Crelax{\varrho_{\text{\sf r}}}

\graphicspath{{figures/}}

\newlength{\noteWidth}
\long\def\notes#1{\ifinner
{\tiny #1}
\else 
\marginpar{\parbox[t]{\noteWidth}{\raggedright\tiny#1}}
\fi\typeout{#1}}

\def\notes#1{}

\def\bl#1{{\color{blue}#1}}         

\def\mindex#1{\index{#1}}

\def\sq{\hbox{\rlap{$\sqcap$}$\sqcup$}}
\def\qed{\ifmmode\sq\else{\unskip\nobreak\hfil
\penalty50\hskip1em\null\nobreak\hfil\sq
\parfillskip=0pt\finalhyphendemerits=0\endgraf}\fi\medskip}

\long\def\defbox#1{\framebox[.9\hsize][c]{\parbox{.85\hsize}{%
\parindent=0pt
\baselineskip=12pt plus .1pt      
\parskip=6pt plus 1.5pt minus 1pt 
 #1}}}

\long\def\beginbox#1\endbox{\subsection*{}%
\hbox{\hspace{.05\hsize}\defbox{\medskip#1\bigskip}}%
\subsection*{}}

\def\endbox{}

\def\transpose{{\intercal}}

\def\darrow{\buildrel{\rm d}\over\longrightarrow}

\newsavebox{\junk}
\savebox{\junk}[1.6mm]{\hbox{$|\!|\!|$}}
\def\lll{{\usebox{\junk}}}

\def\limsup{\mathop{\rm lim\ sup}}

\def\state{{\sf X}}

\def\bx{{{\cal B}(\state)}}

\newcommand{\field}[1]{\mathbb{#1}}

\def\posRe{\field{R}_+}
\def\Re{\field{R}}

\def\One{\mbox{\rm{\large{1}}}}

\def\ind{\field{I}}

\def\intgr{\field{Z}}

\def\bfmath#1{{\mathchoice{\mbox{\boldmath$#1$}}%
{\mbox{\boldmath$#1$}}%
{\mbox{\boldmath$\scriptstyle#1$}}%
{\mbox{\boldmath$\scriptscriptstyle#1$}}}}

\def\bfmB{\bfmath{B}}

\def\bfmD{\bfmath{D}}

\def\bfmQ{\bfmath{Q}}

\def\bfmY{\bfmath{Y}}

\def\bfmhhaY{\bfmath{\hhaY}} 
\def\bfmhhaY{\hbox to 0pt{$\widehat{\bfmY}$\hss}\widehat{\phantom{\raise 1.25pt\hbox{$\bfmY$}}}}

\def\bfmW{\bfmath{W}}

\def\bfPhi{\bfmath{\Phi}}

\def\haP{{\widehat P}}

\def\haf{{\hat f}}
\def\hag{{\hat g}}

\def\til={{\widetilde =}}

\def\tilP{{\widetilde P}}

\def\tiltheta{\widetilde \theta}

\def\tiltheta{{\tilde \theta}}

\def\clB{{\cal B}}

\def\clD{{\cal D}}
\def\clE{{\cal E}}
\def\clF{{\cal F}}
\def\clG{{\cal G}}

\def\clN{{\cal N}}

\def\clP{{\cal P}}

\def\clR{{\cal R}}
\def\clS{{\cal S}}
\def\clT{{\cal T}}

\def\clX{{\cal X}}

\def\clZ{{\cal Z}}

 \def\FRAC#1#2#3{\genfrac{}{}{}{#1}{#2}{#3}}

\def\ddt{{\mathchoice{\FRAC{1}{d}{dt}}%
{\FRAC{1}{d}{dt}}%
{\FRAC{3}{d}{dt}}%
{\FRAC{3}{d}{dt}}}}

\def\ddtp{{\mathchoice{\FRAC{1}{d^{\hbox to 2pt{\rm\tiny +\hss}}}{dt}}%
{\FRAC{1}{d^{\hbox to 2pt{\rm\tiny +\hss}}}{dt}}%
{\FRAC{3}{d^{\hbox to 2pt{\rm\tiny +\hss}}}{dt}}%
{\FRAC{3}{d^{\hbox to 2pt{\rm\tiny +\hss}}}{dt}}}}

\def\half{{\mathchoice{\FRAC{1}{1}{2}}%
{\FRAC{1}{1}{2}}%
{\FRAC{3}{1}{2}}%
{\FRAC{3}{1}{2}}}}

\def\fourth{{\mathchoice{\FRAC{1}{1}{4}}%
{\FRAC{1}{1}{4}}%
{\FRAC{3}{1}{4}}%
{\FRAC{3}{1}{4}}}}

\def\eqdef{\mathbin{:=}}

\def\Prob{{\sf P}}

\def\Expect{{\sf E}}

\def\average#1,#2,{{1\over #2} \sum_{#1}^{#2}}

\def\eye(#1){{\bf(#1)}\quad}
\def\epsy{\varepsilon}

\def\varble{\,\cdot\,}

\newcommand{\beqn}[1]{\notes{#1}%
\begin{eqnarray} \elabel{#1}}

\newcommand{\eeqn}{\end{eqnarray} }

\newcommand{\beq}[1]{\notes{#1}%
\begin{equation}\elabel{#1}}

\newcommand{\eeq}{\end{equation}}

\def\bdes{\begin{description}}
\def\edes{\end{description}}

\def\barf{{\overline {f}}}

\def\barA{{\overline {A}}}

\newcounter{rmnum}

\newenvironment{romannum}{\smallbreak%
\begin{list}{{\upshape (\roman{rmnum})}}{\usecounter{rmnum}
\setlength{\leftmargin}{18pt}
\setlength{\listparindent}{15pt}
\setlength{\topsep}{0pt}
\setlength{\parsep}{0pt} 
\setlength{\rightmargin}{8pt}
\setlength{\itemindent}{5pt}
}}{\smallbreak\end{list}}

\newcounter{anum}

{\end{list}}

\def\ass(#1:#2){(#1\ref{#1:#2})}

\def\ritem#1{
\item[{\sf \ass(\current_model:#1)}]
}

\newenvironment{recall-ass}[1]{%
\begin{description}
\def\current_model{#1}}{
\end{description}
}

\newcommand{\bd}{\begin{description}}
\newcommand{\ed}{\end{description}}
\newcommand{\bt}{\begin{theorem}}
\newcommand{\et}{\end{theorem}}
\newcommand{\ba}{\begin{array}{rcl}}
\newcommand{\ea}{\end{array}}

\def\thetaPR{\theta^{\text{\tiny\sf  PR}}}
\def\tilthetaPR{\tilde{\theta}^{\text{\tiny\sf  PR}}}

\def\zPR{z^{\text{\tiny\sf  PR}}}

\def\tTheta{{\text{\tiny$\Theta$}}}

\def\SigmaTheta{\Sigma_\tTheta}

\def\thetaPR{\theta^{\text{\tiny\sf  PR}}}
\def\tilthetaPR{\tilde{\theta}^{\text{\tiny\sf  PR}}}
\def\SigmaPR{\Sigma^{\text{\tiny\sf PR}}_\tTheta}

\makeatletter
\newcommand\gobblepars{%
    \@ifnextchar\par%
        {\expandafter\gobblepars\@gobble}%
{}}
\makeatother

\def\wham#1{\smallbreak\pagebreak[3]%
	\noindent\textbf{#1}\ \ \gobblepars}

\def\whamc{\wham{$\circ$}}

\def\whamem#1{\wham{\emph{#1}}}
\def\Oops{\Upsilon}  
\def\haOops{\widehat{\Upsilon}}  

\def\sumalphaW{\clS^W}
\def\ScaledST{\mathcal{P}}
\def\STy{\mathcal{P}^A}

\theoremstyle{plain}

\def\state{{\sf X}}
\def\eqdef{\mathbin{:=}}

\def\Cov{\text{\rm Cov}\,}

\def\ind{\field{I}}
\def\posRe{\field{R}_+}
\def\Re{\field{R}}

\def\hatDelta{\widehat \Delta}
\def\MD{\zeta}
\def\MDstar{\MD^*}

\def\Mart{M}
\def\MartPR{H}
\def\haclE{\widehat{\clE}}
\def\haMD{\hat{\zeta}}
\def\haclT{\widehat \clT}

\def\diffalpha{\gamma}

\def\barpi{\widebar{\uppi}}

\def\bardiffalpha{\bar{\gamma}}

\def\Vdrift{\updelta}  

\def\barL{\bar L}
\def\barb{\bar b}

\def\odestate{\upvartheta}

\def\bfodestate{\bm{\odestate}} 
\def\hatodestate{\hat{\odestate}}
\def\hattheta{\hat \theta}
\def\odestatepull#1#2{\odestate_{#1\mid \mydash}^{(#2)}}
\def\ODEstate{\Uptheta}

\def\tilf{\tilde{f}}
\def\tilg{\tilde{g}}

\def\bfODEstate{\bm{\ODEstate}}

\def\tilODEstate{\widetilde{\ODEstate}}

\def\hatODEstate{\widehat{\ODEstate}}

\def\scerrorSymbol{Z}
\def\scerror#1#2{\scerrorSymbol_{#1}^{(#2)}}

\def\intDrift{Y}

\def\scerrorD#1#2{\intDrift_{#1}^{(#2)}}
\def\scerrorMart#1#2{M_{#1}^{(#2)}}
\def\scerrorpull#1#2{\scerror{#1 \mid \mydash}{#2}}
\def\scerrorpullDintDrift#1#2{\scerrorD{#1 \mid \mydash}{#2}}
\def\scerrorpullMart#1#2{\scerrorMart{#1 \mid \mydash}{#2}}

\def\sclim{X}

\def\barfinf{\barf_{\infty}}
\def\SAtime{\uptau}

\def\Lv{L_\infty^v}
\def\La{L^{\text{a}}}
\def\Lb{L^{\text{b}}}

 \title{The ODE Method for Asymptotic Statistics \\
	 in Stochastic Approximation and Reinforcement Learning%
 }

\author{Vivek Borkar\thanks{Dept. of Electrical Engineering, Indian Institute of Technology Bombay, Powai, Mumbai 400076, India.  V.B. was supported in part by a
S. S. Bhatnagar Fellowship.
}%
		\and
		Shuhang Chen\thanks{Department of Mathematics at the University of Florida, Gainesville}%
		\and
		Adithya Devraj\thanks{Department of Electrical Engineering, Stanford University, Stanford, CA 94305.}%
		\and 		
		Ioannis Kontoyiannis\thanks{
		Statistical Laboratory, University of Cambridge, UK.
		{\tt yiannis@maths.cam.ac.uk.}
I.K. was supported in part by the Hellenic Foundation for Research 
		and Innovation (H.F.R.I.) under the ``First Call for H.F.R.I. 
		Research Projects to support Faculty members and Researchers and 
		the procurement of high-cost research equipment grant'', Project 
		ID: HFRI-FM17-1034.}%
		\and 
		Sean  Meyn\thanks{Department of ECE at the University of Florida, Gainesville.  		{\tt meyn@ece.ufl.edu}.
		Financial support from ARO award W911NF2010055
and National Science Foundation grants  
CCF 2306023,
EPCN 1935389
are gratefully acknowledged.
}
}

\begin{document}

 \maketitle

\begin{abstract}

The paper concerns 
the stochastic approximation recursion,
\[
\theta_{n+1}= \theta_n + \alpha_{n + 1} f(\theta_n, \Phi_{n+1})
	    \,,\quad n\ge 0,
\]
where the {\em estimates} $\theta_n\in\Re^d$,  and $\bfPhi \eqdef \{ \Phi_n \}$ 
is a stochastic process 
on a general state space, satisfying a conditional Markov property that allows for parameter-dependent noise.  
In addition to standard Lipschitz assumptions
and conditions on the vanishing step-size sequence,
it is assumed that the associated \textit{mean flow}  
$ \ddt \odestate_t = \barf(\odestate_t)$, is globally asymptotically stable,
with stationary point denoted $\theta^*$.
\notes{Too complex to explain:
\\
where $\barf(\theta)=\Expect[f(\theta,\Phi)]$ with $\Phi$
having the stationary distribution of the chain.}
The main results are established under additional conditions on the mean flow and   a version of the Donsker-Varadhan  Lyapunov drift condition
known as~(DV3) for $\bfPhi$:

\wham{(i)}
A Lyapunov function is constructed 
that implies convergence of the estimates
in $L_4$.

\wham{(ii)}   
A functional central limit theorem (CLT) is established,  as well as the usual 
one-dimensional CLT for the normalized 
error.
Moment bounds combined with the CLT imply convergence of the normalized 
covariance
$\Expect [ z_n z_n^\transpose  ]$
to the asymptotic covariance 
$\SigmaTheta$ in the CLT,
where $z_n\eqdef (\theta_n-\theta^*)/\sqrt{\alpha_n}$.   

\wham{(iii)}    
The CLT holds for the normalized version
  $\zPR_n\eqdef  \sqrt{n} (\thetaPR_n -\theta^*)$,
of the averaged parameters
$\thetaPR_n$,
subject to standard assumptions on the step-size.
Moreover, the covariance of   $\zPR_n$
converge to $\SigmaPR$,
the minimal covariance of Polyak and Ruppert.
 
\wham{(iv)} An example is given
where $f$ and $\barf$ are linear in $\theta$, and 
$\bfPhi$
is a geometrically ergodic 
Markov chain 
but does not satisfy~(DV3).
While the algorithm is convergent,  the second moment of $\theta_n$
is unbounded and in fact diverges.

\smallskip

\noindent
\bl{\textbf{This arXiv version represents a major extension of the results in prior versions:}   
The main results now allow for parameter-dependent noise, as is often the case in applications to reinforcement learning.}
 
\smallskip

\noindent
\textit{MSC2020 subject classifications:}  Primary  62L20;   secondary 60F17, 68T05 93E35, 60J20.
\end{abstract}

%

\section{Introduction}
\label{s:intro}

The stochastic approximation (SA) method of Robbins and 
Monro~\cite{robmon51a} was designed to solve  the $d$-dimensional    
root-finding problem $\barf(\theta^*) = 0$,  
where $\barf\colon\Re^d\to\Re^d$ is defined as the expectation  
$\barf(\theta)\eqdef \Expect[f(\theta, \Phi)]$, $\theta\in\Re^d$,  with 
$f:\Re^d\times\state$ and $\Phi$ a random variable with values in $\state$.  
Interest in SA has 
grown over the past decade with increasing interest in reinforcement 
learning (RL) and other ``stochastic algorithms''~\cite{sutbar18,CSRL,bac21}.  

Algorithm design begins with recognition that $\theta^* $ is a stationary point of the \textit{mean flow},
\begin{equation}
  \ddt \odestate_t = \barf(\odestate_t).
\label{e:ODESA1}
\end{equation}
A standard assumption is that this ODE is globally asymptotically stable, so that in particular solutions converge to $\theta^*$ from each initial condition $\odestate_0\in\Re^d$.    In applications, the first step in the ``ODE method'' is to construct the function $f$ that determines $\barf$ so that stability and other desirable properties are satisfied.  

The mean flow  can be approximated using an Euler scheme.   Subject to 
standard Lipschitz conditions, it is recognized that  the Euler approximation
is robust to ``measurement noise'',  which motivates the SA recursion,
\begin{equation} 
\theta_{n+1}= \theta_n + \alpha_{n + 1} f(\theta_n, \Phi_{n+1}),
\label{e:SAa}
\end{equation}
where $\{\alpha_n\}$ is the non-negative 
step-size sequence, and the distribution of $\Phi_{n}$ converges to that of $\Phi$ as $n\to\infty$.     
Writing $ \Delta_{n+1} \eqdef f(\theta_n, \Phi_{n+1}) - \barf(\theta_n) $,  the interpretation as a noisy Euler approximation is made explicit:
\begin{equation}
\theta_{n+1} = \theta_n +\alpha_{n+1} [ \barf(\theta_n) + \Delta_{n+1} ]\,,\quad n\ge 0.
\label{e:SAintro}
\end{equation}

 Analysis  of the SA recursion traditionally proceeds by comparison with solutions of the mean flow~\eqref{e:ODESA1}~\cite{bor20a,benmetpri12}. 
This requires  a time-scaling:
Take, $\SAtime_0=0$ and define,
\begin{equation}
\SAtime_{k+1}  = \SAtime_k + \alpha_{k+1},\qquad k\ge 0.
\label{eq:tau}
\end{equation}
Two continuous-time processes are then compared:

\wham{(i)}  \textit{Interpolated parameter process}:
\begin{equation}
\text{$ \ODEstate_t  = \theta_k$   when $t=\SAtime_k$, for each $k\ge 0$,    }
\label{e:ODEstate}
\end{equation}
 and defined for all $t$ through piecewise linear interpolation.

\wham{(ii)}  \textit{Re-started ODE}:   For each $n\ge 0$,   let 
$\{ \odestate^{(n)}_t :  t\ge \SAtime_n \}$ denote the solution 
to~\eqref{e:ODESA1},  initialized according to the current parameter estimate:
\begin{equation}
\ddt \odestate^{(n)}_t  =\barf(\odestate^{(n)}_t )\,,  \quad  t\ge \SAtime_n\,,\qquad   \odestate^{(n)}_{\SAtime_n}  =\theta_n.
\label{e:ODEn}
\end{equation}

Iteration of~\eqref{e:SAintro} gives,   for any $0<n<K$,
\begin{equation}
\begin{aligned}
 \ODEstate_{\SAtime_K}    =   \theta_n   +   \sum_{i=n}^{K-1}   \alpha_{i+1}  \barf( \theta_i)  +   \sum_{i=n}^{K-1} \alpha_{i+1} \Delta_{i+1}  
=     \ODEstate_{\SAtime_n}    +    \int_{\SAtime_n}^{\SAtime_K} \barf(\ODEstate_\tau )\, d\tau + \clE_{K}^{(n)} , 
\end{aligned} 
\label{e:InterationOfSAintro}
\end{equation}
where $\clE_{K}^{(n)}  $ is the sum of cumulative disturbances and the error resulting from the Riemann-Stieltjes  approximation of the integral.  
The disturbance term $\clE_{K}^{(n)}  $ will vanish with $n$ uniformly in $K$ subject to conditions on $\{\Delta_i\}$ and the step-size.  This is, 
by definition, the ODE approximation 
of $\{\theta_n\}$~\cite{bor20a,benmetpri12}.

A closer inspection of the theory shows that two ingredients are required to establish convergence of SA: The first is stability of the mean flow~\eqref{e:ODESA1}, as previously noted.   The second requirement is that the parameter sequence $\{ \theta_n\}$ obtained from the SA recursion is bounded with probability one.     Rates of convergence through the 
central limit theorem (CLT) and the law of the iterated logarithm  require further assumptions~\cite{benmetpri12,kusyin97,kon02,bor20a,gapkra75,mokpel05}.

There are two well-known approaches to establishing boundedness of the parameter sequence based on properties of the mean-flow.   
The relationship between these two approaches is discussed 
in~\cite[Ch.~4]{CSRL}  and  in~\cite{vid23,vid22}.

\wham{Lyapunov criterion:}  The existence of a smooth function 
$V\colon\Re^d\to\Re_+$ that has a Lipschitz gradient $\nabla  V$, 
along with $\epsy>0$ 
 satisfying $\barf(\theta)^\transpose \nabla V\, (\theta)  \le -\epsy \|  \theta -\theta^*  \|^2$ for all $\theta$.    See~\cite{karmiamouwai19} for examples of this approach and further history.
 
\wham{ODE@$\infty$:}  Suppose that the following limit exists for each $\theta$:   
\begin{equation}
\barfinf(\theta)\eqdef
\lim_{r\to\infty} r^{-1}   \barf(r\theta).
\label{e:barfinfty}
\end{equation}
This defines the vector field for the ODE@$\infty$ of~\cite{bormey00a,bor20a}.   
The  so-called (see, e.g.,~\cite{bha11,rambha17}) ``Borkar-Meyn theorem'' 
states that: If this ODE   is globally asymptotically stable,   
and the process $\bfPhi=\{\Phi_n\}$ is such that the sequence $\{\Delta_n\}$ 
appearing in~\eqref{e:SAintro} is a martingale difference sequence,  
then the sequence $\{ \theta_n\}$  is bounded with probability one.  
Relaxations of the assumptions of~\cite{bormey00a} are given 
in~\cite{bha11,rambha17}.   Reference~\cite{rambha19} presents an 
extension of~\cite{bormey00a} in which $\bfPhi=\{\Phi_n\}$ 
is parameter dependent.    The setting is 
adversarial:  It is assumed that the ODE@$\infty$ is stable under 
the worst-case noise sequence.   

\smallskip

The advantage of the ODE@$\infty$ over Lyapunov techniques is that $\barfinf$ is often very simple compared to the vector field $\barf$,   and its stationary point is always the origin.   Consider for example stochastic gradient descent for the  Rastrigin loss function $\Gamma$, 
whose gradient is given by $[\nabla \Obj \, (x) ]_i=   [2  x_{i}  + b\sin(2 \uppi x_{i}) ]$, with $b>0$.    In this case $\barf =  -   \nabla \Obj $ results in    $\barfinf = -2x$.   

Of the many applications of the results in~\cite{bormey00a},
the majority are in the context of RL. However, two essential assumptions 
in this prior work appear to be often overlooked and may indeed be
violated, especially in RL applications:   
\wham{$\bullet$}
 The martingale difference sequence assumption 
in~\cite{bormey00a} holds only in very special cases, 
such as \textit{tabular} Q-learning; 
the martingale difference property 
is exploited in~\cite{sze97} to obtain error bounds for this special case.  
\wham{$\bullet$}   It is assumed in~\cite{bormey00a}  that $\bfPhi$ is a Markov chain, which rules out parameter-dependent exploration, such as $\epsy$-greedy policies.

%

\smallskip

One of the main aims of this paper is to rectify these problems.
We provide much more realistic assumptions on the 
``noise'' process $\bfPhi$, which are valid in a broad range of applications 
and lead to far stronger conclusions than anticipated in prior research.     
Specifically, we consider a {\em family} of Markov processes
$\bfPhi^\theta$ and assume that the evolution of each
$\bfPhi^\theta $ is specified by a transition kernel $P^\theta$
from family 
$\{P_\theta : \theta\in\Re^d \}$;
these define a conditional Markov property as 
in~\cite{pezheu97,rambha19}. For example, such a parameter 
dependent model may take the form, 
\begin{equation}
\Phi_{n+1}  =  g(\Phi_n,\theta_n,W_{n+1}), 
\label{e:PhiParDependent}
\end{equation}
in which $\bfmW=\{W_n\}$ is an independent and identically distributed 
(i.i.d.) sequence.    Such models are required,
e.g., in analysis of Q-learning, of  off-policy TD-learning with 
parameter-dependent exploration, 
and of actor-critic methods~\cite{sutbar18,CSRL}.

It is assumed that each transition kernel $P_\theta$   satisfies~(DV3)---a slightly weaker version of the well-known Donsker-Varadhan
Lyapunov drift condition~\cite{donvarI-II, donvarIII}---which is 
used in~\cite{konmey03a,konmey05a} 
to establish exact large deviations asymptotics for the partial sums of functions of a Markov chain; see Section~\ref{s:MCs}.  Details and connections with other drift conditions are discussed in the next section.

The main results we obtain are the following.
We hope to make the theory more accessible by conveying the 
main ideas in the body of the paper, and leaving tedious calculations 
to the Appendix.  

\newpage

\noindent
{\bf Contributions}
\begin{itemize}
\item
Subject to Lipschitz bounds on $f$ and $\barf$,  a Lyapunov function is constructed for the joint process $\{\theta_n,\Phi_n\}$.     
It satisfies a drift condition that implies boundedness of $\{ \theta_n\}$  
almost surely and in $L_4$. 
(\Cref{t:BigConvergence,t:BigBounds}.)

\item
Conditions are provided that ensure almost sure convergence for the 
parameter sequence,   and also an associated CLT:   
$z_n \darrow N(0,\SigmaTheta)$, where the convergence is in distribution, 
with $z_n \eqdef \tiltheta_n/ \sqrt{\alpha_n},\,$  
$\tiltheta_n \eqdef \theta_n - \theta^*$,   
and with the covariance $\SigmaTheta$ explicitly identified.   
A corresponding functional CLT (FCLT)  is also established  (\Cref{t:FCLT,t:CLT}). 

\end{itemize}

\noindent
 Convergence and the FCLT follow quickly from the strong bounds in~\Cref{t:BigBounds} combined with techniques in  the existing literature~\cite{bor20a}. 
The following results are not in the current literature in any form, 
except in a few very special cases discussed below.
 
\begin{itemize}

\item
Conditions are provided to ensure the sequence $\{z_n \}$ is bounded  
in $L_4$, and  convergence of the normalized covariance is established in~\Cref{t:CLT}:
\begin{equation}
	\lim_{n \to \infty} \frac{1}{\alpha_n} 
	\Expect [ \tiltheta_n\tiltheta_n^\transpose  ] 
	=	\lim_{n \to \infty}  \Expect [ z_n  z_n^\transpose  ] = \SigmaTheta \, .
\label{e:MSESA}
\end{equation}
These results have significant implications for
estimates obtained using the averaging technique of Polyak and Ruppert:   
\begin{equation}
\thetaPR_n = \frac{1}{n}\sum_{k=1}^n \theta_k \,,\qquad n\ge 1.
\label{e:thetaPR}
\end{equation}

\item
\Cref{t:PR}:
Suppose that the step-size
$ \alpha_{n + 1}  =1/ (n+1)^\rho$ is used 
in~\eqref{e:SAa}, with $\half < \rho < 1$.  Under the same conditions
as above, the CLT holds for the normalized sequence 
$\zPR_n\eqdef  \sqrt{n} \tilthetaPR_n$,
with $\tilthetaPR_n = \thetaPR_n -\theta^*$.   
Moreover,  the rate of convergence is optimal,
in that,
\begin{equation}
	\lim_{n \to \infty} n \Expect [ \tilthetaPR_n (\tilthetaPR_n)^\transpose  ] 
		= 	\lim_{n \to \infty}  \Expect [ \zPR_n (\zPR_n)^\transpose  ] = \SigmaPR \,,
\label{e:PR}
\end{equation}
where $\SigmaPR$ is   minimal in a matricial sense~\cite{rup88,pol90}.   
The proof (in~\Cref{s:PR}) also establishes the refinement,   
\begin{equation}
 \Expect [  \| A^*  \tilthetaPR_n  +  \tfrac{1}{n} \MartPR^*_{n}  \|^2  ]   \le  \bdd{t:PRbig}  \alpha_n ^2
  \,,\qquad   n\ge 1 \,, 
\label{e:PRbig}
\end{equation}
in which the matrix $A^*$ is defined in~\eqref{e:def-Q} below,   $\{\MartPR^*_k :  k\ge 1\}$ is a martingale,  and the constant $ \bdd{t:PRbig} $  is independent of $\alpha$, but depends on $\rho$ and the initial condition $(\theta_0,\Phi_0)$.  This bound easily implies~\eqref{e:PR} based on the definition of  $\SigmaPR$ in~\Cref{t:PR}.

\notes{who cares?
\\ The error term  can be reduced to    $O(n^{-1-\rho} )$ for linear SA with $f(\theta,\Phi) = A(\Phi)\theta-b(\Phi)$, subject to strong conditions on  $\{ A(\Phi_n) : n\ge 1\}$.
}

\end{itemize}

\notes{Note new emphasis here that DV3 is almost necessary}

\noindent
The discussion surrounding~\Cref{t:DV3multBdd} indicates that 
condition~(DV3) cannot be dropped. On the other hand,~(DV3) 
is not overly restrictive, e.g., it holds for linear state space 
models with disturbance whose marginal has Gaussian tails,  
and the continuous-time version holds for the Langevin diffusion 
under mild conditions~\cite{konmey05a}.  

However,~(DV3) \textit{fails} for  Markov chains on Euclidean space 
that are \textit{skip-free}, i.e., when the increments $\Phi_{n+1} - \Phi_n$,
$n\geq 0$, are deterministically bounded.   For example, 
the M/M/1 queue in discrete time is geometrically 
ergodic under the standard load assumption~\cite[Ch.~3]{CTCN}, 
but it does not satisfy~(DV3). 
This motivates the 
further contribution:

\begin{itemize}

\item
An example is given in which 
the assumptions of~\Cref{t:BigConvergence} hold;     a   scalar model 
in which $\barf(\theta)=-\theta$.   
Consequently,  $\displaystyle 
\lim_{n\to\infty} \theta_n = \theta^*=0$ with probability one from each initial condition.   
The driving noise is in some sense ideal: The (parameter-independent) Markov chain  $\bfPhi$ is constructed so that it   is reversible and geometrically ergodic.      It also satisfies the L-mixing conditions imposed in~\cite{ger92}. 
However,~(DV3) does \textit{not} hold.
Therefore, not only is there is no available theory to obtain moment bounds
but indeed it is shown in~\Cref{t:unbounded-mom} that the second moment
diverges:
$\displaystyle 
\lim_{n\to\infty}  \Expect [ \| \theta_n \|^2  ] = \infty 
$.
\end{itemize}

\wham{Prior work} 
There are   few results on convergence of moments,
such as~\eqref{e:MSESA} or~\eqref{e:PR},
in the SA literature.
Most closely related to 
this paper is~\cite{ger92}  which contains the error 
bound $\Expect[\| \tiltheta_n\|^q]^{1/q} = O(1/\sqrt{n})$ 
for \textit{every} $q\ge 1$, but this conclusion is
for a version of SA that requires
\textit{resetting}: A compact region $\Theta^0\subset\Re^d$ 
is specified,  along with a   state $\theta^0\in\Theta^0$,  
and the algorithm sets $\theta_{n+1}=\theta^0$ if 
$\theta_n\not\in\Theta^0$. It is assumed that $\theta^*$ lies in 
the interior of $\Theta^0$.    The major statistical assumptions include 
Condition~1.1: $\{f(\theta, \Phi_{n+1})\}$  {\em and} 
$\{\partial_\theta 
f(\theta, \Phi_{n+1})\}$ are L-mixing,   
and Condition~1.2:  For some $\epsy_0>0$,  
$\sup_n\Expect[ \exp(\epsy_0 f(\theta, \Phi_{n+1}) )]$ is uniformly 
bounded for $\theta$ in compact sets.    It is not clear how these 
conclusions can be extended to the present setting,
or if finer results such as~\eqref{e:PR}
are even achievable.   
 
The optimal asymptotic variance in the CLT for SA and techniques to obtain 
the optimum for scalar recursions were introduced by Chung~\cite{chu54o},
soon after the introduction of SA.  See also~\cite{sacks1958asymptotic,fab68} 
for early work,   and surveys 
in~\cite{zhu96,pezheu97,benmetpri12,kusyin97,bor20a}.   
The averaging technique came in the independent work of Ruppert~\cite{rup88}  and Polyak and Juditsky~\cite{pol90,poljud92};   see~\cite{bacmou11} for an elegant technical summary, and~\cite{durmounausam24} for the best results for linear SA. 
The reader is referred to~\cite{benmetpri12,kusyin97,bor20a}  for more history on the substantial literature on asymptotic statistics for SA.

The roots of the ODE@$\infty$ can be traced back to the fluid model techniques 
for stability of networks~\cite{dai95a,daimey95a,CTCN}, which were extended to 
skip-free  Markov chains in~\cite{formeymoupri08a}.
See~\cite{bor20a,CSRL} for further 
history since~\cite{bormey00a}.    

Much of the recent literature on convergence rates for SA 
is motivated by applications to 
RL, and seeks finite-$n$ error bounds rather than asymptotic 
results of the form considered here.   
The article~\cite{sriyin19} treats temporal difference (TD) learning 
(a particular RL algorithm) with Markovian noise, which justifies the linear 
SA recursion of the form $f(\theta,\Phi) = A(\Phi)\theta + b(\Phi)$. 
The primary focus is on the constant step-size setting, with parallel results 
for vanishing step-size sequences. 
The uniform bounds assumed in~\cite{sriyin19} will hold under uniform 
ergodicity and boundedness  
of $A$ and $b$ as functions of $\Phi$.  
Improved bounds are obtained in~\cite{durmounausamscawai21}, but 
subject to i.i.d.\ noise;  
this paper also contains a broad survey of results in this domain.  
While not precisely finite-$n$ bounds as 
in~\cite{sriyin19}, they are asymptotically tight since they 
are refinements of the CLT.     

Finite-$n$ bounds are obtained in~\cite{bacmou11,bacmou13} 
for stochastic gradient descent,
with martingale difference noise.   
This statistical assumption is relaxed in~\cite{durmounausamwai21,ruidyvan23}  
for linear SA recursions;   the assumptions of~\cite{durmounausamwai21} include~(DV3),  along with conditions on the functions $A$ and $b$ that 
are related to the bounds imposed here.
Most  of the effort is devoted to algorithms with constant step-size combined with averaging as in~\eqref{e:thetaPR}.
 
In~\cite{josheu00} a functional law of the iterated logarithm is obtained for 
\textit{linear} SA recursions with constant step-size, and subject to an L-mixing condition similar to~\cite{ger92}.
\notes{This is a sweet little paper.   Both Heunis references are great.}

%
%
%
%

  
\section{Background and Assumptions}

This section contains some of the main assumptions that will
remain in effect for most of our result, and an outline
of important background material. 

\subsection{SA model}
\label{s:SAmodel}

The stochastic process $\bfPhi=\{\Phi_n\}$ evolves on a Polish state 
space  $\state $, equipped with its Borel sigma-algebra $\bx$.  Its dynamics 
are defined by a family of transition kernels $\{ P_\theta : \theta\in\Re^d\}$.
In the special case~\eqref{e:PhiParDependent} we define  $P_\theta(x,A) 
\eqdef  \Prob\{ g(x,\theta,W_{n+1})   \in A \}$,   for each $n\ge 0$, $\theta\in\Re^d  \,, \ x \in\state\,,\  A\in \bx$.     
This is consistent with the general model in which,
\begin{equation}
\Prob\{ \Phi_{n+1}\in A \mid  \clF_n ;   \theta_n = \theta\,, \  \Phi_n=x \}  =  P_\theta(x,A)\,,    
\quad \text{with} \ \clF_n = \sigma (\theta_k\,,\ \Phi_k : k\le n ).
\label{e:condMarkov}
\end{equation}

While $\bfPhi $ is not necessarily Markovian,  analysis is based on 
consideration of the Markov chain  $\bfPhi^\theta $ with transition kernel 
$P_\theta$,  for   $\theta\in\Re^d$ fixed. For the 
realization~\eqref{e:PhiParDependent}, this evolves according to 
the recursion,  
$\Phi_{n+1}^\theta \eqdef g(\Phi_n^\theta,\theta,W_{n+1}) $,  $ n\ge 0$.
%

It is assumed that each $\bfPhi^\theta$ is geometrically ergodic  with unique probability invariant  measure~$\pi_\theta$.  The mean flow vector field is then defined as the expectation,
\begin{equation}
\barf(\theta)  =  \int f(\theta, x)  \, \pi_\theta(dx).
\label{e:barf}
\end{equation}
For any $\bx$-measurable function $g\colon\state\to\Re^m$
and any measure $\mu$ on $(\state,\clB(\state))$,
we use the compact notation $\mu(g) \eqdef\int g(x)\, \mu(dx)$,
whenever the integral is well defined.

\subsection{Markov chains}
\label{s:MCs}

 We next lay out notation for a single Markov chain $\bfPhi $ on $\state$ with transition kernel $P$.   As above,   it is assumed that $\bfPhi$ is geometrically ergodic with unique probability invariant 
measure~$\pi$.   

A pair $(C,\nu)$ are called \textit{small} if $C\in\bx$ and $\nu$ is a 
probability measure on $(\state,\bx)$, such that,
for some $\epsy>0$ and $n_0\ge 1$ the \textit{minorization condition} holds:
\[
P^{n_0}(x,A)  \ge \epsy \nu(A) \,, \qquad  \text{for $A\in\bx$ and   $x\in C$}.
\]
If the particular $\nu$ is not of interest, then we simply say that the set $C$ is small.  
It is sometimes convenient to   replace the set $C$ by a function $s\colon\state\to\Re_+$.  Then we say that $s$ is small if,
\begin{equation}
P^{n_0}(x,A)  \ge  \epsy s(x)  \nu(A) \,, \qquad  \text{for $A\in\bx$ and all $x\in\state$}.
\label{e:small}
\end{equation}
In analysis it is often more convenient to work with the 
resolvent $R:= \sum_{n=0}^\infty 2^{-n-1} P^n$, which satisfies a one-step minorization condition under~\eqref{e:small}: 
\begin{equation}
R(x,A)  \ge    s_+(x)  \nu(A)      \,, \quad  \text{for $A\in\bx$ and all $x\in\state$,}
\label{e:smallR1}
\end{equation}
 where $s_+(x) = 2^{-n_0-1} \epsy \int R(x,dy) s(y)$.

The Markov chain is called \textit{aperiodic} if there exists a probability 
measure $\nu$ and $\epsy>0$ such that, for each $x\in\state$  and $A\in\bx$,  
there is $n(x,A)\ge 1$ such that:
\[
P^{n}(x,A)  \ge \epsy \nu(A) \,, \qquad   \text{for all $n\ge n(x,A)$.}
\]
\notes{We had a nasty error here---previous inequality implied uniform ergodicity}

For any measurable function $g:\state\to [1,\infty)$, the Banach space $L_\infty^g$ is defined to be the set of measurable functions $\phi\colon\state\to\Re$ satisfying:
\[
\| \phi\|_g  \eqdef\sup_x \frac{1}{g(x)}  |\phi(x)| < \infty.
\]
For a pair of functions $g,h:\state\to [1,\infty)$,  and  any
linear operator $\haP\colon L_\infty^{g}  \to  L_\infty^{h}$,
 the induced operator norm is denoted:
\begin{equation}
	\lll \haP \lll_{g, h}
	\eqdef   \sup \Bigl\{ \frac{\|\haP \phi\|_{h}}{\|\phi\|_{g}} 
	:  \phi\in L^{g}_\infty,\ \|\phi\|_{g}\neq 0\Bigr\}.
	\label{e:Vnorm}
\end{equation}
We write $	\lll \haP \lll_{g}$ instead of $	\lll \haP \lll_{g, h}$ when $g=h$. In particular, we view of transition kernels
$Q$ as linear operators acting on functions $g$
on $\state$ via $Qg(x)=\int g(y)P(x,dy)$,
$x\in\state$.

Suppose that $\pi(g):=\int g d\pi  <\infty $.   Then the rank-one operator $\mathbbm{1} \otimes \pi$ has finite induced operator norm,  
$\lll \mathbbm{1} \otimes \pi \lll_{g, h} <\infty$, for any choice of $h$,  where,
\[
[ \mathbbm{1} \otimes \pi] \phi \, (x)    = \pi(\phi)  \,, \qquad \text{for   $x\in\state$ and $\phi\in  L_\infty^{g} $}.
\] 
Consider the centered semigroup $\{\tilP^n = 
P^n -\mathbbm{1} \otimes \pi  : n\ge 0\}$.
The chain $\bfPhi$ is called {\em $g$-uniformly ergodic} if $\pi(g)$ is finite, 
and  $\lll \tilP^n \lll_g\to 0$ as $n\to\infty$.  Necessarily the rate of convergence is geometric, and it is known that seemingly weaker definitions of geometric ergodicity are equivalent to the existence of $g$ for which 
$g$-uniform ergodicity holds~\cite{MT}.
We say $\bfPhi$ is {\em geometrically ergodic} if it is $g$-uniformly
ergodic for some $g$.

The {\em Lyapunov drift criterion}~(V4) 
{\em holds with respect
to the Lyapunov function $v:\state\to[1,\infty)$}, if:
$$
\left. 
\mbox{\parbox{.85\hsize}{\raggedright
For a small function $s$,   and constants $\Vdrift >0$, $b <\infty$,
\[
	\Expect[v(\Phi_{k+1}) - v(\Phi_k) \mid \Phi_k =x] \leq -\Vdrift v(x) + bs(x)\,, \qquad x\in\state .
\]
}}
\right\}
\eqno{\hbox{\bf (V4)}}
$$
Under~(V4) and aperiodicity, the Markov chain is $v$-uniformly 
ergodic~\cite[Ch.~15.2.2]{MT}.

We establish almost sure (a.s.) boundedness and convergence of $\{\theta_n\}$
in~\eqref{e:SAa} subject to~(V4) for the family of   transition kernels $\{P_\theta : \theta\in\Re^d \}$,
and suitable conditions on the step-size sequence and on $\barf$.  
A stronger drift condition is imposed to establish moment bounds:
$$
\left. 
\mbox{\parbox{.85\hsize}{\raggedright
For functions $V\colon\state\to\Re_+$,  $ W\colon\state\to [1, \infty)$, 
a small function $s\colon\state\to [0,1]$, and $b>0$:\,
\[
	\Expect\bigl[  \exp\bigr(  V(\Phi_{k+1})      \bigr) \mid \Phi_k=x \bigr]  
			\le  \exp\bigr(  V(x)  - W(x) +  b s(x)  \bigl) \,, \qquad x\in\state.\]
}}
\right\}
\eqno{\hbox{\bf (DV3)}}
$$
The bound~(DV3) implies~(V4) with $v=e^V$;
see~\cite{konmey05a}  or~\cite[Ch.~20.1]{MT}.

\begin{proposition}
\label[proposition]{t:v-uni-w}
\begin{subequations}
Suppose that the Markov chain $\bfPhi$ is aperiodic.  Then:
\begin{romannum}
\item
 Under~{\em (V4)}, there exists $\varrho \in (0, 1)$ and $b_v<\infty$ such that for each $g \in \Lv$,  
\begin{equation}
		\bigl| \Expect[g(\Phi_n) \mid \Phi_0= x] - \pi(g) \bigr| \leq b_v \|g\|_v v(x) \varrho^n  \,, \qquad n \geq 0\,, \ x\in \state  \, .
\label{e:v-uni-w}
\end{equation}
\item
	If~{\em (DV3)} holds   then the conclusion in~{\em (i)} holds with $v=e^V$.    Consequently,  for each $g\in L_\infty^{V+1}$ and $\kappa>0$, we can find $b_{\kappa, g}<\infty$ such that
\begin{equation}
	\begin{aligned}
			 \Expect[ |g(\Phi_n)|^\kappa \mid \Phi_0= x]   \leq b_{\kappa,g} v(x)    \,, \qquad n \geq 0\,, \ x\in \state  \, .
\end{aligned}
\label{e:v-uni-w-alpha}
\end{equation} 
\end{romannum}
\end{subequations}
\end{proposition}

\begin{proof} 
The geometric convergence~\eqref{e:v-uni-w} follows from~\cite[Th.~16.1.4]{MT},
and~\eqref{e:v-uni-w-alpha} follows from $v$-uniform ergodicity and the fact 
that $\| |g|^\kappa\|_v<\infty$ under 
our assumptions~\cite{konmey03a,konmey05a}.
\end{proof}


\subsection{Disturbance decomposition and ODE solidarity}
\label{s:ode-method}
 
The decomposition  of  {M\'etivier} and {Priouret}~\cite{metpri87} provides a representation for $\{ \Delta_{n+1} \eqdef \tilf(\theta_n,\Phi_{n+1}) :  n\ge 0\}$,   where 
 $\tilf(\theta,x) := f(\theta,x) - \barf(\theta)$ for $\theta\in\Re^d$ and $x\in\state$,   with $f$ defined in~\eqref{e:SAa} and $\barf$ defined in~\eqref{e:barf}.
Let $\haf:\Re^d\times \state \to \Re^d$ solve \emph{Poisson}'s equation:
\begin{equation}
	\label{e:fish-f}
	\Expect[\haf(\theta, \Phi^\theta_{n+1}) - \haf(\theta, \Phi^\theta_{n})\mid \Phi^\theta_n = x] = - \tilf(\theta, x)  \,, \qquad  \theta \in \Re^d\,,  \ x \in \state .
\end{equation}

 Provided a measurable solution exists, we obtain a disturbance decomposition based on the three sequences:
\begin{equation}
 \begin{aligned}
	\MD_{n+1}&  \eqdef  \haf(\theta_n, \Phi_{n+1})  - \Expect[  \haf(\theta_n, \Phi_{n+1})   \mid \clF_{n}]  \,,  
\quad \clT_{n+1} \eqdef   \uppsi(\theta_{n+1}, \Phi_{n+1})     \,,
\\	  
	 \Oops_{n+1} & \eqdef  \frac{1}{\alpha_{n+1}}  \big[  \uppsi(\theta_{n+1}, \Phi_{n+1}) - \uppsi(\theta_{n}, \Phi_{n+1})    \big] \,, 
	 \quad  
\textit{with} \ \     \uppsi(\theta,x) \eqdef  f(\theta,x) - 
\haf(\theta,x)
\end{aligned}
\label{e:noise-decomp}
\end{equation}
 Recall the sequence $\{\SAtime_K\}$ defined in~{\em (\ref{eq:tau})}.

\begin{lemma}
\label[lemma]{t:noise-decomp}
We have,
$\Delta_{n} 
= \MD_{n} - \clT_{n} + \clT_{n-1} - \alpha_{n} \Oops_{n}$,  and,
\begin{equation}
\begin{aligned}
	&	\sum_{i = n+1}^{K} \alpha_{i} \Delta_i
		 =  \Mart_{\SAtime_{K}} -\Mart_{\SAtime_{n}} +   \clE_{n,K},  \quad   0\le n <K \,,  
		\\
		\text{with} \quad
		\Mart_{\SAtime_{K}}  & \eqdef \sum_{k=1}^K  \alpha_k \MD_k\,,    \qquad 1 \le K <\infty \,,  \ \   \Mart_{\SAtime_0} =0   \,,
		\\
		 \clE_{n,K}  & \eqdef
-  \sum_{i=n+1}^{K} \alpha_{i}\Oops_{i}      + \alpha_{n+1}\clT_n	- \alpha_{K+1}\clT_{K} 
-  \sum_{i=n+1}^{K}[    \alpha_{i} - \alpha_{i+1}] \clT_{i}.
\end{aligned}
	\label{e:sum-by-parts}
\end{equation}	 
Moreover, $\{\MD_k\}$ is a martingale difference sequence and $\{\Mart_{\SAtime_K}\}$ is a martingale.    
\end{lemma}

\begin{proof}
The representation for $\{ \Delta_{n+1}\}$ follows from~\eqref{e:fish-f},  and~\eqref{e:sum-by-parts} is then obtained through summation by parts.   
\end{proof}

The disturbance decomposition in~\cite{metpri87} subsequent papers is
applied in the simpler setting in which $\bfPhi$ is Markovian~\cite{benmetpri12,kusyin97,bor20a}.   
Obtaining useful bounds on $\{ \Oops_{n} \} $ is more challenging in
our present, more general setting.
The value of~\Cref{t:noise-decomp} is clear only after we establish that $\{ \clE_{n,K} \}$ has a very small 
contribution to estimation error as compared to the martingale  
$\{\Mart_{\SAtime_K}\}$;  
much of the work in the Appendix is devoted to quantifying this claim.
 
The indexing of the martingale   $\{\Mart_{\SAtime_K}\}$ by  $\SAtime_K$ (defined in~\eqref{eq:tau}) is to facilitate \textit{ODE approximations}.     These are   defined by comparison of 
two continuous-time processes: $\bfODEstate=\{\Theta_t\}$, defined in~\eqref{e:ODEstate},  and 
the solution $\bfodestate^{(n)}=\{\odestate^{(n)}_t\}$
to the re-started ODE~\eqref{e:ODEn}.   The two processes are compared over 
time blocks of length approximately $T$,  with $T>0$ chosen sufficiently large.  

Let  $T_0=0$ and $T_{n+1} = \min \{\SAtime_k: \SAtime_k \geq T_n + T\}$. 
 The $n$-th time block is $[T_n, T_{n+1}]$ and $\displaystyle T\leq T_{n+1} - T_{n} \leq T+ \bar{\alpha} $ by construction, where   $\bar{\alpha} \eqdef \sup_k \alpha_k  $ is assumed bounded by~1 in Assumption~(A1) below. The following notation is required throughout: 
 \begin{equation}
\text{$m_0\eqdef 0$,
	 and $m_n$ denotes the integer satisfying $\SAtime_{m_n} = T_n$,\ for each $n\geq1 $. }
\label{e:mofn}
\end{equation}
\Cref{t:noise-decomp} is one step in establishing the  ODE approximation,
\begin{equation}
\lim_{n\to\infty} \sup_{T_n\le t\le T_{n+1} }  \|  \ODEstate_t   -  \odestate^{(n)}_t \| =0 \,,\qquad \mbox{a.s.}
\label{e:ode-method}
\end{equation}
This, combined with global asymptotic stability of the mean flow~\eqref{e:ODESA1}, quickly leads to convergence of $\{\theta_n\}$, provided this sequence is bounded.     Boundedness is established through examination of a scaled vector field defined next.


\subsection{ODE@\boldmath{$\infty$}}
\label{s:ode-method-infinity}

The ODE@$\infty$ technique for establishing stability of SA is based on the 
scaled vector field, denoted $\barf_{c}(\theta) \eqdef \barf({c}\theta)/{c},$ for any ${c} \geq 1$,   and requires the existence of a continuous limit:   $\barf_{\infty}(\theta)  \eqdef \lim_{c\to\infty} \barf_{c}(\theta) $ for any $\theta\in\Re^d$.  
The  ODE@$\infty$ is then denoted:
\begin{equation}
	\label{e:ode-infty}
\ddt \odestate_t = \barf_{\infty}(\odestate_t).
\end{equation}
We always have $\barf_{\infty}(0) = 0$, and asymptotic stability of this equilibrium is equivalent to global asymptotic stability~\cite{bormey00a,bor20a}.

\begin{subequations}
An ODE approximation is obtained for the scaled parameter estimates: 
On  denoting $c_n \eqdef \max\{1, \| \theta_{m_n}\|\}$, 
and  $ \hattheta_k  =  \theta_k /{c_n}$  for  each $n$ and $m_n \leq k < m_{n+1}$,  we compare two processes in continuous time on the interval $[T_n, T_{n+1})$:
\begin{align}
\hatODEstate_t  &=  \frac{1}{c_n}\ODEstate_t \,,  \quad T_n \le t  < T_{n+1},
	\label{e:rescaled-traj}
	\\
	\ddt  \hatodestate_t &= \barf_{c_n}( \hatodestate_t)   \,,   \quad  \text{with initial condition} \ \ \hatodestate_{T_n} = \hatODEstate_{T_n} = \tfrac{1}{ c_n}   \theta_{m_n }.
	\label{e:rescaled-ode}
\end{align} 
\end{subequations}

The scaled iterates $\{\hattheta_k\}$ satisfy,  for $m_n \leq k < m_{n+1} $,
\begin{equation}
	\label{e:hat-theta-sa}
\begin{aligned}
		\hattheta_{k+1} 
		&= \hattheta_k + \alpha_{k+1}f_{c_n}(\hattheta_k, \Phi_{k+1}) = \hattheta_k + \alpha_{k+1}\bigl[\barf_{c_n}(\hattheta_k) + \hatDelta_{k+1} \bigr] \,,
\end{aligned}
\end{equation}
where $f_{c_n}(\theta_k, \Phi_{k+1}) \eqdef f(c_n\theta_k, \Phi_{k+1})/c_n$ and $\hatDelta_{k+1} \eqdef \Delta_{k+1}/c_n$.  
The representation~\eqref{e:hat-theta-sa} motivates an ODE analysis similar to what is performed to establish~\eqref{e:ode-method}.  
The desired approximation is established in~\Cref{t:hat-ode-approx-as}:
\begin{equation}
\lim_{n\to\infty} \sup_{T_n\le t < T_{n+1} }  \|  \hatODEstate_t  - \hatodestate_t \| =0 \,,\qquad\mbox{a.s.}
\label{e:ode-method-infty}
\end{equation}
The short proof of~\Cref{t:bddpar} shows that~\eqref{e:ode-method-infty} implies a.s.\  boundedness of $\{ \theta_k \}$.

\section{Main Results}
 
Since we will be working with a family of Markov chains
$\{\bfPhi^\theta\}$, where each $\bfPhi^\theta$ has transition
kernel $P_\theta$ in some family $\{P_\theta:\theta\in\Re^d\}$,
we require a family of drift conditions.   
For simplicity we assume a common Lyapunov function, and   a common small set condition in the form~\eqref{e:smallR1}.  In particular,  
from now on, when we say that~(DV3) holds we mean the following:  
$$
\left. 
\mbox{\parbox{.85\hsize}{\raggedright
For functions $V\colon\state\to\Re_+$,  $ W\colon\state\to [1, \infty)$, 
  $s\colon\state\to [0,1]$, and $b>0$,
\[
	\Expect\bigl[  \exp\bigr(  V(\Phi^\theta_{k+1})      \bigr) \mid \Phi^\theta_k=x \bigr]  
			\le  \exp\bigr(  V(x)  - W(x) +  b s(x)  \bigl) \,,  
\]
for all $ x\in\state\,, \theta\in\Re^d$.   Moreover,   $s(x)>0$ for all $x$,   and 
there is a   probability measure $\nu$ on $\bx$ such that,
\begin{equation}
R_\theta (x,A)  \ge    s(x)  \nu(A) \,, \qquad  \text{for $A\in\bx$ and all $x\in\state$, $\theta\in\Re^d$.} 
\label{e:smallR}
\end{equation}%
}}
\right\}
\eqno{\hbox{\bf (DV3)}}
$$
The uniform version of~(V4) is analogous.

We require a Lipschitz continuity assumption for the family of transition kernels:     For a constant $b_d$,
and any $\theta,\theta' \in\Re^d$,    
\begin{equation}
\lll P_\theta  - P_{\theta'}  \lll_H  \le  \frac{b_d}{1+   \| \theta \| + \| \theta' \|  }   { \| \theta - \theta' \| },
\label{e:LipPtheta}
\end{equation}
where the choice of the function of $H\colon\state \to [1,\infty)$ 
depends on the context; recall~\eqref{e:Vnorm}
for notation.

In applications to RL, this Lipschitz condition  can be obtained by design.   In particular, in the stability analysis of Q-learning found in~\cite{mey24},  exploration is designed so that  $P_{r\theta} = P_\theta$  for $\|\theta\|\ge M$ and  all $r\ge 1$,   where $M\gg1$ is a design choice.  
Local Lipschitz continuity of $\{ P_\theta\}$  in $L_\infty^H$  then implies the global bound~\eqref{e:LipPtheta}.

Our main results  depend on Assumptions~(A1)--(A3) below.  Assumption~(A1) is standard in the SA literature,  and~(A3) is the starting point for analysis based on the ODE@$\infty$.  The special assumptions on $\bfPhi$ in~(A2) are less common;  especially the introduction of~(DV3).

\smallskip

\noindent
{\bf (A1)}
The non-negative gain sequence $\{\alpha_n\}$ satisfies:
\begin{equation}
	\text{$0 \leq \alpha_n \leq 1$, for all $n$,}
	\qquad
	\sum_{n=1}^{\infty} \alpha_n = \infty,
	\qquad \sum_{n=1}^{\infty} \alpha_n^2 < \infty.
		\label{e:stepsize-cond}
\end{equation}
	Moreover, with $\diffalpha_n \eqdef  \alpha_{n+1}^{-1} - \alpha_n^{-1}$,   $\diffalpha = \lim_{n \to \infty} \diffalpha_n$ exists and is finite.

\smallskip

\noindent
{\bf (A2)}
There exists a measurable function $L :\state \to \posRe$ such that for each $x\in\state$, 
\begin{equation}
	\label{e:lip-f}
\begin{aligned}
		\| f(0, x)\| &\leq L(x), \\
		\|f(\theta,x) - f(\theta', x)\| &\leq L(x)\|\theta - \theta'\| \,, \qquad \theta, \theta'\in \Re^d,
\end{aligned}
\end{equation}
with $L\colon\state\to\posRe$.   
The remaining parts of~(A2) depend on the drift condition imposed:

\whamc  Under~(V4), it is assumed that  $L^8\in \Lv$ and that~\eqref{e:LipPtheta}
holds 
with $H=v^\delta$, for any $1/4\le \delta\le 1$.

\whamc  Under~(DV3), it is assumed that $\delta_L \eqdef \|L\|_W < \infty$ (so that  
   $L\in L_\infty^W$)  and that~\eqref{e:LipPtheta}
holds 
with $H=1+V^p$ for $p=1,2$.   In addition,  for  each $r>0$,
    \begin{subequations}%
   \begin{align}
   \!   \!   \!   \!   \!
S_W(r)  &:= \{ x :  W(x)\le r \}  \quad    \text{is either small or empty for any $P_\theta$. }
\label{e:Wunbdd}
   \\
	b_V(r) & \eqdef \sup\{ V(x) :  x\in S_W(r) \}  <\infty.
\label{e:V-bounded-level}
\end{align}
\label{e:VWbdds}
\end{subequations}%

\noindent
{\bf (A3)} The scaled vector field $\barfinf(\theta)$ exists:
 $\barf_c(\theta) \to \barfinf(\theta)$ as $c \to \infty$, for each $\theta\in \Re^d$.  Moreover, the ODE~\eqref{e:ode-infty} is globally asymptotically 
stable.

\smallskip

Given $c \geq 1$, let $\phi_c(t,\theta)$ denote the solution to,
\begin{equation}
\ddt \odestate_t = \barf_c(\odestate_t)\,, \qquad  \odestate_0=\theta\,,
\label{e:scaledODEsemigroup}
\end{equation}
and define $\phi_\infty(t, \theta)$ accordingly,
so that it solves~\eqref{e:ode-infty}
 with initial condition $  \odestate_0=\theta$. 
It is known that~(A3) implies global exponential  
asymptotic stability of the ODE@$\infty$:
 
\begin{proposition}
\label[proposition]{t:contraction-ratio}
Under~{\em (A3)} there exists $ \Trelax >0$   such that $\| \phi_\infty(t, \theta)\| \leq \half \|\theta\|$ for   $t\geq \Trelax$ and every $\theta\in\Re^d$.   
Moreover:
\whamem{(i)}
There exist positive constants $b$ and $\delta$ such that,
for any $\theta\in\Re^d$ and $t\ge 0$,
\[
\| \phi_\infty(t, \theta)\|  \le b  \| \theta \|  e^{-\delta t}.
\]

\whamem{(ii)}
There exists $c_0>0$ and $\half < \Crelax < 1$ such that whenever $\|\theta\| \le 1$, 
\[
	\| \phi_c(t, \theta) \| \leq \Crelax \,, \qquad \text{for all } t\in [\Trelax, \Trelax + 1]\,, \  c\geq c_0.
\]
\end{proposition}

Part~(i) is~\cite[Lem.~2.6]{bormey00a}  and~(ii) 
is~\cite[Cor.~4.1]{bor20a}.  

\subsection{Convergence and moment bounds}  

\Cref{t:BigConvergence} is anticipated from well established   theory.

\begin{theorem}
\label[theorem]{t:BigConvergence}
  Suppose that the mean flow~\eqref{e:ODESA1} is globally asymptotically stable.    
If, in addition,~{\em (V4)} and~{\em (A1)--(A3)} hold, then the parameter sequence  $\{\theta_k\}$ converges a.s.\ to the invariant set of $\ddt \odestate_t = \barf(\odestate_t)$.
\end{theorem}

The main challenge in the proof  is establishing boundedness of the parameter sequence.   Once this is established, convergence of $\{\theta_k\}$ to the invariant set of the ODE follows from standard arguments~\cite[Th.~2.1]{bor20a},  subject to uniform bounds on the terms in~\Cref{t:noise-decomp}.   
The proof of boundedness is established in~\Cref{s:ASbdd}, based on the ODE approximation~\eqref{e:ode-method-infty}.  
The inequalities obtained in~\Cref{s:ASbdd}  will also be used to obtain moment bounds.
 
 The subscript `{\sf r}' in $\Trelax$ stands for ``relaxation time'' for the 
ODE@$\infty$.  
Moment bounds require either a bound on  $\Trelax$, or a 
stronger bound on $L$.    Recall
$\delta_L \eqdef \| L \|_W$ was defined in~(A2).
We say that $L = o(W)$ if,   \notes{Sean to Yiannis:  the colon after "if" doesn't work here!}
\[
\lim_{r\to\infty}  \sup_{x\in\state}  \frac{ |L(x)|}{\max\{r,W(x)\}} =0.  
\]
That is,  $\lim_{r\to\infty}  \| L \|_{W_r} =0$,  
with $W_r (x):= \max\{r,W(x)\}$ for   $x\in\state$. 
These are summarized in the following two-part assumption.
\wham{(A4a)}    $\Trelax$ in~\Cref{t:contraction-ratio}
can be chosen so that   $ \Trelax < 1/(4\delta_L)$
 \qquad 
\textbf{(A4b)} 
   $L = o(W)$.
   
\begin{theorem}
	\label[theorem]{t:BigBounds}
Suppose that~{\em (DV3)} and~{\em (A1)--(A3)}  hold.    Assume in addition 
that~{\em (A4)} holds in form~{\em (a)} or~{\em (b)}.    

Then,   
$\displaystyle 	\sup_{k \geq 0}\Expect\bigl[( \| \theta_k \| + 1)^4   \exp\bigl( V(\Phi_{k +1 }) \bigr) \bigr] < \infty $.
  \notes{Sean to Yiannis:  the colon after "if" doesn't work here!  Comma is better.   Max agrees :-)
  \\
  And I got rid of the displayed equation to make the point clear.}
\end{theorem}

The bound  $ \Trelax < 1/(4\delta_L)$ assumed in Assumption~(A4a) appears difficult to validate.  It is included because it is what is essential in the proof 	of~\Cref{t:BigBounds}.  
Assumption~(A4b) implies~(A4a) provided we modify the function $W$ used 
in~(DV3).

\begin{lemma}
\label[lemma]{t:A4ab}
Suppose that~{\em (DV3)} holds, subject to the bounds~\eqref{e:VWbdds}.  Then:
\whamem{(i)}
	{\em (DV3)} holds and the bounds~\eqref{e:VWbdds} continue to hold for the pair $(V,W_r)$ for any $r\ge 1$,   with $W_r (x):= \max\{r,W(x)\}$.

\whamem{(ii)}
	If in addition~{\em (A4b)} holds, then~{\em (A4a)} also 
holds with $W$ replaced by $W_{r_0}$ for   $r_0\ge 1$ sufficiently large.
 \end{lemma}
 
 The proof of~\Cref{t:A4ab} follows from the definitions.

\notes{New material here to reinforce that we are bringing new ideas}
The bound~(DV3) is used  in multiplicative ergodic theory for Markov 
chains~\cite{konmey03a,konmey05a}.  Motivation in the present paper is similar, and made clear from consideration of the linear scalar recursion, with $f(\theta,\Phi) = a(\Phi)\theta$,  so that  for $m,n\ge 0$,
\begin{equation}
| \theta_{n+m} |  =\Bigl| \theta_m \prod_{k=m+1}^{n+m}  (1+\alpha_k a(\Phi_k))  \Bigr|   \le   | \theta_m| \exp\Bigl( \sum_{k=m+1}^{n+m}  \alpha_k a(\Phi_k)  \Bigr).
\label{e:SAscalarBdd}
\end{equation}
This suggests that
strong assumptions are required to obtain bounds on $L_p$ moments of the parameter sequence even in this simplest of examples.
The following is a crucial step in the proof of~\Cref{t:BigBounds}.  
Its proof   is contained in~\Cref{s:MC}.  
 
\begin{proposition}
	\label[proposition]{t:DV3multBdd}
The following holds under~{\em (DV3)}:    
For  any initial conditions $\theta_0,\Phi_0$, 	
  any $n_0, n$, 
		and any non-negative sequence $\{\delta_k : 1\le k\le n-1 \}$ satisfying $\sum \delta_k\le 1$,  
\begin{equation} 
	\Expect\Bigl[\exp\Bigl( V(\Phi_{n_0+n}) +\sum_{k=n_0}^{n_0+n-1}   
	\delta_k W(\Phi_k)  \Bigr)  \Big| \clF_{n_0}  \Bigr]
		\le  b_v^2  e^{b} v(\Phi_{n_0})  \quad a.s. ,
		\label{e:DV3multBdd+}
\end{equation}  
where $b_v>0$ is as
in~\Cref{t:v-uni-w}, and $b>0$ is the constant in~{\em (DV3)}.
\end{proposition}

\subsection{Asymptotic statistics}

We now turn to rates of convergence, and for this it is assumed that the ODE~\eqref{e:ODESA1} is globally asymptotically stable with unique equilibrium denoted  $\theta^* \in\Re^d$.
The accent `tilde'  is used to represent error: We write $\tilODEstate_{\SAtime_k}^{(n)}   \eqdef \ODEstate_{\SAtime_k}  - \odestate^{(n)}_{\SAtime_k}$ and   $\tiltheta_k \eqdef \theta_k - \theta^*$.     
Two normalized error sequences are considered:
\begin{equation}
 z_k  \eqdef \frac{1}{\sqrt{\alpha_k}}  \tiltheta_k     \,,
  \qquad
\scerror{\SAtime_k}{n}  \eqdef\frac{1}{\sqrt{\alpha_k}}  \tilODEstate_{\SAtime_k}^{(n)}  \,,\qquad k\ge n\, .
\label{e:scerror} 
\end{equation}%
The domain of the latter is extended to  all $\SAtime\ge   \SAtime_n  $  by piecewise linear interpolation.  
 It is convenient to move the origin of the time axis as follows:   
 \begin{equation} 
\scerrorpull{\SAtime}{n} \eqdef \scerror{\SAtime_n +\SAtime}{n}     \,, \qquad \SAtime \ge  0.
\label{e:scerror_norm}
\end{equation}

The two fundamental approximations of interest here are:

\smallskip

\noindent
\textbf{Central Limit Theorem} (CLT):  For any bounded continuous 
function $g\colon\Re^d\to\Re$,
\begin{equation}
\lim_{k\to\infty} \Expect[ g( z_k ) ]  =\Expect[g(\sclim)],
\qquad\mbox{where}\; X\sim N(0,\SigmaTheta),
\label{e:CLT}
\end{equation}
for an appropriate covariance matrix $\SigmaTheta$.

\noindent
\textbf{Functional Central Limit Theorem} (FCLT):  The sequence 
of stochastic processes $\{\scerrorpull{\varble}{n} : n\ge 1 \}$ converges in distribution to the solution of the Ornstein-Uhlenbeck equation,
\begin{equation}
 d\sclim_{t} = F \sclim_t \,dt+ D\,dB_{t} \,, \qquad \sclim_0 =0\, ,
 \label{e:FCLT}
\end{equation}  where 
  $\bfmB=\{B_t\}$ a standard $d$-dimensional Brownian motion,  and $F,D$  
are  $d\times d$ matrices.
  
\smallskip
  

Consider the family of martingale difference sequences, 
parametrized by $\theta$,
\begin{equation}
\MD_n (\theta) \eqdef \haf(\theta, \Phi_n^\theta) - \Expect[\haf(\theta, \Phi_n^\theta)  \mid \clF_{n-1}],
\quad \text{  $\bfPhi^\theta$ is Markovian with tr.\ kernel $P_\theta$. }
\label{e:MD}
\end{equation} 
The steady-state  covariance matrices are denoted:
\begin{equation}
	\label{e:def-Q}
\begin{aligned}
\Sigma_{\MD}(\theta)  &\eqdef \Expect_{\pi_\theta} [ \MD_n(\theta)   \MD_n(\theta) ^\transpose]  \,,
	\quad &  \Sigma_{\MD}^*& =  \Sigma_{\MD}(\theta^*).
\end{aligned} 
\end{equation}
The CLT is obtained under the assumption that 
$A^* $ is Hurwitz, where $A^* \eqdef 
A(\theta^*)$ with $A=A(\theta)$ given in~(A5a) below. 
This assumption combined with
global asymptotic stability of the mean flow 
 implies that  it is globally \textit{exponentially} asymptotically stable~\cite[Prop.~A.11]{laumey24b}:
for some     $\bexp< \infty$, and $\rhoexp>0$,
\begin{equation}
\| \phi(t;\theta)  -\theta^* \|   \le \bexp\| \theta -\theta^* \| e^{ -\rhoexp t  } \,,  \qquad    \theta\in\Re^d\,, \   t\ge 0  
\label{e:EAS}
\end{equation}
This bound is required in consideration of bias, starting with the approximation $\Expect[\theta_k -\theta^*] \approx 
\phi(\SAtime_{k-n};  \theta_n) - \theta^*$    for large $n$ and all $k\ge n$. For the CLT to hold, the bias must decay faster than $1/\sqrt{\alpha_k}$.  Assumption~(A5b) is introduced to ensure this.

\smallskip

\noindent
{\bf (A5a)}
$\barf:\Re^d \to \Re^d$ is continuously differentiable in $\theta$, and the Jacobian matrix $A = \partial \barf$ is  uniformly bounded and uniformly Lipschitz continuous.

\smallskip

\noindent
{\bf (A5b)}   The step-size is $\alpha_n = 1/n^\rho$ with $\half <\rho\le 1$,   and~\eqref{e:EAS} holds
with   $\rhoexp>0$.  
It is furthermore assumed that  $\rhoexp>1/2$ in the special case   $\rho=1$.
 
\begin{theorem}[FCLT]
\label[theorem]{t:FCLT}
Suppose~{\em (DV3)},~{\em (A1)--(A4)} and~{\em (A5a)} hold.   
Assume moreover that the mean flow~\eqref{e:ODESA1} is globally asymptotically stable, and that
  $ \tfrac{\diffalpha}{2}I + A^* $ is Hurwitz, where $A^* \eqdef 
	A(\theta^*)$ with $A=A(\theta)$ given in~{\em (A5a)}
	and $\diffalpha$   defined in~{\em (A1)}.    \notes{The previous statement was scrambled!   We made no reference to SA, only to the SDE~\eqref{e:FCLT}!!}
Then, the FCLT holds:  $\{ \scerrorpull{\SAtime}{n} :  \SAtime\ge 0 \}$ converges in distribution to the solution of~\eqref{e:FCLT}, 
with $F = \tfrac{\diffalpha}{2}I + A^*$,   $D$ any solution to $ \Sigma_{\MD}^* = D D^\transpose$,  and $\SigmaTheta>0$  is the unique solution to the Lyapunov equation:
\begin{equation}
		[\tfrac{1}{2} \diffalpha I+A^*] \SigmaTheta + \SigmaTheta [\tfrac{1}{2} \diffalpha I+A^*]^\transpose +  \Sigma_{\MD}^* = 0.
\label{e:Sigmatheta}
\end{equation} 
\end{theorem}

We write $h =o( \upomega)$ if $\displaystyle \lim_{r\to\infty}  \sup_{\theta \in \Re^d}  \frac{ |h(\theta )|}
{\max\{r,\upomega (\theta)\}} =0  $,  for functions $h , \upomega:\Re\to\Re_+$.
We can now establish a strong version of the CLT which, in addition
to weak convergence, includes convergence of the 
normalized covariance.

\begin{theorem}[CLT]
\label[theorem]{t:CLT}
Suppose that~{\em (A5b)} holds, along with the assumptions of~\Cref{t:FCLT}.  
Then the CLT~\eqref{e:CLT} holds,  with asymptotic covariance given in~\eqref{e:Sigmatheta}.
Moreover, the set of functions for which~\eqref{e:CLT} holds  includes any 
continuous function $g\colon\Re^d\to\Re$ satisfying $\|g\|=o(\upomega)$ with $\upomega(\theta) = \|\theta\|^4$ for $\theta\in\Re^d$. 
 In particular, the following limit holds:
\[
		\lim_{n \to \infty} \frac{1}{\alpha_n} \Expect\bigl[ \tiltheta_n\tiltheta_n^\transpose \bigr] = \SigmaTheta.
\]
\end{theorem}

\Cref{t:CLT}  suggests that the step-size $\alpha_{n+1} = \alpha_0/(n+1)$ 
is required 
to obtain the optimal convergence rate of $O(1/n)$ for the mean-square error  
(subject to the constraint that $ \tfrac{\alpha_0}{2}I + A^* $ is Hurwitz).   
In fact, it is far better to use a larger step-size sequence and employ 
averaging via~\eqref{e:thetaPR}.
This is made clear in~\Cref{t:PR}.

One representation of the covariance matrix $\SigmaPR$ appearing in~\eqref{e:PR}  is based on the stationary version of the Markov chain
 with transition kernel $P_{\theta^*}$, denoted $\{  \Phi^*_{k}  : k \in\intgr\}$.    Let 
$\Sigma_{\Delta}^* $ denote  the asymptotic covariance of $\{\Delta_k^* \eqdef f(\theta^*, \Phi^*_{k})  : k \in\intgr\}$,   and $G \eqdef - ( A^* )^{-1}$
the \textit{stochastic Newton-Raphson gain} of Ruppert~\cite{rup85}. 
\Cref{t:PR} that follows justifies the representation,
 \begin{equation}
\SigmaPR =   G  \Sigma_{\Delta}^* G^\transpose  .
\label{e:PRrep}
\end{equation}
 
The standard definition of asymptotic covariance gives:
 \begin{equation}
\begin{aligned}
\Sigma_{\Delta}^* = \sum_{k=-\infty}^\infty \Expect [  f(\theta^*, \Phi^*_0)f(\theta^*, \Phi^*_{k})^\transpose]
=  \Sigma_{\MD}^* .
\end{aligned}
\label{e:SigmaMD}
\end{equation}
We also have $\Sigma_{\Delta}^* = \Sigma_{\MD}^*$  (defined in~\eqref{e:def-Q})---a consequence of the   noise decomposition~\eqref{e:noise-decomp}.
 
It is well known that $\SigmaPR$ is the optimal achievable covariance for SA  when $\{ \Delta_{n+1} \}$  is a martingale difference sequence, and this is achievable using the averaging technique of 
Polyak and Ruppert~\eqref{e:thetaPR} (recall the history at 
outlined at the close of~\Cref{s:intro}).  The following extends this 
result to the far more general setting of this paper.  
 
\notes{Might need to explain how this is optimal.  Not Cramer Rao, as some suggest.
\\
June 24, 2023:  I dropped the minus sign on $G$ in the submission.
Also, we should write $\Sigma_{\Delta}^* = \sum_{k=-\infty}^\infty \Expect_{\pi_{\theta^*}} [ \Delta_k^*  \{\Delta_k^*\} ^\transpose]=  \Sigma_MD^*
$
}
 
\begin{theorem}[Optimizing asymptotic covariance]
\label[theorem]{t:PR}
Suppose the assumptions of~\Cref{t:CLT} hold.   Then,~\eqref{e:PR}   holds for the estimates~\eqref{e:thetaPR} with 
$
\SigmaPR  
$ given in~\eqref{e:PRrep}.
\end{theorem}

\Cref{t:PR} improves upon~\Cref{t:CLT} in two respects: The mean square error convergence is accelerated to $O(1/n)$ rather than $O(\alpha_n)$,  
  and the asymptotic covariance is optimal -- that is, minimal.   Its proof 
   follows from~\Cref{t:PRbig}, in which the coupling bound~\eqref{e:PRbig}  is established.

\notes{misplaced here:
 The constant $ \bdd{t:PRbig}  $ vanishes if   $f(\theta,\Phi)$ is linear as a function of $\theta$, and also 
 $\{ \Oops_n\}$ defined in~\eqref{e:noise-decomp} is identically zero.  
 }
 
 
\subsection{Counterexample}
\label{s:counterexample}

Consider the M/M/1 queue with uniformization,  that is,
a reflected random 
walk  $\bfmQ =\{Q_n\}$ on $\state=\{0,1,2,\dots\}$,
\begin{equation}
Q_{n+1} = \max(0, Q_n + D_{n+1})  \,, \qquad n\ge 0\,,
\label{e:RRW}
\end{equation}
in which $\bfmD$ is i.i.d\ on $\{-1,1\}$ with $\alpha = \Prob\{D_k=1\}$,   $\mu = 1-\alpha = \Prob\{D_k= - 1\}$.
This is a reversible, geometrically ergodic Markov chain 
when $\rho=\alpha/\mu<1$, 
with geometric invariant measure $\pi$ and with $\displaystyle \eta \eqdef \Expect_{\pi}[Q_n] = \rho/(1-\rho)$. 

However, the drift inequality~(DV3) fails to hold for any unbounded 
function $W$~\cite{konmey05a}.   \notes{We should be more specific, but I'm feeling lazy.}

Consider the scalar stochastic approximation recursion,  
\begin{equation}
	\theta_{n+1} = \theta_n + \tfrac{1}{n+1} \{ (Q_{n+1} - \eta-1)\theta_n + \CEdist_{n+1} \},  \quad \theta_0\in\Re\,,
	\label{e:sa-mm1}
\end{equation}
where $\bfmW=\{ \CEdist_{n} \}$ is i.i.d.\ and independent of $\bfmQ$, with Gaussian marginal $N(0,1)$. 
The mean vector field associated with~\eqref{e:sa-mm1} is linear,
\[
\barf(\theta) \eqdef \Expect_\pi\bigl[(Q_{n+1} - \eta -1)\theta + \CEdist_{n+1} \bigr] = -\theta,
\]
so that~\Cref{t:BigConvergence} implies convergence.   Moreover, this example satisfies all of the assumptions imposed in~\cite{ger92},  so that $L_p$ bounds hold   subject to the resetting step imposed there.

To obtain moment bounds for~\eqref{e:sa-mm1}
without resetting,  it is helpful to compare solutions:
For two initial conditions $\theta_0^1, \theta_0^2$ we obtain a bound similar 
to~\eqref{e:SAscalarBdd}:     \notes{New equation to help reader see why this is hard}
\[
| \theta_{n}^1 - \theta_{n}^2 |  \le   | \theta_0^1-\theta_0^2| \exp\Bigl( \sum_{k=1}^{n}  \alpha_k [Q_{k} - \eta-1]  \Bigr) \,. 
\]
We find that the right-hand side cannot admit useful bounds for all $\rho<1$.

\begin{proposition}
\label[proposition]{t:unbounded-mom}
The following hold for the  SA recursion~\eqref{e:sa-mm1}:
\whamem{(i)} 	For all $\rho   < 1$, we have
$\displaystyle
		\lim_{n\to \infty}  \theta_n =\theta^* = 0$ a.s.
		\quad
\textbf{\emph{(ii)}}  But if $\rho > 1/2$, then
$\displaystyle
		\lim_{n\to \infty} \Expect[\theta_n^2] = \infty$.  
\end{proposition}

The proof of~\Cref{t:unbounded-mom} is rooted in large deviations theory for    reflected random walks.   A standard object of study is the scaled process  $q^n_t = \frac{1}{n} Q_{\lfloor n t \rfloor}$.   As $n\to\infty$ this converges to zero for each $t$ with probability one.   The large deviations question of interest is the probability of a large excursion over a finite time-window as illustrated in~\Cref{f:tent-path}.  There is elegant theory for estimating this probability, e.g.,~\cite{ganoco02},  which leads to the proof of~\Cref{t:unbounded-mom}.

\begin{figure}[ht]

\centering
\includegraphics[width= 0.95\hsize]{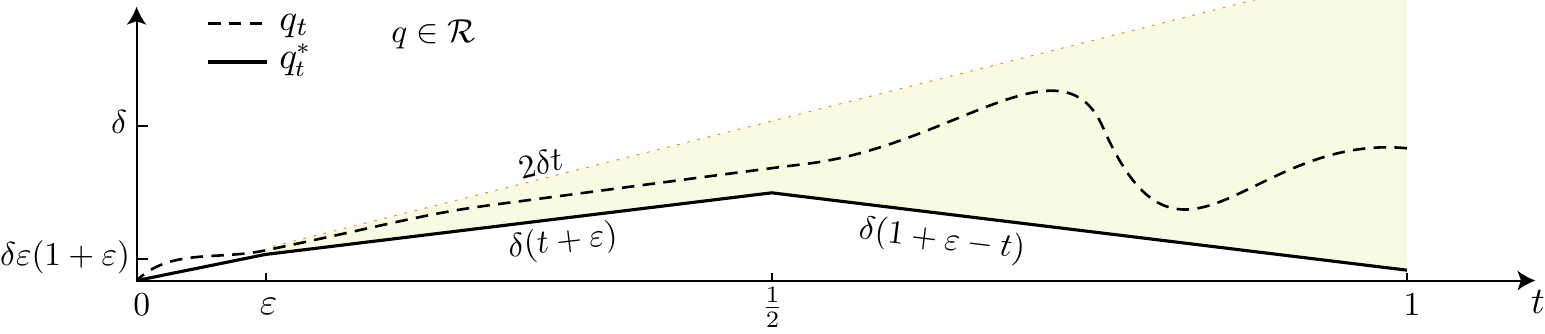}
 	\caption{Constraint region for scaled queue length process.}
 \label{f:tent-path}
\end{figure}

\Cref{t:unbounded-mom} does \textit{not} say that the CLT fails,
rather, it says the limiting covariance cannot be used to inform 
the convergence analysis.
Given the form~\eqref{e:sa-mm1} and the fact that $\{\theta_n\}$ converges to zero,  we might establish a CLT by comparison with the recursion,
\[
	\theta_{n+1}^\circ = \theta_n^\circ + \tfrac{1}{n+1}[\barf(\theta_n^\circ ) + \CEdist_{n+1}]
				= \theta_n^\circ +\tfrac{1}{n+1} [ - \theta_n^\circ + \CEdist_{n+1}],
\]
for which the CLT does hold with asymptotic variance equal to one.

All of assumptions for the CLT imposed in~\cite{pezheu97} are satisfied for this example, except for one:   Assumption~C.6  would imply the following:  For each $r\ge 1$ there exists  $b_{p,r} $ such that for each initial condition $\theta_0$ and $Q_0=q_0$,   \notes{This is new.  Is it understandable?}
\[
\Expect_{\theta_0,q_0} [\ind\{ \|  \theta_k \|  \le r  :  1\le k\le n\}  Q_n^p]   \le        b_{p,r} [1+ q_0^p].
\]
This is unlikely, since $\theta_k\to 0$ almost surely,   and  theory implies the weaker bound when the indicator function is removed:
$\displaystyle
\Expect_{q_0} [   Q_n^p]   \le        b_{p} [1+ q_0^{p+1}]$;
see~\cite[Prop.~14.4.1]{MT}.    \notes{  understandable?}

\Cref{f:CounterQ} shows histograms of the normalized error $ \sqrt{n} \theta_n$ for two values of load.   In each case the approximating density is $N(0,1)$.   What is missing from these plots is outliers that were removed before plotting the histograms.   For load  $\rho = \alpha/\mu= 3/7$ the outliers were few and not large.    For   $\rho= 6/7$,  nearly 1/3 of the samples were labeled as outliers,  and in this case the outliers were massive: Values exceeding $10^{20}$ were observed in  about 1/5 of runs.

   \begin{figure}[ht]
 \centering 
    \includegraphics[width=.7\hsize]{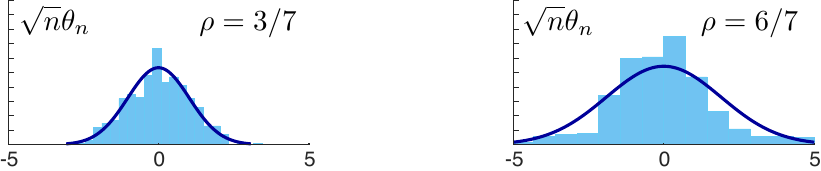}
  \caption{Histograms  of the normalized error for the scalar SA 
recursion~\eqref{e:sa-mm1},  with outliers discarded, and its Gaussian density approximation. 
The time horizon was $10^4$ for the smaller load, and $10^5$ for $\rho= 6/7$.}
\label{f:CounterQ} 
\end{figure}

 \section{Conclusions}

This paper provides simple criteria for convergence and asymptotic statistics 
for stochastic approximation and RL.   There are of course 
many important avenues for future research:
 
\wham{(i)}   We are currently investigating the value of resetting as defined
in~\cite{ger92}.   In particular, with resetting, can versions of~\Cref{t:CLT,t:PR}  be obtained under~(V4) and appropriate  assumptions on $f$?

\wham{(ii)} 
   Extensions to two time-scale SA are of interest in many domains.   
Such generalizations  will be made possible 
based on the present results,
combined with ideas from~\cite{karbha18,kalmounautadwai20a,faibor23}.

\wham{(iii)}  Is there any hope for finer asymptotic statistics,
such as an Edgeworth expansion or finite-$n$ bounds of
Berry-Ess\'{e}en type?

\wham{(iv)}  Many practitioners in RL opt for constant step-size algorithms.  When $\alpha_n\equiv \alpha_0$ is fixed but sufficiently small, it is possible to obtain the drift condition~\eqref{e:Lcontraction} under the assumptions of~\Cref{t:BigBounds} (with~(A1) abandoned).     In~\cite{laumey23a} it is shown that this implies that the fourth moment of $\tiltheta_n$ is uniformly bounded.     Open questions here include:
 

\whamc  
Can we   establish a topological formulation of geometric ergodicity for $\{ (\theta_n,\Phi_n) : n\ge 0\}$ to justify steady-state bias and variance formulae?  
Ideas in the prior work~\cite{huakonmey02a} might be combined with recent advances on topological ergodicity for Markov chains---see~\cite{devkonmey20,qinhob22} and their references.

\whamc
Once we have established some form of ergodicity, does it follow that $\{z_k\}$ converges in distribution as $n\to\infty$? 
   Is the limit approximately Gaussian for small $\alpha_0$?
Can we obtain convergence of moments as in~\Cref{t:CLT}?



  
\bibliographystyle{abbrv}  

\def\cprime{$'$}\def\cprime{$'$}

 
\newpage

\appendix
\section{Technical Proofs}
\label{s:app}

{The appendix begins with bounds on sums of functions of the Markov 
chains $\bfPhi^\theta$.    
\Cref{s:MC} is devoted to  bounds under~(V4) leading to a 
proof of~\Cref{t:DV3multBdd},   and finer bounds are applied in~\Cref{s:Z} to 
obtain bounds on solutions to Poisson's equation~\eqref{e:fish-f}.
The remainder of the appendix is organized in correspondence with the presentation of the main results:  
\Cref{s:ASbdd} establishes boundedness of parameter estimates under the 
assumptions of~\Cref{t:BigConvergence},  and hence convergence.  \Cref{s:Lp} contains parallel theory for establishing $L_p$ bounds on the parameter estimates
and, based on this,~\Cref{s:AS}
contains proofs of the main results of the paper 
concerning asymptotic statistics.  

\subsection{Markov chain bounds}
\label{s:MC}

This subsection is organized into two parts.  We first investigate implications of~(V4), and then establish a key lemma under a weaker drift condition.

\wham{Consequences of~(V4)}

Recall the constant $b_v$ was introduced in~\Cref{t:v-uni-w} in the ergodic theorem~\eqref{e:v-uni-w}, which is equivalently expressed  $\lll P^n -\One\otimes\pi\lll _v  \le b_v   \varrho^n$. 

In this Appendix we   consider the family of transition kernels  $\{P_\theta : \theta\in\Re^d \}$.    It follows from the main result of~\cite{meytwe94b} that  the following hold under~(V4) or~(DV3), for  constants $b_v$,  $   \varrho<1$ independent of $\theta$:
\begin{equation}
\lll \tilP_\theta^n   \lll_v   \le b_v    \varrho^n
\quad
\text{and}
\quad
\lll P_\theta^n\lll_v \le b_v    \,, \qquad n \geq 0,
\label{e:bv}
\end{equation}
with $v\eqdef e^V$ under~(DV3),   and  $\tilP_\theta^n   \eqdef P_\theta^n -\One\otimes\pi_\theta$.

We can generalize the second bound in~\eqref{e:bv} to the stochastic process $\bfPhi$:
\begin{lemma}
\label[lemma]{t:PthetaBdd}
Under~{\em (DV3)} we have
the bound $\Expect[ v( \Phi_n) \mid \clF_k]    \le  b_v  v(\Phi_k)$  a.s.,
for a possibly larger constant $b_v$,
for any initial conditions $\theta_0$,  $\Phi_0$,  and any $n>k\ge 0$.    
\end{lemma}

\begin{proof}  
Subject to~(DV3), denote  $S = \{ x   :   \exp (W(x)  -b s(x) ) \le  2\}$.  
Then, for any $\theta_0,\Phi_0$,
\[
\begin{aligned}
\Expect[ v( \Phi_{m+1}) \mid \clF_m]  & \le   \exp\big(  V( \Phi_{m})-W(\Phi_m) + bs (\Phi_m) \big) 
   \\
            & \le   \ind_{S^c} (\Phi_m)  \half   v( \Phi_{m})    
            +    \ind_{S} (\Phi_m)   \Big[  \sup_{x\in S}  \exp\big( V(x) -W(x) + bs (x) \big)  \Big]
               \\
            & \le     \half   v( \Phi_{m})    +   b_S  \,,   \qquad m\ge 0\,, 
\end{aligned} 
\]
where $b_S<\infty$ can be found due to~\eqref{e:V-bounded-level}.       The desired result then follows, with $b_v = 2 b_S$,  by induction and the smoothing property of conditional expectation.
\end{proof}

These uniform bounds lead to the proof of~\Cref{t:DV3multBdd}:

\begin{proof}[Proof of~\Cref{t:DV3multBdd}]   
The proof amounts to obtaining a bound  on the nonlinear program,  $
	\max\{ \Gamma(\delta) :   \delta\in \Re^n_+\,,  \   \sum \delta_k\le 1 \}$,
	with $\Gamma$ the function of $\delta$ on the left-hand side of  \cref{e:DV3multBdd+}.
To simplify discussion we take $n_0=0$, so that the objective becomes,
\[
	\Gamma(\delta) \eqdef	\Expect\Bigl[\exp\Bigl( V(\Phi_n) +\sum_{k=0}^{n-1}   \delta_k W(\Phi_k)  \Bigr)  \Bigr],
\]
with $\theta_0\in\Re^d$ and $\Phi_0=x\in\state$ arbitrary.

The function $\Gamma\colon\Re^n_+\to\Re_+ \cup \{\infty\}$ is strictly convex,  which means the maximum  is attained at an extreme point:  The optimizer $\delta^*$  is a vector with all entries but one equal to zero.  Consequently,  
$\displaystyle
	\Gamma(\delta^*) = \max_{0\le k < n}   \Expect\bigl[\exp\bigl( V(\Phi_n) +  W(\Phi_k)  \bigr)  \bigr]
$.

To complete the proof of~\eqref{e:DV3multBdd+} we obtain the bound,
\begin{equation}
		\Expect\bigl[\exp\bigl( V(\Phi_n) +  W(\Phi_k)  \bigr)  \bigr]
		\le  b_v^2  e^{b} v(x)    \,,\qquad 0\le k\le n-1\,,\ \  \Phi_0=x\in\state \, .
		\label{e:DV3multBdd+Extreme}
\end{equation}
On applying~\Cref{t:PthetaBdd} we have for each $k$,  
\[
\Expect\bigl[\exp\bigl( V(\Phi_n) +  W(\Phi_k)  \bigr) \mid  \clF_{k+1} \bigr] 
\le b_v v (\Phi_{k+1}) \exp\bigl(    W(\Phi_k)  \bigr) ,
\]
Hence by the smoothing property of conditional expectation,
\[
			\Expect\bigl[\exp\bigl( V(\Phi_n) +  W(\Phi_k)  \bigr)  \bigr]  
		  \le  b_v  \Expect\bigl[\exp\bigl( V(\Phi_{k+1}) +  W(\Phi_k)  \bigr)  \bigr]
		  \le  b_v  e^b  \Expect\bigl[\exp\bigl( V(\Phi_k)      \bigr)  \bigr] \,, 
\]
	where the second inequality follows from~(DV3) and the  uniform bound $s(x)\le 1$.  
This and a second application of~\eqref{e:bv}
	completes the proof of~\eqref{e:DV3multBdd+Extreme}.
\end{proof}

\wham{Condition (V3) and its consequences.}  

Condition~(V3) is a relaxation of~(V4):  For  functions $G, H:\state\to[1,\infty]$,   a small function   $s\colon\state\to ( 0,1]$,   and constant  $b <\infty$,
\begin{equation}
	\Expect[ H(\Phi^\theta_{k+1}) - H(\Phi^\theta_k) \mid \Phi^\theta_k =x] \leq -  G(x) + bs(x)\,, \qquad x\in\state .
\label{e:V3}
\end{equation}
Or in operator-theoretic notation, 
$P_\theta H \le H - G  + bs $.
 Consequences of this drift criterion is a focus of~\cite[Ch.~14]{MT}.
 Condition~(DV3) implies~(V3) with $V=H$ and $G=W$.

In later results, in particular~\Cref{t:fund-kernel-z},
we construct a solution to~\eqref{e:V3} in which the small function differs from the function $s$ appearing in~\eqref{e:smallR}.  The next result justifies replacing $s$ by an indicator function of a sublevel set of the function $H$.

 \begin{lemma}
\label[lemma]{t:uniSmall}
Suppose that~\eqref{e:V3} and~\eqref{e:smallR}
 hold.     
Then the set $C_m = \{ x : H(x) \le m \}  $ is either empty or small for each $m\ge 1$:      There exists $\epsy_m>0$ 
independent of $\theta$ 
such that   $ R_{\theta}(x,A)  \ge \epsy_m \nu(A) $   for each $A\in\bx$ and   $x\in C_m$.   
\end{lemma}

\begin{proof}
The fact that $C_m$ is small follows from~\cite[Th.~14.2.3]{MT}.   The proof of the lemma involves obtaining $\epsy_m$,  depending only on $\{H, b, s, \nu \}$.


It is assumed in~\eqref{e:smallR} that $s$ is everywhere positive.      Consequently, there is $\delta>0$ such that the set $S = \{ x : s(x)\ge \delta \} $  
 satisfies   $\nu(S)>0$.    We next obtain a bound on the \textit{first return time} to the set $S$,  denoted $\uptau_S = \min\{ k\ge 1 : \Phi^\theta_k \in S \}$.
We obtain from~\cite[Th.~14.2.2]{MT} and the assumption that $G\ge 1$,
 \[
 \Expect[\uptau_S]  \le V(x)  + b \Expect\Big[\sum_{k=0}^{\uptau_S -1}     s(\Phi^\theta_k)  \Bigr]
\le V(x)  + b \frac{1}{\nu(S)}  \Expect\Big[\sum_{k=0}^{\uptau_S -1}     R_{\theta}(\Phi^\theta_k,  S)  \Bigr]  \,, \quad \text{for all $x = \Phi^\theta_0$, 
}
 \]
 where the second inequality follows from~\eqref{e:smallR}.
 The following bound on the sum on the-right hand side  is justified in the proof of~\cite[Th.~14.2.3]{MT}:
\[
\Expect\Big[\sum_{k=0}^{\uptau_S -1}     R_{\theta}(\Phi^\theta_k,  S)  \Bigr]    =   \sum_{n=0}^\infty 2^{-n-1}  \Expect\Big[\sum_{k=0}^{\uptau_S -1}    
			  \ind\{ \Phi^\theta_{k+n}\in S \}    \Bigr]   
			\le \sum_{n=0}^\infty 2^{-n-1}  n    = 1.
\]
From the definition of $C_m $ we conclude that  $ \Expect[\uptau_S]  \le  B_m  \eqdef m +   b  / \nu(S) $  for all $x \in C_m$.  

Markov's inequality then gives $\Prob \{  \uptau_S  \le  n  \}  \ge 1  -  B_m / n$ for any $n$ and any $x\in C_m$.   
Hence with   $n_m\ge 1$ chosen so that  $B_m / n_m \le 1/2$, we have
 $\sum_{k=1}^{n_m}  P_\theta^k(x,S) =  \Prob \{  \uptau_S  \le  n_m  \}  \ge 1/2$ for $x\in C_m$.

We have $P_\theta^k R_{\theta}   \le  2^{-k-1} R_{\theta}$ for any $k$, giving for $x\in C_m$ and $A\in \bx$,   
\[
\Big[ \sum_{k=1}^{n_m}  2^{-k-1}  \Big]   R_{\theta}(x,A) \ge  \Big( \sum_{k=1}^{n_m} P_\theta^kR_{\theta}  \Bigr)    \,    (x,A)  \ge \half  \min_{y\in S} R_{\theta}(y,A)   \ge \half \delta \nu(A).
\] 
This gives the desired inequality with $\epsy_m = \half \delta\big[ \sum_{k=1}^{n_m}  2^{-k-1}  \big]^{-1}$. 
 \end{proof}

\subsection{Poisson's equation and consequences}
\label{s:Z}

This subsection is devoted to bounds on the solution to Poisson's equation $\haf$ appearing in~\eqref{e:fish-f},  which is the basis of the noise decomposition~\eqref{e:noise-decomp}.  

We first require a representation in terms of the \textit{fundamental kernel}~\cite{MT},    defined as the inverse $\clZ_\theta \eqdef  [I - \tilP_\theta   ]^{-1} $,
with $\tilP_\theta$ defined below~\eqref{e:bv}.  We will see that the inverse exists on an appropriate domain under the assumptions of our main results. 
Throughout the Appendix we consider exclusively the   solution to~\eqref{e:fish-f} defined by:
\begin{equation}
 \haf(\theta, x) \eqdef  \int \clZ_\theta(x,dy)  f(\theta,y)\,,\qquad  \theta\in\Re^d  \,,  \  x\in\state .
\label{e:hafZ}
\end{equation}
When adopting operator theoretic notation we denote $ \haf_\theta (x) = \haf(\theta, x)  
$   and $f_\theta(x) =f (\theta,x)$  for $x\in\state$,   giving $\haf_\theta = \clZ_\theta f_\theta$.

Subject to geometric ergodicity, the fundamental kernel may be expressed as the sum:
\begin{equation}
	\label{e:fundamental-z}
	\clZ_\theta  = \sum_{k=0}^\infty\tilP_\theta ^k = I + \sum_{k=1}^\infty [P_\theta ^k - \mathbbm{1}\otimes \pi_\theta ] .
\end{equation}
Under~(V3) we then have $\| \haf_\theta\|  \le c_\theta v(x)$ for all $x$, with $c_\theta<\infty$ for each $\theta$.
The following provides more useful bounds.

\begin{proposition}
\label[proposition]{t:V3bdds}
Suppose that~\eqref{e:V3} and~\eqref{e:smallR}
 hold.     
Then the fundamental kernels admit the uniform bound
   $\lll \clZ_\theta \lll_{G,H} \le   b_z \eqdef  b_u b_\nu  +  b_u [ 1 + b_u b_\nu]$, with
$  b_u \eqdef  1 + 2b     $  and   $  b_\nu \eqdef    b + \| s /H \|_\infty^{-1}  $.
\end{proposition}

\begin{proof}  
Under the assumptions of the proposition 
the Markov chain is positive Harris recurrent with $\pi_{\theta}(G)<\infty$  and  $\lll \clZ_\theta \lll_{G,H} < \infty$~\cite[Th.~2.3]{glymey96a}.  
The remaining work is to construct the universal constant $b_z$.  

The minorization condition is  expressed $R_{\theta}\ge s\otimes \nu$, which justifies the notation
$
U_{\theta} = \sum_{n=0}^\infty (R_{\theta} - s\otimes \nu)^n
$.
It is well known that $\mu_{\theta} = \nu U_{\theta}$ is an invariant measure for both $R_{\theta}$ and $P_\theta$,  that $\mu_{\theta}(s) =1$,  and that $U_{\theta} s = 1$ everywhere  (see~\cite{konmey03a,konmey05a} for history).   

Based on this theory we obtain a representation for the fundamental kernel as follows:   
First,  $ \clZ_{\theta,r} \eqdef U_{\theta}  [I-\One\otimes \pi_{\theta}]  $ is a version of the fundamental kernel for $R_{\theta}$,  so that $R_{\theta} \clZ_{\theta,r} = \clZ_{\theta,r}  -  [I-\One\otimes \pi_{\theta}]  $.    
The identity $ R_\theta[P_\theta-I] = R_\theta-I$
 then implies that $\clZ_{\theta,0} \eqdef R_{\theta} \clZ_{\theta,r}  $   solves $P \clZ_{\theta,0} = \clZ_{\theta,0}  -  [I-\One\otimes \pi_{\theta}]  $.   
We add a rank one operator to obtain:
\[
\clZ_{\theta} \eqdef \One\otimes \pi_{\theta} +	\clZ_{\theta,0} =	   \One\otimes \pi_{\theta} + U_{\theta} R_{\theta} [I-\One\otimes \pi_{\theta}].
\]
This solves $P \clZ_{\theta}= \clZ_{\theta} -  [I-\One\otimes \pi_{\theta}]  $, and also  $\pi_{\theta}(\clZ_{\theta}g) = \pi_{\theta}(g)$ for any $g \in L_\infty^G$,  which is consistent with~\eqref{e:fundamental-z}.

The value of $b_z$ will be obtained as an upper bound on the right-hand side of the   inequality,
\begin{equation}
 \lll \clZ_{\theta}\lll_{G,H}   \le   \lll  \One\otimes \pi_{\theta} \lll_{G,H}+    \lll U_{\theta}R_{\theta} \lll_{G,H}  [1+  \lll\One\otimes \pi_{\theta} \lll_G].
\label{e:fundGHpre-bdd}
\end{equation}
 The first step   is to express~(V3) in terms of the resolvent.   Writing~\eqref{e:V3}  as $[P-I]H \le  -G+ b s$ we apply $R_{\theta}$ to each side and use  $R_\theta[P_\theta-I] = R_\theta-I$
 to obtain a version of~(V3) for the resolvent kernel,    $ [R_{\theta}-I]H = R_{\theta}[P-I]H \le  -R_{\theta}G+ b R_{\theta}s$, and hence $R_{\theta}G  \le  bR_{\theta}s + [I - (R_{\theta}-s\otimes \nu)] H $.

Following standard arguments (e.g.,~\cite[Lem.~3.2]{konmey03a}) we conclude that:
\begin{equation}
 U_{\theta} R_{\theta} G \le H + b U_{\theta} R_{\theta} s    =  H + b     U_{\theta} R_{\theta} s.
 \label{e:URGbdd}
\end{equation}
 From the definition of $U_{\theta}$ we have $U_{\theta} R_{\theta} =    U_{\theta} (R_{\theta}-  s\otimes \nu)  +  [ U_{\theta}  s] \otimes \nu    = U_{\theta}-I +  [ U_{\theta}  s] \otimes \nu $.
Recalling $\mu_{\theta}(s) =1$ and   $U_{\theta} s \equiv 1$ gives 
 $ U_{\theta} R_{\theta} s  \le U_{\theta} s + \nu(s)  U_{\theta} s   \le 2  $,  and   from~\eqref{e:URGbdd},
 \[
  \lll U_{\theta}R_{\theta} \lll_{G,H}   \le   \sup_x \frac{1}{H(x)}  [  H(x)  + b U_{\theta} R_{\theta} \, s(x)]   \le  b_u \eqdef  1 + 2b .
 \]
 
Next, consider $ \lll  \One\otimes \pi_{\theta} \lll_{G,H}  =  \sup_x  \pi_{\theta}(G)/H(x) \le \pi_{\theta}(G)$.     We have,
\[
\pi_{\theta}(G)   =  \frac{ \mu_{\theta}  (G) }  { \mu_{\theta} (\state)}  \le  \mu_{\theta}  (G)  =  \mu_{\theta}  (R_{\theta}G)  = 
\nu (  U_{\theta} R_{\theta} G )\le 
\nu(H)   \lll U_{\theta}R_{\theta} \lll_{G,H}  \le 
\nu(H)  b_u,
\]
where the first inequality follows from $\mu_{\theta} = \nu U_{\theta}\ge \nu$,  so that $ \mu_{\theta} (\state) \ge 1$, 
and   the  identity that follows is a consequence of invariance of $\mu_{\theta}$. 

  To bound $\nu(H) $ we combine the pair of bounds  $R_{\theta} H \, (x) \ge s(x)  \nu(H)$  and $R_{\theta}H \le H + b     s(x)  $,    
  to obtain  $\nu(H) \le  b      +  H(x)/s(x)   $.   The ratio admits the bound  
   $H(x)/s(x)  \le  \| s /H \|_\infty^{-1}     $  for all $x$,   from which we obtain   $\nu(H) \le  b_\nu \eqdef    b + \| s /H \|_\infty^{-1}     $,  and hence   $ \lll  \One\otimes \pi_{\theta} \lll_{G,H}  \le b_u b_\nu$.      This and~\eqref{e:fundGHpre-bdd}
 completes the proof.   
\end{proof}

 Lipschitz bounds on $\haf_\theta$  are obtained by invoking  the \textit{resolvent equation},
 \begin{equation}
\clZ_\theta -  \clZ_{\theta'}
=  \clZ_\theta  [\tilP_\theta - \tilP_{\theta'}  ] \clZ_{\theta'}  \,,\qquad \theta\,,\ \theta'\in\Re^d,
\label{e:Fund-z-perturb}
\end{equation}
and its corollary $
\pi_\theta - \pi_{\theta'}  =   \pi_\theta [P_\theta - P_{\theta'}  ]  	\clZ_{\theta'}$. 
This is one ingredient in perturbation theory for Markov chains~\cite{sch68}, which is at the heart of the actor-critic   method of RL---see commentary in~\cite{CSRL}. \Cref{t:pi-pert} provides conditions under 
which~\eqref{e:Fund-z-perturb} is justified and an essential bound. 

\begin{lemma}
\label[lemma]{t:pi-pert}
Suppose that for functions measurable functions $H_i   \colon \state\to [1,\infty)$,  $i=1,2,3$, 
 we have
$\lll   P_\theta \lll_{H_i}  +  \lll \clZ_\theta \lll_{H_i, H_{i+1}}   <\infty$,
and    $\pi_\theta(H_{i+1})<\infty$
for each  $i=1,2$ and each $\theta$.

Then,  
$\displaystyle
\lll \tilP_\theta - \tilP_{\theta'}  \lll_{H_1}   \le   \lll P_\theta - P_{\theta'}  \lll_{H_1}   +  \pi_\theta(H_2) \lll P_\theta - P_{\theta'}  \lll_{H_2}   \lll	\clZ_{\theta'}\lll_{H_1,H_2}
$
\
 for each $\theta,\theta'\in\Re^d$.
\end{lemma}

When~\eqref{e:bv} holds we have $\lll \clZ_\theta \lll_v \le   b_v /( 1-\varrho  )$.   Consequently,  if~(A2) also holds then   $\haf_\theta \in\Lv$,   and 
from~\eqref{e:Fund-z-perturb} we obtain a useful formula for differences:   For $\theta,\theta'\in\Re^d$,  
\begin{equation}
\haf_\theta -  \haf_{\theta'}   
	= 
  \clZ_\theta [f_\theta - f_{\theta'} ]  +   \clZ_\theta  [\tilP_\theta - \tilP_{\theta'}  ] \     \haf_{\theta'}.
\label{e:haf-diffs}
\end{equation}

Far better bounds are obtained in~\Cref{t:bounds-H} whose proof is based on the following:
\begin{proposition}
\label[proposition]{t:fund-kernel-z}	
	{\bf (i)} Under~{\em (V4)},  $\sup_\theta\lll \clZ_\theta \lll_{v^\epsy} < \infty $ for each  $\epsy\in (0,1]$.

\smallskip

\noindent
{\bf (ii)} Under~{\em (DV3)},    $\sup_\theta\lll \clZ_\theta \lll_{G_p, V_p  } < \infty $  with $G_p = 1 + W V^{p-1}$,   $V_p = 1 + V^p$,   for each      $p\ge 1$.
\end{proposition}

\begin{proof}
Part~(i) follows from Jensen's inequality: If~(V4) holds,  
then the same drift condition holds for $v^\epsy$ using $P_\theta v^\epsy \le (P_\theta v)^\epsy$; see~\cite[Th.~2.3]{glymey96a} or~\cite[Lem.~15.2.9]{MT}.	

The proof of part~(ii) also begins with Jensen's inequality.   The function $\clG(x) = \log(x)^p$ is concave on the interval $[a_p,\infty)$ with $a_p = \exp(p-1)$.    Consequently, under~(DV3),
\[
\begin{aligned}
\Expect[ | V(\Phi^\theta _{k+1}) &+ W(\Phi^\theta _k) -  b s(\Phi^\theta _k)  +k_p |^p \mid \clF_k ]
\qquad\qquad \text{(where $k_p = a_p + b $)}
   \\
            & = \Expect[\clG \big( \exp \big(V(\Phi^\theta _{k+1}) + W(\Phi^\theta _k) -  b s(\Phi^\theta _k)  +k_p \big) \big)  \mid \clF_k ]
            \\
            &\le \clG \big( \exp \big(V(\Phi^\theta _{k})  +k_p \big) \big)  =    \big|  V(\Phi^\theta _{k})  + k_p \big|^p.
\end{aligned} 
\]
  From this we   obtain a version of the drift condition~\eqref{e:V3}: 
  \[
\Expect[ V_p( \Phi^\theta _{k+1})   \mid \clF_k ]  \le V_p( \Phi^\theta _k)  - \delta_p G_p( \Phi^\theta _k) 
+ b_p \ind_S( \Phi^\theta _k),
\]
where $\delta_p>0$,  $b_p<\infty$, and  $S = \{x :  V_p(x)\le m \}$ for some $m$.
The proof is completed on combining~\Cref{t:V3bdds} and~\Cref{t:uniSmall}.
\end{proof}

\begin{proposition}
\label[proposition]{t:bounds-H}	Suppose that~{\em (A2)}  holds.
We then obtain a Lipschitz bound on  $\haf$ defined in~\eqref{e:hafZ} under one of the following conditions:
\begin{subequations}
\whamem{(i)}
If~{\em (V4)} holds and $L^4\in \Lv$, then  there exists a constant $b_{\haf} < \infty$ such that:
\label{e:bounds-L}
	\begin{align}
			\| \haf(\theta, x)\| & \leq b_{\haf} v(x)^{\fourth} \bigl[1+\|\theta\| \bigr] \,, 
		 \label{e:growth-L}
	 \\
			\| \haf(\theta, x) - \haf(\theta', x)\| & \leq b_{\haf} v(x)^{\fourth} \|\theta - \theta'\|\,, \qquad \theta,\theta'\in\Re^d\,,   x\in\state \, .
 \label{e:lip-L}
\end{align}
\end{subequations}

\begin{subequations}
 \label{e:bdds+lip-H}
\whamem{(ii)}
If~{\em (DV3)} holds and $L\in L_\infty^W$, then  there exists a constant $b_{\haf} < \infty$ such that:
		\label{e:bounds-H}
	\begin{align}
			\| \haf(\theta, x)\| & \leq b_{\haf}\bigl(1+V(x) \bigr) \bigl[ 1+\|\theta\| \bigr]\,, 
 \label{e:growth-H} 
			\\
			  \| \haf(\theta, x) - \haf(\theta', x)\|  &    \leq b_{\haf} \bigl(1+  V(x)  \bigr) \|\theta - \theta'\|\,,   \qquad \theta,\theta'\in\Re^d\,,   x\in\state.%
\label{e:lip-H}%
\end{align}%
\end{subequations}%
Under either~(i) or~(ii),  the function $\haf$ solves Poisson's equation in the form~\eqref{e:fish-f}.
\end{proposition}

\begin{proof}
We adopt the alternate notation  $ \haf_\theta (x) = \haf(\theta, x)  $   and $f_\theta(x) =f (\theta,x)$.    
 
We begin with~(i).    It follows from~\eqref{e:lip-f} and the assumption  $L^4\in \Lv$ that,
\[
\text{
$ 
	\| f_\theta(x)\| \leq L(x)[1+\|\theta\|]\leq  b_f(\theta)  v(x)^{\fourth}
$,    \quad  where $b_f(\theta) =
\|L\|_{v^\fourth}   [1+\|\theta\|  ]$.  }
\]
 That is,   $  \| f_\theta \|_{v^{1/4}}  \le   b_f(\theta)$,   and  we then obtain~\eqref{e:growth-L}: 
\[
		\frac{1}{v(x)^{\fourth}}\| \haf_\theta( x)\|   \le    \| \clZ_\theta f_\theta\|_{v^{1/4}} 
		 \leq     b_f(\theta)  \lll \clZ_\theta \lll_{v^\fourth}    \,, \quad \text{  for each $    \theta\in\Re^d$, $ x\in\state$,}
\]

The proof of~\eqref{e:lip-L} uses~\eqref{e:haf-diffs}:  
\[
\begin{aligned} 
\|  \haf_\theta  -  \haf_{\theta'}   \|_{v^{1/4}} 
 		 &  \le
	   \lll \clZ_\theta \lll_{v^\fourth}   \|  f_\theta - f_{\theta'}  \|_{v^\fourth}  
	+
	   \lll \clZ_\theta \lll_{v^\fourth}    \lll \tilP_\theta - \tilP_{\theta'} \lll_{v^\fourth}   \|    \haf_{\theta'}  \|_{v^\fourth}
	   \\
	    		 &  \le
	   \lll \clZ_\theta \lll_{v^\fourth}    \| L\|_{v^\fourth} \|\theta-\theta'\|  
	+
	   \lll \clZ_\theta \lll_{v^\fourth}    \lll \tilP_\theta - \tilP_{\theta'} \lll_{v^\fourth}   \|    \haf_{\theta'}  \|_{v^\fourth},
\end{aligned} 
\]
where the second inequality follows 
from~\eqref{e:lip-f} in~(A2).  Applying~\eqref{e:growth-L} and~\Cref{t:pi-pert} with $H_i= v^\fourth$  for each $i$ 
we bound the second term,
\[
 \lll \tilP_\theta - \tilP_{\theta'} \lll_{v^\fourth}   \|    \haf_{\theta'}  \|_{v^\fourth}
 		\le  \Bigl( \frac{B_d}{1+   \| \theta \| + \| \theta' \|  }  
		  \| \theta - \theta' \|    \Bigr)     b_{\haf}  \bigl[1+\|\theta\| \bigr] 
		\le  b_{\haf}  B_d  \| \theta - \theta' \| ,
\]
where $
\displaystyle B_d =  b_d \sup_{\theta,\theta'} [1+     \pi_\theta(H)    \lll	\clZ_{\theta'}\lll_{G} ]$.

 Part~(ii) requires repeated applications of~\Cref{t:fund-kernel-z}~(ii), which tells us that   $\clZ_\theta$ is a bounded linear operator from  $L_\infty^{G_p} $ to $L_{\infty}^{V_p}$ for each $p$,   with   $G_p = 1 + W V^{p-1}$ and $V_p = 1 + V^p$.    	Recalling~\eqref{e:Vnorm}, this is expressed $\lll \clZ_\theta\lll _{G_p, V_p}<\infty $ for each $p\ge 1$.   Applying this result with $p=1$   gives   $   \| \haf_\theta \|  _{V_1} \le   \lll \clZ_{\theta'} \lll_{G_1,V_1}  \|    f_{\theta}  \|_{G_1}$.  This bound   combined with~\eqref{e:lip-f}
implies~\eqref{e:growth-H}. 
 
 The proof of~\eqref{e:lip-H} uses~\eqref{e:haf-diffs} as in the 
proof of~(i):  We have, with $p=2$,
\[
\begin{aligned} 
\|  \haf_\theta & -  \haf_{\theta'}   \|_{V_2} 
 		\le
	   \lll \clZ_\theta \lll_{G_2,V_2}   \|  f_\theta - f_{\theta'}  \|_{G_2}  
	+
	   \lll \clZ_\theta \lll_{G_2,V_2}      \lll \tilP_\theta - \tilP_{\theta'} \lll_{V_1,G_2}    
	     \lll \clZ_{\theta'} \lll_{G_1,V_1} 
	    \|    f_{\theta'}  \|_{G_1}\,.
\end{aligned} 
\]
Each term is bounded via~\Cref{t:fund-kernel-z}~(ii) as in the proof 
of~\eqref{e:lip-L}, except for a slightly different application 
of~\Cref{t:pi-pert}:
\[
   \lll \tilP_\theta - \tilP_{\theta'} \lll_{V_1,G_2}  \le    \lll \tilP_\theta - \tilP_{\theta'} \lll_{V_1}   
   	 \le  \frac{B_d}{1+   \| \theta \| + \| \theta' \|  }  
		  \| \theta - \theta' \| ,
\] 
where  $B_d =      b_d \sup_{\theta,\theta'} [1+  \pi_\theta(V_1)    \lll	\clZ_{\theta'}\lll_{V_1,V_2} ]  \le   b_d \sup_{\theta,\theta'} [1+  \pi_\theta(V_1)    \lll	\clZ_{\theta'}\lll_{G_2,V_2} ]  $. 
\end{proof}

\begin{subequations}

The following  is a crucial corollary to \Cref{t:bounds-H}.

\begin{proposition}
\label[proposition]{t:lip-noise}
Under~{\em (A2)} we obtain bounds on the terms in 	the decomposition of $\Delta_{n+1}$ in~\Cref{t:noise-decomp}.    Under~{\em (V4)},
\begin{equation}
\begin{aligned}
			\| \MD_{n+1} \| & \leq b_{\haf} \bigl[ (\Expect[v^\fourth(\Phi_{n+1}) \mid \clF_{n}] + v^\fourth(\Phi_{n+1}) \,  \bigr] \bigl[1  + \|\theta_n\| \bigr]  
\\
			\|\clT_{n+1}\|    & \leq  b_{\haf} v^\fourth(\Phi_{n+1}) \bigl[1 + \|\theta_{n}\| \bigr] 
			 \\
			\| \Oops_{n+1}\| & \leq   b_{\haf} \|L\|_{v^\fourth}   v^\half(\Phi_{n+1}) \bigl[1 + \|\theta_n\| \bigr] \, .    
\end{aligned}
\label{e:lip-noise-L}
\end{equation}
Under~{\em (DV3)},
\begin{equation}
\begin{aligned}
	\| \MD_{n+1} \| & \leq b_{\haf} \bigl[ 2 + \Expect[V(\Phi_{n+1}) \mid \clF_{n+1}] + V(\Phi_{n+1}) \bigr] \bigl[1  + \|\theta_n\| \bigr] 
						 \\
			\|\clT_{n+1}\|    & \leq b_{\haf}[1+V(\Phi_{n+1})] \bigl[1 + \|\theta_{n}\| \bigr]
			\\
			\| \Oops_{n+1}\| & \leq   b_{\haf} [1+V(\Phi_{n+1})] L(\Phi_{n+1})  \bigl[1 + \|\theta_n\| \bigr].
\end{aligned} 
\label{e:lip-noise}%
\end{equation}
\end{proposition}
\end{subequations}%

\subsection{Almost-Sure Bounds}
\label{s:ASbdd}

As discussed immediately after~\Cref{t:BigConvergence},  boundedness of 
the parameter sequence $\{\theta_k\}$   is sufficient to obtain a.s.\ 
convergence under~(A1) and~(A2). \Cref{t:step-size} is one ingredient in 
establishing convergence;  the  bounds are easily established under~(A1).   
\begin{lemma}
\label[lemma]{t:step-size}
	Under {\em (A1)}, the following hold:  
	{\rm (i)}   $\displaystyle 
    \lim_{n \to \infty} \frac{ \alpha_n}{ \alpha_{n+1} } =1$,  
	\[ 
  \text{\rm (ii)} \
  \sum_{n=1}^{\infty} |\alpha_{n+1} - \alpha_{n}| <\infty
    \qquad
  \text{\rm (iii)} \   
  \lim_{n \to \infty} \sum_{k=m_n}^{m_{n+1}-1} |\sqrt{\alpha_{k+1}} - \sqrt{\alpha_{k+2}}|= 0.
  \]
 \end{lemma}

Recall from~\Cref{s:ode-method-infinity} that we introduce a ``hat'' on the parameter or its candidate ODE approximation to denote the scaled process.
In particular,   $ \hattheta_k  =  \theta_k /{c_n}$  for    $m_n \leq k < m_{n+1}$,  with  $c_n \eqdef \max\{1, \| \theta_{m_n}\|\}$.  
In~\Cref{t:bddpar}
 we show that the ODE approximation~\eqref{e:ode-method-infty} for $\{\hattheta_k\}$ implies boundedness of $\{\theta_k\}$.   The remainder of the section is devoted to verifying the~\eqref{e:ode-method-infty} under the assumptions of~\Cref{t:BigConvergence}.

 \begin{proposition}
\label[proposition]{t:bddpar}
If~\eqref{e:ode-method-infty} holds for each initial condition, then there is a deterministic constant $ \bdd{t:bddpar}<\infty$ such that $\displaystyle
\limsup_{k\to\infty }  \|\theta_k\| \le  \bdd{t:bddpar}$,
for each   $(\theta_0,\Phi_0)$.
\end{proposition}

\begin{proof}
  Assuming that  $T_{n+1} - T_n \ge \Trelax$ for each $n$, 
  and recalling the definition~\eqref{e:rescaled-ode},
  \[
 \frac{1}{\max\{1, \| \theta_{m_n}\|\} }   \| \theta_{m_{n+1}}\|  =   \|\hatODEstate_{T_{n+1}} \|  \le   
  \| \hatodestate_{T_{n+1} - }  \|  +  \|\hatODEstate_{T_{n+1}}  -  \hatodestate_{T_{n+1} -}  \|    \le  \half + o(1) \,, 
 \]
 in which $\displaystyle \hatodestate_{T_{n+1} -} \eqdef  \lim_{t\uparrow T_{n+1} } \hatodestate_{t}$, 
 from which it follows that $\limsup_{n\to\infty}   \| \theta_{m_n}\| \le 2$  a.s.     
 Applying the same arguments to $  \| \theta_k\|  /   \| \theta_{m_n}\|$, 
 \[
  \frac{1}{\max\{1, \| \theta_{m_n}\|\} }  \max_{m_n\le k < m_{n+1}}   \| \theta_k\|     
   \le   
\sup_{T_n\le t<T_{n+1}}   \| \hatodestate_{t }  \|  +  \sup_{T_n\le t<T_{n+1}}  \|\hatODEstate_t  -  \hatodestate_t \|    .
 \]
 This combined with~\eqref{e:ode-method-infty} establishes the desired bound with  $ \bdd{t:bddpar} = 2 \sup_{0\le t\le T+1}  \sup_{\hatodestate}  \| \hatodestate_{t }  \|  $,   where the inner supremum is over all solutions to the ODE@$\infty$ satisfying $\| \hatodestate_0\| \le 1$.
\end{proof}

Given~\Cref{t:bddpar}, the remaining   work required in establishing boundedness of the parameter estimates rests on establishing   solidarity with the ODE@$\infty$:

 \begin{proposition}
\label[proposition]{t:hat-ode-approx-as}
If~{\em (V4)} and~{\em (A1)--(A2)}  hold, then~\eqref{e:ode-method-infty} follows:
\[
 \lim_{n \to \infty}  \sup_{T_n\le t < T_{n+1} }  \|\hatODEstate_t -  \hatodestate_t \| = 0 \,, 
 \quad \text{and}\qquad
	\sup_{k\geq 0} \|\hattheta_k\| < \infty\quad a.s.
\]
 \end{proposition}

 The proof of the proposition takes up the remainder of this subsection.   
 While many arguments follow closely~\cite{bormey00a} and~\cite[Th.~4.1]{bor20a},  the Markovian setting introduces additional complexity.  
Key inequalities established here are also required for moment bounds.

We begin with an   application of the fundamental kernel and~\Cref{t:fund-kernel-z}.   
\begin{proposition}
\label[proposition]{t:LLNscaled}
Suppose that the assumptions of~\Cref{t:BigConvergence}
hold, so that in particular~{\em (V4)} holds along with the bound~\eqref{e:LipPtheta} using $v$ and $L^8\in \Lv$. 
Consider any function $g\colon\state\to\Re$ satisfying   $|g|^{8/3} \in   \Lv$,  and denote   $\tilg_\theta = g - \pi_\theta(g)$.     
Then,  
the partial sums converge: For each initial condition $\Phi_0,\theta_0$, there is a square-integrable random variable $S^g_\infty$ such that,
\[
\lim_{n\to\infty}  \sum_{i= 1}^n \alpha_i  \tilg_{\theta_i}(\Phi_i)   =  S^g_\infty  \qquad 
\mbox{a.s.}
\]       
\end{proposition}

\begin{proof}
The proof is similar to~\cite[Part~II, Sec.~1.4.6, Prop.~7]{benmetpri12},  though more complex because the noise is not exogenous. 
 In the exogenous case we only require $g^2\in\Lv$, and $\tilg_\theta$ is independent of $\theta$.

Denote $ \hag_\theta  \eqdef    \clZ_\theta g$.   
 \Cref{t:fund-kernel-z} tells us that $ |\hag_\theta(x)|   \le b_g v^{3/8}(x)  $ for a fixed constant $b_g$ and all $\theta, x$.   
 We have
$
\Expect[ \hag_{\theta_{i}}( \Phi_{i+1}) \mid \clF_i ]  =  \hag_{\theta_{i}}( \Phi_{i}) -  \tilg_{\theta_i}(\Phi_i)   
$
 for  each $i$, from which we  obtain $\tilg_{\theta_i}(\Phi_i)   =    \MD_i^g - \{  \clT^g_i-  \clT^g_{i-1}   \}  +     \alpha_i \clR^g_i    $,   where,
\[
  \MD_i^g  = \hag_{\theta_{i}}( \Phi_{i+1} ) - 
  \Expect[   \hag_{\theta_{i}}( \Phi_{i+1} )  \mid \clF_{i} ]    \,,
\quad
  \clT^g_i=     \hag_{\theta_{i}}( \Phi_{i+1})  \,, 
  \quad
   \clR^g_i  =  \tfrac{1}{\alpha_i}  [ \hag_{\theta_{i}}( \Phi_i) - \hag_{\theta_{i-1}}( \Phi_i) ]  \, .
\] 
Hence the partial sums of interest can be expressed,
\begin{equation}
 S^g_n \eqdef
 \sum_{i= 1}^n \alpha_i  \tilg_{\theta_i}(\Phi_i)   =  M^g_n -      \sum_{i= 1}^n   \alpha_i \{ \clT^g_i -  \clT^g_{i-1} \}  +
   \sum_{i= 1}^n   \alpha_i^2     \clR^g_i 
 \,,
\label{e:Sgn}
\end{equation}
where $ M^g_n = 
      \sum_{i= 1}^n   \alpha_i  \MD_i^g  $ is a 
  martingale that is square integrable:
\[
\Expect[ ( M^g_n )^2] =        \sum_{i= 1}^n   \alpha_i^2 \Expect[( \MD_i^g  )^2 ] 
	 \le b_g^2  \bigl(   \sup_{n \ge 0} \Expect[  v(\Phi_n)    ]\bigr)  \sum_{i= 1}^\infty  \alpha_i^2 .
\]
Therefore, it is convergent to a square-integrable random variable, denoted $ M^g_\infty $.  

Convergence of the second term in~\eqref{e:Sgn}
is established using summation by parts:
\[
  \sum_{i= 1}^n   \alpha_i \{ \clT^g_i -  \clT^g_{i-1} \}    
=
 \alpha_n \clT^g_n  -
 \alpha_1 \clT^g_0  -
  \sum_{i= 1}^{n-1}  \clT^g_i    \{ \alpha_{i+1} -\alpha_i \} .   
\]
From this we obtain a candidate expression for the limit:
\begin{equation}
S_\infty^g = M_\infty^g  + R_\infty^\tau+ R_\infty^g  + \alpha_1 \clT^g_0  \,,  \ \    \text{with}  \ \ 
R_\infty^\tau = \sum_{i= 1}^\infty  \clT^g_i    \{ \alpha_{i+1} -\alpha_i \}   \,, \ \ 
R_\infty^g =
   \sum_{i= 1}^\infty   \alpha_i^2     \clR^g_i.
\label{e:Sg_def}
\end{equation} 
Justification requires $\lim_{n\to\infty}  \alpha_n \clT^g_n  = 0$ a.s.,   and the existence of the two 
infinite sums in~\eqref{e:Sg_def}
as square integrable random variables. 

The first limit is obtained as follows:
\[
\Expect\Bigl[  \sum_{n=1}^\infty \{  \alpha_n \clT^g_n \}^2 \Bigr]  
\le \Bigl(\sup_{n\ge0}    \Expect [   \{   \clT^g_n \}^2  ] \Bigr) \sum_{n=1}^\infty  \alpha_n^2 <\infty.
\]
This implies that $ \alpha_n \clT^g_n\to0$ as $n\to\infty$ a.s.\ and in $L_2$.
Similarly, applying the triangle inequality in $L_2$,
\[
\| R_\infty^\tau \|_{L_2}  \le 
 \sum_{i= 1}^\infty  \| \clT^g_i  \| _{L_2}    |\alpha_{i+1} -\alpha_i | 
 \le
 \sqrt{  \sup_{n\ge0}    \Expect [   \{   \clT^g_n \}^2  ] }  \,    \sum_{i= 1}^\infty |\alpha_{i+1} -\alpha_i |   <\infty.
\]

To show that $ R_\infty^g  $ exists as a random variable in $L_2$ it is sufficient to obtain a uniform $L_2$ bound on $\{\clR^g_i \}$.
In view of~\eqref{e:Fund-z-perturb}  
we have for any $x$ and $i$,  and with    $H=v^{3/8}$,
\[
\frac{1}{H (x) } | \hag_{\theta_{i}}( x) - \hag_{\theta_{i-1}}( x) |
\le 
b_g
\lll \clZ_{\theta} \lll_{H}   \clZ_{\theta'}  \lll_{H}   \lll  \tilP_\theta - \tilP_{\theta'}   \lll_{H} 
\,, \quad \theta = \theta_i\,, \ \theta' = \theta_{i-1}.
\]
This, combined with~\eqref{e:LipPtheta} implies that there is a constant $b_g'$ such that,
\[
|   \clR^g_i |^2   \le    b_g'  \big[ \tfrac{1}{\alpha_i}  H(\Phi_i)  \| \theta_i - \theta_{i-1} \| \big]^2    
    \le    b_g'      H^2(\Phi_i)   L(\Phi_i)^2    \le b_g'    v^{1/4}(\Phi_i)  v^{3/4} (\Phi_i)    \le b_g'    v (\Phi_i) .
\] 
Applying~\Cref{t:PthetaBdd} it follows that $\sup_i \Expect[ |   \clR^g_i |^2   ]<\infty$.
\end{proof}

This is now used in the following proof establishing boundedness of  $\{\hattheta_k\}$. 

\begin{subequations}
\begin{lemma}
\label[lemma]{t:hat-bound-as}
Under~{\em (V4)} and~{\em (A1)--(A2)}, we have:
\begin{align}
1 + \| \hattheta_{k+1} \|  
			 &\leq \exp\bigl(\alpha_{k+1} L(\Phi_{k+1})\bigr) (1 + \|\hattheta_k\|) \,, \quad   k\ge 0,
\label{e:exp-bound-hat-as} 
			  \\
\limsup_{k\to\infty}  \|\hattheta_k\|  &\le  2 \exp\bigl(\barpi(L) T \bigr) \,,
\quad    \text{where $\barpi(L) = \sup_\theta  \pi_\theta(L)<\infty$.
}
\label{e:bound-hat-as} 
\end{align}
\end{lemma}

\end{subequations}

\begin{proof}
We have $  	\| \hattheta_{k+1}\|  \le \|  \hattheta_k \| + \alpha_{k+1}  \| f_{c_n}(\hattheta_k, \Phi_{k+1}) \|$ for each $m_n \leq k < m_{n+1}$.
The Lipschitz bound for $f$ is inherited by $f_{c_n}$, giving $\|f_{c_n}(\hattheta_k, \Phi_{k+1})\| \leq L(\Phi_{k+1})(1 + \|\hattheta_k\|)$. 
Applying the bound $1+z\le  e^z $ with  $z =   \alpha_{k+1} L(\Phi_{k+1}) $
  gives~\eqref{e:exp-bound-hat-as}.

On iterating this bound we obtain,  for each $m_n < k \leq m_{n+1}$,
\[
	1 + \| \hattheta_k \| \leq (1 + \|\hattheta_{m_n}\|) \exp\Bigl(\sum_{i=m_n+1}^{k} \alpha_{i} L(\Phi_{i})\Bigr) \leq 2 \exp\Bigl(\sum_{i=m_n+1}^{k} \alpha_{i} L(\Phi_{i})\Bigr),
\]
which implies, with with $g_\theta = L - \pi_\theta(L)$,
\[
\begin{aligned}
		\max_{m_n\leq k \leq m_{n+1}} \|\hattheta_k\| 
		& \leq 2 \exp\Bigl(\sum_{i=m_n+1}^{m_{n+1}} \alpha_{i} L(\Phi_{i})\Bigr) 
		\\
		& \leq 2 \exp\bigl(\barpi(L) [T_{n+1} - T_n] \bigr) \exp\Bigl(\sum_{i=m_n+1}^{m_{n+1}} \alpha_{i} g_{\theta_{i}}(\Phi_{i}) \Bigr),
\end{aligned}
\]
where we used  $ \|\hattheta_{m_n}\|\le 1$ in the first inequality. The bound~\eqref{e:bound-hat-as} 
follows from~\Cref{t:LLNscaled}.        
\end{proof}

\begin{subequations}

The simple proof of the following error bounds  is omitted.   

\begin{lemma}
\label[lemma]{t:DiscreteTimePointsBdd}
Under {\em (A1)} there is a constant $b_0$ such that for each $n \ge 1$,   $m_n<k\le m_{n+1}$ and  $t\in [\SAtime_{k-1},\SAtime_k]$,
\[
\begin{aligned}
\| \ODEstate_t - \odestate^{(n)}_t \|  &  \le  \max \{ \| \ODEstate_{\SAtime_k} - \odestate^{(n)}_{\SAtime_k}  \|     ,\| \ODEstate_{\SAtime_{k-1}} - \odestate^{(n)}_{\SAtime_{k-1}}  \|     \}    
			+ b_0   \max\{   \| \odestate^{(n)}_{\SAtime_k}  \| ,  \|\odestate^{(n)}_{\SAtime_{k-1}}  \| \}\alpha_k
   \\
\| \hatODEstate_t - \hatodestate_t \|  &  \le  \max \{ \| \hatODEstate_{\SAtime_k} - \hatodestate_{\SAtime_k}  \|     ,\| \hatODEstate_{\SAtime_{k-1}} - \hatodestate_{\SAtime_{k-1}}  \|     \}    
			+ b_0   \max\{   \| \hatodestate_{\SAtime_k}  \| ,  \|\hatodestate_{\SAtime_{k-1}}  \| \}\alpha_k.
\end{aligned} 
\]
 \end{lemma}

Bounds on  $ \|\hatODEstate_t -  \hatodestate_t \| $ for $t\in \{ \SAtime_k : k\ge 0\}$
 require the representation of $\Delta_{k+1}$ in~\eqref{e:noise-decomp}.   
 Denote the scaled variables by: 
\begin{equation}
\begin{aligned}
\hatDelta_{k+1}  &\eqdef   \haMD_{k+1} - \haclT_{k+1} + \haclT_{k} + \alpha_{k+1}\haOops_{k+1} 
   \\
            & \eqdef  [  \MD_{k+1}   - \clT_{k+1} + \clT_k - \alpha_{k+1}\Oops_{k+1} ]/c_n, \qquad m_n \leq k < m_{n+1}.
\end{aligned} 
\label{e:haDD}
\end{equation}
On scaling the bounds in~\Cref{t:lip-noise} we obtain,
\begin{align}
		\| \haMD_{k+1} \| & \leq b_{\haf}\bigl[ \Expect[ v^\fourth(\Phi_{k+1}) \mid \clF_{k+1}] +  v^\fourth(\Phi_{k+2}) \, \bigr] \bigl[1  + \|\hattheta_k\| \bigr]  \label{e:lip-md-L-hat}
		\\
		\|\haclT_{k+1}\|    & \leq  b_{\haf}  v^\fourth(\Phi_{k+1})      \bigl[1 + \|\hattheta_{k+1}\| \bigr] \label{e:lip-tele-L-hat} 
		\\
		\| \haOops_{k+1}\| & \leq  b_{\haf} \|L\|_{v^\fourth}   v^\half(\Phi_{k+1}) \bigl[1 + \|\hattheta_k\| \bigr]. \label{e:lip-epsy-L-hat}%
\end{align}

	\label{e:lip-noise-L-hat}

\end{subequations}%

\begin{proposition}
	\label[proposition]{t:noise-diminish-as}
	Under~{\em (V4)} and~{\em (A1)--(A2)}:
\whamem{(i)}
  The martingale   $\displaystyle \Bigl\{ \sum_{k=1}^n \alpha_{k} \haMD_{k} : n\ge 1\Bigr\}$ converges a.s.\  as $n \to \infty$.
  
\whamem{(ii)} $\displaystyle \lim_{n\to \infty} \!\! \max_{m_n\leq k < m_{n+1}} \alpha_{k+1} \| \haclT_{k+1} \| = 0$
			\hfill
\textbf{\emph{(iii)}} $\displaystyle\lim_{n\to \infty} \sum_{k=m_n}^{m_{n+1}-1} |\alpha_{k+1} - \alpha_{k+2}| \| \haclT_{k+1} \|= 0$ a.s.

\whamem{(iv)} $\displaystyle \lim_{n\to \infty} \sum_{k=m_n}^{m_{n+1}-1} \alpha_{k+1}^2 \| \haOops_{k+1}\| = 0$~a.s.
 
\end{proposition}

\begin{proof}
\Cref{t:hat-bound-as} tells us that    $b_{\hattheta}  \eqdef \sup_{k\ge 1}  \|\hattheta_k\|   <\infty $ a.s.
Given~\eqref{e:lip-md-L-hat},
\[
\begin{aligned}
	\Expect[\|\haMD_{k+1}\|^2 \mid \clF_{k}]  
			&\leq 4 b_{\haf}^2 \Expect[v^\half(\Phi_{k+1})\mid \clF_{k}](1 + \|\hattheta_k\|)^2 
	\\
	&\leq 4 b_{\haf}^2 (1 + b_{\hattheta})^2 \Expect[v^\half(\Phi_{k+1}) \mid \clF_{k}]\,, \qquad \text{ for each $k\geq 1$,}
 \\
	\sum_{k=1}^{\infty}  \alpha_{k}^2   \Expect[\|\haMD_{k+1}\|^2 &\mid \clF_{k}]  
		\leq 4 b_{\haf}^2 (1 + b_{\hattheta})^2 \sum_{k=1}^\infty \alpha_{k}^2 \Expect[v^\half(\Phi_{k+1}) \mid \clF_{k}]\,, \quad \mbox{a.s.}
	\end{aligned} 
\]
The right-hand side is finite a.s.\ since by~\Cref{t:PthetaBdd} and Jensen's inequality,
\[
\Expect\Bigl[ \sum_{k=1}^\infty \alpha_{k}^2 \Expect[v^\half(\Phi_{k+1}) \mid \clF_{k}] \Bigr] 
	\le b_v^\half  v^\half(x) 	 \sum_{k=1}^\infty \alpha_{k}^2
	< \infty \,, \quad x=\Phi_0\in\state 
\]
Part~(i) then follows by martingale convergence.

	For~(ii), it follows from~\eqref{e:lip-tele-L-hat} that for each $n\geq 1$,
\[
\max_{m_n < k \le m_{n+1}} \alpha_{k}\|\haclT_{k} \| \leq b_{\haf}(1+b_{\hattheta}) \max_{m_n < k \le m_{n+1}} \alpha_{k}  v^\fourth(\Phi_{k}).
\]
On denoting $\tilg_\theta = v^\fourth - \pi_\theta(v^\fourth)$,
 the series $\{  \sum_{k=0}^n \alpha_{k}  \tilg_{\theta_{k}}(\Phi_{k}) \}$ is convergent, by \Cref{t:LLNscaled},  
 and hence a Cauchy sequence.  
Consequently,
 \[
   \lim_{n\to\infty} 
\max_{m_n < k \le m_{n+1}}   \alpha_{k} v^\fourth(\Phi_{k})  =
   \lim_{n\to\infty} \max_{m_n < k \le m_{n+1}}  |  \alpha_{k}  \tilg_{\theta_{k}}(\Phi_{k})  |   =0\, .
 \]

Part~(iii) follows from~\Cref{t:step-size}~(ii) and~\Cref{t:LLNscaled}, and arguments similar to the proof of~(ii). 
Part~(iv)  follows from the fact that $ \| \alpha_{k}\haOops_{k}\| $   is vanishing as $k\to\infty$: It is square summable in $L_2$ and almost surely du
to~\Cref{t:hat-bound-as}
 and~\eqref{e:lip-noise}.
\end{proof}

The identity~\eqref{e:sum-by-parts} combined with~\eqref{e:hat-theta-sa}  gives for $k>m_n$,
\begin{equation} 
\begin{aligned}
\qquad
	 \hattheta_k  
				=  \hattheta_{m_n} 
				  + & \sum_{i=m_n}^{k-1} \alpha_{i+1} \barf_{c_n}(\hattheta_i) + 	 \haclE_{k} \,,
					 \\
	 \quad \text{with}
					 \quad   \haclE_{k}  \eqdef   
					 	 \sum_{i=m_n+1}^{k} &    \big[ \alpha_{i}  \haMD_{i}  -  \alpha_{i}^2\haOops_{i}  -  [\alpha_{i} - \alpha_{i+1}]\haclT_{i} \big] + \alpha_{m_n+1}\haclT_{m_n}	- \alpha_{k+1}\haclT_{k}  
  \end{aligned}
	\label{e:hat-theta-cumu-all}
\end{equation}

\begin{subequations}%

\Cref{t:pre-hat-ode} is based on~\eqref{e:hat-theta-cumu-all} and the representation:
\begin{equation}
		\label{e:hat-ode-integral}
		\hatodestate_{\SAtime_{k+1}}  = \hattheta_{m_n}  + \int_{T_n}^{\SAtime_{k+1}} \barf_{c_n}( \hatodestate_t) \, dt \,, 
		\quad k\ge m_n \, .
\end{equation}
\begin{lemma}
\label[lemma]{t:pre-hat-ode}
For $m_n\le k<m_{n+1}$,
\begin{align} 
			\hattheta_{k+1} - \hatodestate_{\SAtime_{k+1}} 
			& =   \sum_{i=m_n}^k\alpha_{i+1} \barf_{c_n}(\hattheta_i) 
				 -  \int_{T_n}^{\SAtime_{k+1}} \barf_{c_n}( \hatodestate_t) \, dt +\haclE_{k+1} 
\label{e:hat-ode-error-1} 
\\ 
		& =   \sum_{i=m_n}^k\alpha_{i+1} \{  \barf_{c_n}(\hattheta_i)
 						-    \barf_{c_n} (\hatodestate_{\SAtime_i} )\}  
								 +\haclE_{k+1}  +\haclE_{k+1}^{\clD} \,,
   \label{e:hat-ode-error-2} 
\end{align}
where $\{ \haclE_k:  m_n < k \le m_{n+1} \}$ is given in~\eqref{e:hat-theta-cumu-all}, and   $\haclE_{k+1}^{\clD} $ is the discretization error resulting from the Riemann-Stieltjes approximation of the integral~\eqref{e:hat-ode-error-1}.  
\end{lemma}

\label{e:pre-hat-ode-eqns}
\end{subequations}%

\begin{proof}
The identities (\ref{e:hat-ode-error-1}, \ref{e:hat-ode-error-2}) follow from~\eqref{e:hat-theta-cumu-all} and~\eqref{e:hat-ode-integral}. 
\end{proof}

\begin{subequations}%

We next obtain bounds on the error sequences:
\begin{equation}
\begin{aligned}
b^{\clD}(n) \eqdef\max_{m_n < k\le m_{n+1}} \|\haclE_{k}^{\clD} \|     \,,
\qquad
b^{\clN}(n)   \eqdef \max_{m_n < k\le m_{n+1}} \| \haclE_{k}\|
\end{aligned}
\label{e:bDN}
\end{equation}
from which we obtain the desired ODE approximation for the scaled parameter sequence.

\begin{lemma}
\label[lemma]{t:hat-ode-approx-as-lemma}
Under~{\em (V4)} and~{\em (A1)--(A2)},
the limit~\eqref{e:ode-method-infty}  holds, along with  
the bounds,
\begin{equation}
	\max_{m_n \leq j \leq m_{n+1}} \| \hattheta_j - \hatodestate_{\SAtime_j} \| \leq b(n) \exp\bigl( \barL  T \bigr) \, .
\label{e:haBGtheta}
\end{equation}
where $\barL  $ is the  Lipschitz constant for  $\barf$, 
and $b(n) =b^{\clD}(n)  + b^{\clN}(n) $ is a vanishing sequence:
\begin{align}
b^{\clD}(n)
& \le  \bdd{t:hat-ode-approx-as-lemma}  \sum_{i=m_n+1}^{m_{n+1}} \alpha_{i}^2 \,,  \quad 
	\text{with $ \bdd{t:hat-ode-approx-as-lemma} $ a deterministic constant};
 \label{e:discrete-error-O}
\\
\label{e:disturbance-bound}
			b^{\clN}(n) 
			& \leq \max_{m_n \leq j < m_{n+1}}  \Big\{   
				\big\|   \sum_{i=m_n+1}^{k}   \alpha_{i}  \haMD_{i}  \big\|   +
			  \alpha_{i+1} \| \haclT_{i}  \|  \Big\}+ \alpha_{m_n+1} \| \haclT_{m_n}	 \|
			  \\
			  & \qquad +  \sum_{i=m_n+1}^{k}   \Big\{  \alpha_{i}^2  \| \haOops_{i} \|   +| \alpha_{i} - \alpha_{i+1} | \|\haclT_{i} \| \Big\}
			\nonumber
\end{align}
\end{lemma}

\end{subequations}%

\begin{proof}
	The bound~\eqref{e:discrete-error-O} follows from uniform Lipschitz continuity of $  \barf_{c_n}$,  and 
 \eqref{e:disturbance-bound} follows from the   definition of $\haclE_{k}$ in~\eqref{e:hat-theta-cumu-all}.	  
The conclusion that $b(n) =b^{\clD}(n)  + b^{\clN}(n) $ is a vanishing sequence follows from~\eqref{e:discrete-error-O}  and~\Cref{t:noise-diminish-as}.  

Lipschitz continuity of $\barf_{c_n}$ also gives,
\[
 \sum_{i=m_n}^k\alpha_{i+1} \{  \barf_{c_n}(\hattheta_i)
 						-    \barf_{c_n} (\hatodestate_{\SAtime_i} )\} 
						\le
						\barL \sum_{i=m_n}^k \alpha_{i+1} \| \hattheta_{i} - \hatodestate_{\SAtime_{i}}  \|,
\]
where $\barL  $ is the common Lipschitz constant for  any of the scaled functions $ \{ \barf_{c} : c>0\}  $.   
Consequently,  from~\eqref{e:hat-ode-error-2} 
and the triangle inequality,  
\begin{equation}
\max_{m_n \leq j \leq k+1} \| \hattheta_j - \hatodestate_{\SAtime_j} \| 
			   \leq  \barL \sum_{i=m_n}^k \alpha_{i+1} \max_{m_n \leq j \leq i} \| \hattheta_j - \hatodestate_{\SAtime_j} \|  + b(n).
		\label{e:hat-ode-error-2x}
\end{equation}
The discrete Gronwall inequality and~\eqref{e:hat-ode-error-2x} implies \eqref{e:haBGtheta}.  

		Finally,  \eqref{e:hat-ode-error-2x}
 together with~\Cref{t:DiscreteTimePointsBdd}   completes the proof that~\eqref{e:ode-method-infty}  holds.
 \end{proof}

\subsection{\boldmath{$L_p$} Bounds}
\label{s:Lp}

The proof of~\Cref{t:BigBounds} is similar to the proof of~\Cref{t:BigConvergence}, but more challenging.  
Starting with the scaled iterates $\{\hattheta_k\}$, we first show 
in~\Cref{t:bounded-hat-l4} that the conditional fourth moment of $\hattheta_k$ is uniformly bounded in $k$. \Cref{t:hat-ode-approx} establishes an ODE approximation of $\{\hattheta_k\}$ in the $L_4$ sense, based on the almost sure ODE approximation~\Cref{t:hat-ode-approx-as}. The ODE approximation of $\{\hattheta_k\}$ combined with Assumption~(A3) leads to a recursive ``contraction inequality''  in~\Cref{t:contract-1/4}   which then quickly leads to the proof of~\Cref{t:BigBounds}.

\wham{Details of Proof}

 Recall that $T>0$ specifies the length of time blocks used in the ODE method described in~\Cref{s:ode-method}. We fix $T=\Trelax$ throughout this section with $\Trelax$ defined in~(A4).  All of the bounds obtained in this Appendix are for ``large $n$'':  
 Let $n_g\ge 1$ denote a fixed integer satisfying,
\begin{subequations}
 \begin{equation}
\bar{\alpha} \eqdef \sup \{  \alpha_n :  n\ge n_g \} \leq \frac{3}{4}\Bigl(\frac{1}{4\delta_L} - \Trelax\Bigr),
 \label{e:ng}
\end{equation}
where $\delta_L \eqdef \|L\|_W$.	
 For $n\ge n_g$ we obtain under~(A4),
\begin{equation}
	4\delta_L\sum_{k=m_n}^{m_{n+1}-1}  \alpha_{k+1}
		 \leq 4\delta_L (\Trelax + \bar{\alpha}) \leq 4\delta_L \Bigl(\Trelax + \frac{3}{4}\Bigl(\frac{1}{4\delta_L} - \Trelax\Bigr) \Bigr) = \frac{3}{4} + \frac{1}{4}\Trelax \cdot 4\delta_L < 1.
\end{equation}
\label{e:block-bound}
\end{subequations}

Recall  the definition $\hattheta_k \eqdef \theta_k / c_n$ for $m_n \leq k \leq m_{n+1}$, where $c_n = \max\{1, \|\theta_{m_n}\|\}$.

\begin{proposition}
	\label[proposition]{t:bounded-hat-l4}
	Under~{\em (DV3)} and~{\em (A1)--(A4)},  the following bounds hold for each $n \ge n_g$ and $m_n \le k <   m_{n+1} $:
\whamem{(i)}
$\displaystyle
 \Expect[\exp\bigl( V(\Phi_{k+1}) \bigr)(1 + \|\hattheta_k\|)^{4} \mid \clF_{m_n+1}]  \leq 16b_v^2 e^b \exp\bigl( V(\Phi_{m_n+1}) \bigr) ;
		$
\whamem{(ii)}
$\displaystyle 
			(\|\theta_k\|+1)^4 \leq 16 \exp\bigl(4\delta_L  \sumalphaW_{m_n\, ,  m_{n+1} -1}      \bigr) (\|\theta_{m_n}\| + 1)^4$,
			where,
\begin{equation}
\sumalphaW_{\ell,k}  \eqdef   \sum_{j=\ell}^{k} \alpha_{j+1}W(\Phi_{j+1})   \,,   \qquad  0\le \ell\le k   .
\label{e:sumalphaW}
\end{equation}
\end{proposition}

\begin{proof}    
The bound $L \le \delta_L W$ combined with~\eqref{e:exp-bound-hat-as} gives,
\[
\begin{aligned}
		1 + \| \hattheta_{k+1} \| 		
		& \leq  \exp\bigl(\delta_L \alpha_{k+1} W(\Phi_{k+1})\bigr) (1 + \|\hattheta_k\|).
\end{aligned}
\]  
Iterating this inequality, we obtain for each $m_n \leq k < m_{n+1}$,
\begin{equation}
		\label{e:exp-bound-hat}
	\begin{aligned}
			 1 + \| \hattheta_{k+1} \| 
			&\leq \exp\bigl(\delta_L   \sumalphaW_{m_n,k}       \bigr) (1 + \|\hattheta_{m_n}\|)
			\\
	&\quad \Longrightarrow  			c_n + \| \theta_{k+1} \| 
	 \leq \exp\bigl(\delta_L   \sumalphaW_{m_n,k}   \bigr) (c_n + \|\theta_{m_n}\|).	
\end{aligned}
\end{equation}
Then~(ii) follows from the facts that   $1\le c_n\le 1 + \|\theta_{m_n}\|$.

Since  $1 + \|\hattheta_{m_n}\| \leq 2$ by construction,~\eqref{e:exp-bound-hat} also gives,
\[
\begin{aligned}
	\exp\bigl(V(\Phi_{k+1})\bigr)(1 + \| \hattheta_{k+1} \|)^4
		& \leq 16\exp\Bigl( V(\Phi_{k+1}) +4\delta_L  \clS_{m_n,k}^W\Bigr)
\\
		\Expect\bigl[ \exp(V(\Phi_{k+1})) (1 + \| \hattheta_{k+1} \|)^4 \mid \clF_{m_n+1} \bigr] 
		&\leq 16\Expect\bigl[ \exp\Bigl( V(\Phi_{k+1}) + 4\delta_L       \sumalphaW_{m_n,k}	
		\Bigr) \mid \clF_{m_n+1}  \bigr] \\
		&\leq 16b_v^2 e^b \exp\bigl(V(\Phi_{m_n+1}) \bigr),
\end{aligned}
\]
where the final inequality follows from~\Cref{t:DV3multBdd} and~\eqref{e:block-bound}. This establishes~(i).
\end{proof}

For any random vector $X$, we denote $\|X\|_4^{(n)}\eqdef \Expect \bigl[\|X\|^4 \mid \clF_{m_n+1} \bigr]^{\frac{1}{4}}$, and obtain bounds for various choices of $X$.

\begin{proposition}
\label[proposition]{t:hat-noise-bounds}
Under~{\em (DV3)} and~{\em (A1)--(A4)}, there exists a constant $\bdd{t:hat-noise-bounds} < \infty$ such that for all $n \geq n_g$ and $m_n\le k< m_{n+1} $, 
\begin{subequations}
		\label{e:hat-bound-noise}
	\begin{align}
			 \| \haMD_{k+1}\|_4^{(n)} &\leq \bdd{t:hat-noise-bounds}   \exp(\fourth V(\Phi_{m_n+1})) 
\label{e:hat-bound-md} 
\\
			  \| \haclT_{k+1} \|_4^{(n)} &\leq \bdd{t:hat-noise-bounds}  \exp(\fourth V(\Phi_{m_n+1}))  
\label{e:hat-bound-tele} 
\\
			 \| \haOops_{k+1}\|_4^{(n)} & \leq \bdd{t:hat-noise-bounds}   \exp(\fourth V(\Phi_{m_n+1})). \label{e:hat-bound-epsy}
\end{align}%
\end{subequations}%
\end{proposition}

\begin{proof}
	Bounds identical to~\eqref{e:lip-noise} hold for the scaled noise components: 
\[
\begin{aligned}
		\| \haMD_{k+1} \| & \leq b_{\haf} \bigl(2 + \Expect[V(\Phi_{k+1}) \mid \clF_k] + V(\Phi_{k+1}) \bigr) \bigl[1  + \|\hattheta_k\| \bigr] \\
		\|\haclT_{k+1}\|    & \leq b_{\haf}  \bigl[ 1+  V(\Phi_{k+1})    \bigr] \bigl[1 + \|\hattheta_{k}\| \bigr] \\
		\| \haOops_{k+1}\| & \leq  b_{\haf}   \bigl[ 1+  V(\Phi_{k+1})    \bigr] L(\Phi_{k+1})  \bigl[1 + \|\hattheta_k\| \bigr] .
\end{aligned}
\]
Consider first the martingale difference term $\haMD_{k+1}$:
\[
	\|\haMD_{k+1}\|^4 \leq  b_{\haf}^4 \bigl(2 + \Expect[V(\Phi_{k+1}) \mid \clF_{k}] + V(\Phi_{k+1}) \bigr)^4(1 + \|\hattheta_k\|)^4.
\]
Taking conditional expectations gives,
\[
\begin{aligned}
		\Expect & \bigl[\|\haMD_{k+1}\|^4 \mid \clF_{m_n+1} \bigr]  
		\\
		& \leq b_{\haf}^4 \Expect\Bigl[\bigl(2 + \Expect[V(\Phi_{k+1}) \mid \clF_k] + V(\Phi_{k+1}) \bigr)^4(1 + \|\hattheta_k\|)^4\mid \clF_{m_n+1} \Bigr] \\
		& = b_{\haf}^4  \Expect\Bigl[ (1 + \|\hattheta_k\|)^4\Expect\bigl[\bigl(2 + \Expect[V(\Phi_{k+1}) \mid \clF_k] + V(\Phi_{k+1}) \bigr)^4\mid \clF_k \bigr] \mid \clF_{m_n+1} \Bigr] .
\end{aligned}
\]

We next apply the bound
the bound $V^4 \le b_4^e e^V$ for a finite constant  $b_4^e$, along with 
 $\bigl(2 + \Expect[V(\Phi_{k+1}) \mid \clF_k] + V(\Phi_{k+1}) \bigr)^4 \leq 32 + 16\Expect[V^4(\Phi_{k+1}) \mid \clF_k] + 16V^4(\Phi_{k+1}) $,  
to conclude, 
\[
\begin{aligned}
\Expect\Bigl[\bigl(2 + \Expect[V(\Phi_{k+1}) \mid \clF_k] + V(\Phi_{k+1}) \bigr)^4\mid \clF_k \Bigr] 
	& \leq 32 + 32\Expect[V^4(\Phi_{k+1}) \mid \clF_k] 
	\\
	 &\le  32[1 +  b_4^e] \Expect[  \exp(V(\Phi_{k+1})) \mid \clF_k] 
\end{aligned}
\]
Therefore,  $\displaystyle
		\Expect\bigl[ \|\haMD_{k+1}\|^4 \mid \clF_{m_n+1} \bigr] 
		\leq 32[1 + b_4^e] \Expect\bigl[\exp\bigl( V(\Phi_{k+1}) \bigr)(1 + \|\hattheta_k\|)^4 \mid \clF_{m_n+1}  \bigr]$.   
The bound~\eqref{e:hat-bound-md}  then follows from~\Cref{t:bounded-hat-l4}~(i).
	
The telescoping term  admits the bound 
$\displaystyle
	\|\haclT_{j+1}\|^4 \leq  b_{\haf}^4   [ 1+ V(\Phi_{j+1}) ]^4   [ 1 + \|\hattheta_j\|]^4
$ for any $m_n\le j < m_{n+1}$.   
The fact $\|(1+ V)^4\|_v < \infty$ gives~\eqref{e:hat-bound-tele}.
	
Turning to the final term, we have for some $b_0<\infty$,
\[
\begin{aligned}
		\Expect \bigl[ \|\haOops_{k+1}\|^4 \mid \clF_{m_n+1} \bigr]
	\leq   b_{\haf}^4 \Expect\bigl[  \bigl(1 + \|\hattheta_k\|\bigr)^4 
		&	\big\{  [1+V(\Phi_{k+1})] L(\Phi_{k+1}) \big\}^4 \mid \clF_{m_n+1} \bigr] 
\\
   \textit{with} \quad& \big\{  [1+V(\Phi_{k+1})] L(\Phi_{k+1}) \big\}^4   \le  b_0 \exp(  V(\Phi_{k+1})) \,,
\end{aligned}
\]
where the existence of  $b_0 <\infty$ in the
 second inequality follows from the fact that  $\|L^p\|_v<\infty$ and 
  $\|(1+ V)^p\|_v < \infty$    for any $p\ge 1$ under~(A2).  
Consequently, 
\[
	\Expect\bigl[ \|\haOops_{k+1}\|^4 \mid \clF_{m_n+1} \bigr] \leq   b_{\haf}^4 b_0 \Expect\bigl[ (1 + \|\hattheta_k\|)^4 \exp(V(\Phi_{k+1}))  \mid \clF_{m_n+1} \bigr].
\]
The desired bound in~\eqref{e:hat-bound-epsy} follows from another application of~\Cref{t:bounded-hat-l4}~(i).
\end{proof}

\begin{lemma}
	\label[lemma]{t:bound-md}
	Let $\{X_k\}$ be a martingale difference sequence satisfying $\Expect[\| X_k\|^4] < \infty$ for each $k$. There exists a constant $\bdd{t:bound-md}<\infty$ such that for any non-negative scalar sequence $\{\delta_k\}$ and $m>0, n > m$, 
\[
	\Expect[\max_{m\leq k\leq n} \| \sum_{i=m}^k \delta_i X_i \|^4 ] \leq \bdd{t:bound-md} \bigl(\sum_{i=m}^n \delta_i^2\bigr)^2 \max_{m\leq k\leq n} \Expect[\|X_k\|^4].
\]
\end{lemma}

\begin{proof}
 Burkholder's inequality~\cite[Lem.~6, Ch.~3, Part~II]{benmetpri12} implies the desired bound:  For a constant $\bdd{t:bound-md}<\infty$,
\[
\begin{aligned}
		\Expect\bigl[ \max_{m \leq k \leq n} \bigl\| \sum_{i = m}^k \delta_{i} X_{i} \bigl\|^4 \bigr] 
		\leq \bdd{t:bound-md} \Expect\bigl[ \Bigl(\sum_{i = m}^{n} \delta_{i}^2 \|X_{i}\|^2 \Bigr)^2 \bigr] 
		& \leq \bdd{t:bound-md} \Expect\bigl[ \Bigl(\sum_{i = m}^{n} \delta_{i}^2\Bigr) \Bigl(\sum_{i = m}^{n}\delta_{i}^2 \|X_{i}\|^4 \Bigr) \bigr] \\
		& \leq \bdd{t:bound-md} \bigl(\sum_{i=m}^n \delta_i^2\bigr)^2 \max_{m\leq k\leq n} \Expect[\|X_k\|^4],
\end{aligned}
\]
where the second inequality follows from the Cauchy-Schwarz inequality.
\end{proof}

The bound~\eqref{e:hat-ode-error-2x} continues to hold:
\begin{equation}
\begin{aligned}
		\max_{m_n \leq j \leq k+1} \| \hattheta_j - \hatodestate_{\SAtime_j} \| 
		&  \leq  \barL \sum_{i=m_n}^k \alpha_{i+1} \max_{m_n \leq j \leq i} \| \hattheta_j - \hatodestate_{\SAtime_j} \| + b^{\clD}(n)  + b^{\clN}(n) \,, 
\end{aligned}
\label{e:e:hat-ode-error-2x-RepeatingOurselves}
\end{equation} 
The first error term is bounded by a deterministic vanishing sequence   
(recall~\eqref{e:discrete-error-O}),  and  $\{ b^{\clN}(n) \}$ defined 
in~\eqref{e:bDN}  is bounded in the following.

\begin{lemma}
\label[lemma]{t:bclN4}
For a deterministic vanishing sequence $\{ \bdds{t:bclN4}(n) : n\ge 1 \} $ 
we have:
$$\| b^{\clN}(n) \|_4^{(n)} \le \bdds{t:bclN4} (n) \exp( V(\Phi_{m_n+1})).$$
\end{lemma}

\clearpage

 \begin{proof}
Applying~\Cref{t:bound-md} to the sum of martingale difference terms 
in~\eqref{e:disturbance-bound},
\[
\begin{aligned}
		\Expect\bigl[ \max_{m_n < j \leq k} \bigl\| \sum_{i = m_n+1}^{j} \alpha_{i} \haMD_{i} \bigl\|^4 \mid \clF_{m_n+1} \bigr] 
		&\leq \bdd{t:bound-md}  \Bigl(\sum_{i = m_n+1}^{m_{n+1}} \alpha_{i}^2\Bigr)^2 \max_{m_n <  k \le m_{n+1}} \Expect\bigl[ \|\haMD_{k}\|^4 \mid \clF_{m_n+1} \bigr] 
		\\
		& \leq \bdd{t:bound-md} [\bdd{t:hat-noise-bounds}]^4 \exp(V(\Phi_{m_n+1})) \Bigl(\sum_{i = m_n+1}^{m_{n+1}} \alpha_{i}^2\Bigr)^2 ,
\end{aligned}
\]
where the second inequality follows from~\eqref{e:hat-bound-md} in~\Cref{t:hat-noise-bounds}.
	
Consider next the supremum involving $\{\haclT_i\}$:
\[
\begin{aligned}
		\Expect[\max_{m_n \le i \leq k} \alpha_{i+1}^4 \|\haclT_{i}\|^4 \mid \clF_{m_n+1} ]
		& \leq \sum_{i = m_n}^{m_{n+1}-1} \alpha_{i+1}^4 \Expect[\|\haclT_{i}\|^4 \mid \clF_{m_n+1} ] \\
		& \leq [\bdd{t:hat-noise-bounds}]^4 \exp\bigl(V(\Phi_{m_n+1})\bigr)\sum_{i = m_n}^{m_{n+1}-1} \alpha_{i+1}^4.
\end{aligned}
\]
By the triangle inequality and~\eqref{e:hat-bound-tele} of~\Cref{t:hat-noise-bounds},
\[
\begin{aligned}
		\Bigl\|  \sum_{i=m_n+1}^{k} |\alpha_{i} - \alpha_{i+1}| \|\haclT_{i}\|   \Bigr\|_4^{(n)} 
		& \leq  \sum_{i=m_n+1}^{k} |\alpha_{i} - \alpha_{i+1}| \|\haclT_{i}\|_4^{(n)}
		\\
		& \leq \bdd{t:hat-noise-bounds} \exp\bigl(\fourth V(\Phi_{m_n+1})\bigr)\sum_{i = m_n+1}^{m_{n+1}}|\alpha_{i} - \alpha_{i+1}| .
\end{aligned}
\]
\Cref{t:step-size}~(ii) asserts that the sum vanishes as $n\to\infty$.
	
	Similarly,  it follows from~\eqref{e:hat-bound-epsy} of~\Cref{t:hat-noise-bounds} that for $k < m_{n+1}$,  
\[
\begin{aligned}
	\Bigl\| \sum_{i = m_n+1}^k \alpha_{i}^2\| \haOops_{i}\|\Bigr\|_4^{(n)} 
			& \leq \bdd{t:hat-noise-bounds} \exp(\fourth V(\Phi_{m_n+1})) \sum_{i=m_n+1}^{m_{n+1}} \alpha_{i}^2 .
\end{aligned} 
\]
Combining these bounds completes the proof.
\end{proof}

The next result extends~\Cref{t:hat-ode-approx-as} to an $L_4$ bound.  
\begin{proposition}
\label[proposition]{t:hat-ode-approx}
Under~{\em (DV3)} and~{\em (A1)--(A4)},  there is a vanishing deterministic sequence $\{\bdds{t:hat-ode-approx}_n : n\ge 1\}$ such that:
\begin{equation}
\label{e:hat-ode-approxLP}
		\Expect[\sup_{t\in [T_n, T_{n+1})} \| \hatODEstate_t -  \hatodestate_t\|^4 \mid \clF_{m_n+1} ]^\fourth \leq \bdds{t:hat-ode-approx}_n\exp(\fourth V(\Phi_{m_n+1})).
\end{equation}
\end{proposition}

\begin{proof}
Returning to~\eqref{e:e:hat-ode-error-2x-RepeatingOurselves}, recall that $ b^{\clD}(n) $ is a vanishing deterministic sequence.  
Applying~\Cref{t:bclN4},   we  can find a deterministic vanishing sequence $\{\barb(n) \}$ such that,
\[
\begin{aligned}
\Expect\bigl[  \max_{m_n \leq j \leq k+1} & \| \hattheta_j - \hatodestate_{\SAtime_j} \|^4 \mid \clF_{m_n+1} \bigr]^{\frac{1}{4}}
\\
			&  \leq   \barL\sum_{i=m_n}^k \alpha_{i+1} \Expect\bigl[ \max_{m_n \leq j \leq i} \| \hattheta_j - \hatodestate_{\SAtime_j} \|^4 \mid \clF_{m_n+1} \bigr]^{\frac{1}{4}} +  \barb(n) \exp(\fourth V(\Phi_{m_n+1})),
\end{aligned}
\]
for all $m_n\leq k < m_{n+1}$.
By the discrete Gronwall's inequality,
\[
	\Expect\bigl[  \max_{m_n \leq j \leq m_{n+1}} \| \hattheta_j - \hatodestate_{\SAtime_j} \|^4 \mid \clF_{m_n+1} \bigr]^{\frac{1}{4}} \leq \barb(n) \exp(\fourth V(\Phi_{m_n+1}))\exp(\barL[T+ 1]).
\]
This combined with~\Cref{t:DiscreteTimePointsBdd} completes the proof.  
\end{proof}

\begin{lemma}
	\label[lemma]{t:contract-1/4}
	Under~{\em (DV3)} and~{\em (A1)--(A4)}, we can find a constant $\bdd{t:contract-1/4}<\infty$ such that for all $n\ge n_g$, 
\begin{equation}
		\Expect \bigl[\| \theta_{m_{n+1}}\|^4 \mid \clF_{m_n+1}\bigr]^{\frac{1}{4}} \leq \Bigl( \Crelax +  \bdds{t:hat-ode-approx}_n \exp(\fourth V(\Phi_{m_n+1})) \Bigr)(\|\theta_{m_n}\| + 1)  + \bdd{t:contract-1/4},
				\label{e:contract-1/4}
\end{equation}
	where $0 < \Crelax < 1$ is defined in~\Cref{t:contraction-ratio}.
\end{lemma}

\begin{proof}
The proof is similar to~\Cref{t:bddpar}.
We begin with~\Cref{t:contraction-ratio},  giving:
 \begin{equation}
		\label{e:c0-contract}
		\|\phi_c(t,\theta_0)\| \leq \Crelax < 1\,, \text{ for all } t\in [\Trelax, \Trelax+1] \,,   \ \ c \geq c_0.
\end{equation}
By the triangle inequality, for each $n\ge n_g$ and  $c_{m_n} =\max\{1, \| \theta_{m_n}\|\} $,
\[
\begin{aligned}
  \frac{1}{c_{m_n}} 
  &
		\Expect \bigl[\| \theta_{m_{n+1}}\|^4 \mid \clF_{m_n+1}\bigr]^{\frac{1}{4}}
		 =   \Expect\bigl[\| \hattheta_{m_{n+1}}\|^4 \mid \clF_{m_n+1}\bigr]^{\frac{1}{4}}   
		 \\
		&  \leq \Expect\bigl[\| \hatODEstate_{T_{n+1}}- \hatodestate_{T_{n+1}}\|^4  \mid \clF_{m_n+1}\bigr]^{\frac{1}{4}}   
  + \Expect\bigl[\| \hatodestate_{T_{n+1}}\|^4 \, \ind\{\|\theta_{m_n}\|\geq c_0\} \mid \clF_{m_n+1}\bigr]^{\frac{1}{4}}    \\
		& \qquad + \Expect\bigl[\| \hatodestate_{T_{n+1}}\|^4 \, \ind\{\|\theta_{m_n}\| < c_0\}  \mid \clF_{m_n+1}\bigr]^{\frac{1}{4}}  .
\end{aligned}
\]
As in the proof of~\Cref{t:bddpar}, 
$\displaystyle
\| \hatodestate_{T_{n+1}}\|^4 \, \ind\{\|\theta_{m_n}\| < c_0\}   \le  \sup_{0\le t\le T+1}  \sup_{\hatodestate}  \| \hatodestate_{t }  \|  $,
   where the inner supremum is over all solutions to the ODE@$\infty$ satisfying $\| \hatodestate_0\| \le c_0$.

 \Cref{t:hat-ode-approx}  and~\eqref{e:c0-contract} give, respectively,
 \[
 \begin{aligned}
  	\Expect\bigl[ \| \hatODEstate_{T_{n+1}}- \hatodestate_{T_{n+1}}\|^4  \mid \clF_{m_n+1} \bigr]^{\frac{1}{4}} 
						&\leq \bdds{t:hat-ode-approx}_n  \exp(\fourth V(\Phi_{m_n+1})),
\\
\Expect\bigl[\| \hatodestate_{T_{n+1}}\|^4 \, \ind\{\|\theta_{m_n}\|\geq c_0\} \mid \clF_{m_n+1} \bigr]^{\frac{1}{4}} (\|\theta_{m_n}\| \vee 1) 
					& \leq \Crelax (\|\theta_{m_n}\| \vee 1).
\end{aligned} 
\]
Combining these bounds gives~\eqref{e:contract-1/4}.
\end{proof}

The ODE approximation in~\eqref{e:hat-ode-approxLP} combined with  the asymptotic stability of the origin for $\ddt \odestate_t = \barf_{\infty}(\odestate_t)$ (see~(A3)) leads to the following ``contraction bound'':
 \begin{lemma}
\label[lemma]{t:contract-4}
Under~{\em (DV3)} and~{\em (A1)--(A4)}, there exists constants $0 < \bdde{t:contract-4} < 1$, $\bdd{t:contract-4} < \infty$,    and a deterministic and vanishing sequence $\{ \bdds{t:contract-4}_n  : n\ge n_g\}$	 such that  for all $n\ge n_g$,
 \begin{equation}
		\label{e:contract-4}
		\Expect \bigl[ \bigl(\| \theta_{m_{n+1}}\| + 1 \bigr)^4 \mid \clF_{m_n+1}\bigr]\leq \Bigl( \bdde{t:contract-4} +  \bdds{t:contract-4}_nv(\Phi_{m_n+1}) \Bigr) \bigl(\|\theta_{m_n}\| + 1 \bigr)^4  + \bdd{t:contract-4}.
\end{equation}
\end{lemma}

The following simple bound is used twice in the proof:
 \begin{lemma}
\label[lemma]{t:fourth-epsy} 
 $\displaystyle	b_{c,\epsy} \eqdef \sup_x \{ (x+c)^4 - (1+\epsy) x^4 \}< \infty$,  for any   $c, \epsy >0$.  
\end{lemma}

\begin{proof}[Proof of~\Cref{t:contract-4}]
By~\Cref{t:fourth-epsy}, for any $\epsy>0$, we have,
\begin{equation}
		\label{e:contract-1-epsy}
		\Expect[( \|\theta_{m_{n+1}}\| + 1)^4 \mid \clF_{m_n+1}] \leq (1+\epsy)	\Expect[\|\theta_{m_{n+1}}\|^4 \mid \clF_{m_n+1}] + b_{1,\epsy}.
\end{equation}
	It follows from~\eqref{e:contract-1/4} that,
\[
	\Expect \bigl[\| \theta_{m_{n+1}}\|^4 \mid \clF_{m_n+1}\bigr] \leq \Bigl\{
	\bigl( \Crelax +  \bdds{t:hat-ode-approx}_n \exp(\fourth V(\Phi_{m_n+1})) \bigr)(\|\theta_{m_n}\| + 1)  + \bdd{t:contract-1/4} \Bigr\}^4.
\]
By~\Cref{t:fourth-epsy} once more, for any $\epsy_0 >0$, 
\begin{equation}
		\label{e:contract-2-epsy}
	\begin{aligned}
			\Bigl\{
			\bigl( \Crelax +&  \bdds{t:hat-ode-approx}_n \exp(\fourth V(\Phi_{m_n+1})) \bigr) (\|\theta_{m_n}\| + 1)  + \bdd{t:contract-1/4}
			\Bigr\}^4 \\	
			&\leq (1+\epsy_0)\Bigl( \Crelax +  \bdds{t:hat-ode-approx}_n \exp(\fourth V(\Phi_{m_n+1})) \Bigr)^4(\|\theta_{m_n}\| + 1)^4  + b_1,
\end{aligned}
\end{equation}
with $b_1$ a   constant that depends on $\bdd{t:contract-1/4}$ and $\epsy_0$. 
The following bound holds because   $0<\Crelax<1$ and $V(x)\geq 0$ for $x\in\state$,
\begin{equation}
		\label{e:contract-3-epsy}
	\begin{aligned}
\bigl( \Crelax +  \bdds{t:hat-ode-approx}_n \exp(\fourth V(\Phi_{m_n+1})) \bigr)^4 \leq  \Crelax^4 +  \clE_0(n) \exp(V(\Phi_{m_n+1})),
\end{aligned}
\end{equation}
where  $ \clE_0(n) \eqdef [\bdds{t:hat-ode-approx}_n]^4 + 4\bdds{t:hat-ode-approx}_n + 6[\bdds{t:hat-ode-approx}_n]^2+ 4[\bdds{t:hat-ode-approx}_n]^3$  is a   vanishing sequence.  

We arrive at the bounds, 
\[
\begin{aligned}
&	\Expect \bigl[\| \theta_{m_{n+1}}\|^4 \mid \clF_{m_n+1}\bigr] \leq \Bigl\{ (1+\epsy_0)  \Crelax^4 +  (1+\epsy_0)  \clE_0(n) v(\Phi_{m_n+1})\Bigr\}(\|\theta_{m_n}\| + 1)^4 + b_1,
\\
&		\Expect   [(\|\theta_{m_{n+1}}\| + 1)^4 \mid \clF_{m_n+1}] 
		\\
&\quad		     \leq   (1+\epsy)	\Bigl\{ (1+\epsy_0)  \Crelax^4 +  (1+\epsy_0)  \clE_0(n)v(\Phi_{m_n+1})\Bigr\}(\|\theta_{m_n}\| + 1)^4
  + b_\epsy + (1+\epsy)  b_1,
\end{aligned}
\]
where the first inequality is obtained on combining~\eqref{e:contract-2-epsy} and~\eqref{e:contract-3-epsy}, and the second follows from the first and~\eqref{e:contract-1-epsy}.

Since $\epsy$ and $\epsy_0$ can be taken arbitrarily small, we can find $\bdde{t:contract-4}< 1$ and $\bdd{t:contract-4} <\infty$ such that the desired bound holds for all $n\ge n_g$,
\[
	\Expect \bigl[(\| \theta_{m_{n+1}}\| + 1)^4 \mid \clF_{m_n+1}\bigr] \leq \Bigl\{ \bdde{t:contract-4} +  \bdds{t:contract-4}_n \exp(V(\Phi_{m_n+1}))\Bigr\}(\|\theta_{m_n}\| + 1)^4 + \bdd{t:contract-4},
\]
with $ \bdds{t:contract-4}_n= (1+\epsy)(1+\epsy_0) \clE_0(n)$.
\end{proof}

\Cref{t:contract-4}  inspires us to construct a Lyapunov function of the form $L_{m_n} =  \La_{m_n}  + \beta \Lb_{m_n} $, with  $\La_{m_n}  = (\| \theta_{m_n} \| + 1)^4  $ and,
\[
\Lb_{m_n} =   ( \| \theta_{m_n} \| + 1)^4 \Expect[    \exp\bigl( V(\Phi_{m_n +1 }) + \epsy^\circ  W(\Phi_{m_n })  \bigr)     \mid \clF_{m_n }   ]   .
\]
It is assumed throughout that  $0 < \beta < 1$ and  $\epsy^\circ \eqdef \frac{1}{4} - \delta \Trelax > 0$. It follows from the identity in~\eqref{e:block-bound} that for all $n \geq 0$,
\begin{equation}
	\epsy^\circ + 4\delta \sum_{k = m_n+1}^{m_{n+1}}  \alpha_k \leq \frac{1}{4} - \delta \Trelax + \frac{3}{4} + \frac{1}{4}\Trelax \cdot 4\delta \leq 1.
	\label{e:block-length-epsy}
\end{equation}
The desired contraction is as follows:  For 	constants $ \varrho_0 <1$  $b_0<\infty$ and integer $n_0\ge 1$,
\begin{equation}
	\Expect[ L_{m_{n+1}}    \mid \clF_{m_n  }   ]   \le  \varrho_0  L_{m_n}    +b_0    \,,  \quad  n\ge n_0.
\label{e:Lcontraction}
\end{equation}
It will follow that the mean of $L_{m_n}$ is uniformly bounded, and
then the proof of~\Cref{t:BigBounds} is completed by 
invoking~\Cref{t:bounded-hat-l4}~(ii) and~\Cref{t:DV3multBdd}.

\begin{proof}[Proof of~\Cref{t:BigBounds}]
Most of the work involves establishing the bound~\eqref{e:Lcontraction}.

Applying~(DV3) and~\eqref{e:V-bounded-level},  the following simple bound holds,
\begin{equation} 
		\label{e:LaLbComparison}
		\Lb_{m_n}   \le      \exp\bigl( V(\Phi_{m_n }) + b \bigr)   \La_{m_n} \quad a.s.,
\end{equation}
	and hence $\Lb_{m_n}   \le     \exp ( b_V(r) + b  )   \La_{m_n}    $  when $W(\Phi_{m_n })  \le r$,  for any $n$ and $r$.
	
The bound~\eqref{e:contract-4} combined with the definitions gives:
\begin{equation}
	\begin{aligned} 
			\Expect [  \La_{m_{n+1}}   \mid \clF_{m_n }   ]
			& \le 
			\bdde{t:contract-4}  \La_{m_n}  +\bdds{t:contract-4}_n( \| \theta_{m_n} \| + 1)^4   \Expect\bigl[  v(\Phi_{m_n + 1}  )  \mid \clF_{m_n }  \bigr]  + \bdd{t:contract-4}   \\
			& =    \bdde{t:contract-4}  \La_{m_n}  
			+ \bdds{t:contract-4}_n  \exp\bigl( -\epsy^\circ W(\Phi_{m_n } )    \bigr )    \Lb_{m_n} + \bdd{t:contract-4} .
	\end{aligned} 
		\label{e:gem1}
\end{equation}
By definition and the smoothing property of conditional expectation,
\[
	\Expect[ \Lb_{m_{n+1}}    \mid \clF_{m_n  }   ] \\
	=   \Expect[ (1 + \|\theta_{m_{n+1}}\|)^4\exp\bigl(V(\Phi_{m_{n+1}+1}) + \epsy^\circ W(\Phi_{m_{n+1}}) \bigr)    \mid \clF_{m_n  }   ] .
\]
Using~\Cref{t:bounded-hat-l4}~(ii) implies
that $ \Expect[ \Lb_{m_{n+1}}     \mid \clF_{m_n  }   ] $
is bounded above by:
\[
16(\|\theta_{m_n}\|+1)^4\Expect\Bigl[ \exp\bigl( V(\Phi_{m_{n+1}+1}) + \epsy^\circ W(\Phi_{m_{n+1}}) + 4\delta \sum_{k=m_n+1}^{m_{n+1}} \alpha_k W(\Phi_k)  \bigr)    \mid \clF_{m_n  }   \Bigr] .
\]
Recall that by construction $\epsy^\circ + 4\delta \sum_{k=m_n+1}^{m_{n+1}} \alpha_k \leq 1$ for each $n\ge n_g$ from~\eqref{e:block-length-epsy}. 
\Cref{t:DV3multBdd} then implies,
\[
	\begin{aligned}
		  	\Expect[ \Lb_{m_{n+1}}    \mid \clF_{m_n  }   ]  &  \le   16  (\| \theta_{m_n} \| + 1)^4    b_v^2  e^{b}
			\Expect[ \exp\bigl( V(\Phi_{m_n +1 })  \bigr)    \mid \clF_{m_n  }   ] 
			\\
			&  =   16    b_v^2  e^{b}\exp\bigl( -\epsy^\circ W(\Phi_{m_n   })  \bigr)   \Lb_{m_n}
			\\
			  	\Expect[ L_{m_{n+1}}    \mid \clF_{m_n  }   ] & \le  \bdde{t:contract-4}  \La_{m_n}    
		+  \bigl( \bdds{t:contract-4}_n  + 16  \beta       b_v^2  e^{b} \bigr)\exp\bigl( -\epsy^\circ W(\Phi_{m_n   })  \bigr)   \Lb_{m_n}  \! +  \bdd{t:contract-4},
	\end{aligned} 
\]
where the second bound follows from the first and~\eqref{e:gem1}.

To obtain the desired bound~\eqref{e:Lcontraction} we consider two cases.  If the coefficient of $ \Lb_{m_n} $ satisfies,
\[
	\bigl( \bdds{t:contract-4}_n    +  16 \beta  b_v^2  e^{b} \bigr)\exp\bigl( -\epsy^\circ W(\Phi_{m_n   })  \bigr)      \le  \bdde{t:contract-4},
\]
then we obtain the desired bound with $\varrho_0 = \bdde{t:contract-4} $ and $b_0= \bdd{t:contract-4}$. 
	In the contrary case, we   have,
\[
\exp\bigl( \epsy^\circ W(\Phi_{m_n   })  \bigr) <  \bigl( \bdds{t:contract-4}_n    +  16 \beta  b_v^2  e^{b} \bigr) /\bdde{t:contract-4} <  \bigl( \bdds{t:contract-4}_n    +  16  b_v^2  e^{b} \bigr) / \bdde{t:contract-4}.
\]
We apply~\eqref{e:LaLbComparison}, which implies that we can find a constant $b_V^*$ that is independent of $\beta$, such that $  \Lb_{m_n}   \le  b_V^*   \La_{m_n} $, giving (note that $\exp\bigl( -\epsy^\circ W(\Phi_{m_n   })  \bigr) < 1$):
\[
	\Expect[ L_{m_{n+1}}    \mid \clF_{m_n  }   ]  \le  \bdde{t:contract-4} \La_{m_n}    
	+  \bigl(  \bdds{t:contract-4}_n   +  \beta   16      b_v^2  e^{b} \bigr) b_V^* \La_{m_n} + \bdd{t:contract-4} .
\]
Choose $0 < \beta < 1$ so that $  16 \beta  b_v^2  e^{b} b_V^*< 1- \bdde{t:contract-4}$,  and define,
\[
	\varrho_0 =  \bdde{t:contract-4} +  \max_{n\ge n_0}   
\bigl\{  \bdds{t:contract-4}_n  +  16  \beta  b_v^2  e^{b} \bigr\} b_V^* .
\]
We can choose $n_0\ge 1$ large enough so that $\varrho_0<1$.   
Choosing  $b_L=b_0$ gives~\eqref{e:Lcontraction}:
\[
	\Expect[ L_{m_{n+1}}    \mid \clF_{m_n  }   ]  \le  \varrho_0  \La_{m_n}    +b_0 \le  \varrho_0  L_{m_n}    +b_0    \,,  \quad  n\ge n_0 \,, 
\]
which gives $\displaystyle
	\sup_{n\geq 0} \Expect[L_{m_n}] < \infty$.
\Cref{t:bounded-hat-l4}~(ii) combined with \Cref{t:DV3multBdd} finishes the 
proof. 
\end{proof}

\subsection{Asymptotic statistics}
\label{s:AS}


Recall that the FCLT is concerned with the continuous-time process defined 
in~\eqref{e:scerror_norm},   and the CLT  concerns the discrete-time 
stochastic process defined in~\eqref{e:scerror}.   Most of the work here concerns the FCLT, with the CLT obtained as a corollary.  
 
For the FCLT it  is convenient to introduce the notation: 
\begin{equation}
\begin{aligned}
y_k^{(n)} &\eqdef  \tilODEstate_{\SAtime_k}^{(n)}  = \theta_k - \odestate_{\SAtime_k}^{(n)}     \,,
\qquad    z_k^{(n)} \eqdef     \scerror{\SAtime_k}{n} =\frac{y_k^{(n)}}{\sqrt{\alpha_k}}    \,,    \ \ \qquad k\ge n. 
\end{aligned}
\label{e:scaled-error-def}
\end{equation}
\begin{subequations}%
Analysis is based on familiar comparisons:  
For $k\ge n$,  
\begin{align}
\odestate^{(n)}_{\SAtime_{k+1}}  & = \odestate^{(n)}_{\SAtime_k} + \alpha_{k+1}[\barf(\odestate^{(n)}_{\SAtime_k}) - \clE_k^{\clD}]
\label{e:resj-disc}
\\
\theta_{k+1} & = \theta_k + \alpha_{k+1}[\barf(\theta_k) + \Delta_{k+1}],
\label{e:SA-theta-noisedecomp}
\end{align}
where $ \clE_k^{\clD}$  denotes  the error in replacing   the integral $\int_{\SAtime_k}^{\SAtime_{k+1}} \barf(\odestate^{(n)}_t)\, dt$  by $\alpha_{k+1}  \barf(\odestate^{(n)}_{\SAtime_k})  $.
 \label{e:resj-disc-all}
 \end{subequations}

We require the following corollary to~\Cref{t:BigBounds}:
\begin{corollary}
	\label[corollary]{t:noise-bound-l4}
	Suppose the conditions of~\Cref{t:BigBounds} hold. Then, there exists a finite constant $\bdd{t:noise-bound-l4}$ such that for all $n\ge n_g$,
\begin{equation}
		\label{e:noise-bound-l4}
		\| \MD_{n+1}\|_4 \leq \bdd{t:noise-bound-l4}\,, \qquad \|\clT_n\|_4\leq \bdd{t:noise-bound-l4} \,, \qquad \| \Oops_{n+1}\|_4 \leq  \bdd{t:noise-bound-l4}.
\end{equation}
\end{corollary}

\wham{Overview of proofs of the CLT and FCLT}

 Define $w_n = \min\{k: \SAtime_k \geq \SAtime_n + T\}$.  If $n=m_{n_0}$ for some integer $n_0$, then   $w_n = m_{n+1}$ follows from~\eqref{e:mofn}. 
We obtain an update for the scaled error sequence $\{z_k^{(n)}  :  n\le k\le w_n \}$ in order to  establish tightness
of the family of stochastic processes: $\{\scerrorpull{t}{n} :    t\in[0,T] \}_{n=1}^\infty $  
defined in~\eqref{e:scerror_norm} 
 and also identify the limit. 
 
 The proof of~\Cref{t:bound-taylor-reminder} is found in~\Cref{s:AS}.

\begin{lemma}
\label[lemma]{t:bound-taylor-reminder}
Under~{\em (A2)--(A5a)} we have  the approximation,
\begin{equation}
\label{e:SA-y} 
y_{k+1}^{(n)} = y_k^{(n)} + \alpha_{k+1}[ \barA^{(n)}_k y_k^{(n)} + \clE_k^T + \clE_k^{\clD} + \Delta_{k+1} ] \,,\quad k\ge n,
 \end{equation} 
 where $\barA^{(n)}_k= A(\odestate^{(n)}_{\SAtime_k} )$, and
  the error terms are interpreted and bounded as follows:

\whamem{(i)}
    $\clE_k^T  $ is the error in the Taylor expansion:     
\begin{equation}
\clE_k^T  \eqdef  \barf(\theta_k)-
\barf(\odestate^{(n)}_{\SAtime_k})  -\barA^{(n)}_k   [  \theta_k -\odestate^{(n)}_{\SAtime_k}   ]  \,,  \qquad \|\clE_k^T \|= O(\|y_k^{(n)}\|^2 \wedge \|y_k^{(n)}\|)
\label{e:clET}
\end{equation}
 
\whamem{(ii)}
   $\| \clE_k^{\clD}\|  = O(\alpha_k \| \theta_k\| )$  
\emph{(defined below; cf.~\eqref{e:resj-disc-all})}.

\whamem{(iii)}
$\Delta_{k+1} = \MD_{k+1} - \clT_{k+1} + \clT_k - \alpha_{k+1}\Oops_{k+1}$, in which $\{\MD_k\}$ is a martingale difference sequence.  The remaining terms  
$\clT_k =   \haf(\theta_k, \Phi_{k+1}) $,  and $\alpha_{k+1}\Oops_{k+1} =   \haf(\theta_k,  \Phi_{k+2}) - \haf(\theta_{k+1}, \Phi_{k+2}) $ satisfy,
with $b_{\haf}$ a finite constant,
\[
\| \clT_k \| \le b_{\haf}  \bigl(1+V(\Phi_{k} ) \bigr)   \|   \theta_k \| \,,
	   \qquad \|\alpha_{k+1}\Oops_{k+1} \| \le b_{\haf}   \bigl(1+V(\Phi_{k+2} ) \bigr)    \| \theta_{k+1} - \theta_k \|.
\]
\end{lemma}

  \begin{subequations}

  The following companion to~\Cref{t:step-size} is required.   The proof is omitted.

Recall that 
  $\diffalpha_k$ is defined in (A1) and $ n_g$ is defined in  \eqref{e:ng}.
\begin{lemma}
\label[lemma]{t:sqrtRatioBdd}
The following bounds hold for each $k\ge 1$,
\begin{equation}
\begin{aligned}
\sqrt{\frac{\alpha_k}{\alpha_{k+1}}} & = 1 + \frac{\diffalpha_k}{2} \alpha_k  + O(\alpha_k^2)  \,,
\ \  
| \alpha_k - \alpha_{k+1} |   =   
\diffalpha_k \alpha_k^2  + O(\alpha_k^3),
\end{aligned} 
\label{e:sqrtalphas}
\end{equation} 
Moreover, with $\bdd{t:sqrtRatioBdd}=  e^{  \bardiffalpha T/2 }/T $   and every $n$ satisfying  $n\ge n_g$,
\begin{align}
\prod_{i=l}^{k-1}(1+ \diffalpha_i \alpha_i)^\half &\leq  \exp(\half \sum_{i=n}^{w_n} \diffalpha_i \alpha_i)   && 
\label{e:simpleExpBdd}
	\\
 \prod_{i=l}^{k-1} \frac{\sqrt{\alpha_{i}}}{\sqrt{\alpha_{i+1}}}    & \le 1+ \bdd{t:sqrtRatioBdd}   [\SAtime_k - \SAtime_l  ],  \quad 
 			&&  n \leq   l<  k \leq w_n. \qquad \qed
\label{e:sqrtProdExpBdd}
 \end{align}
\end{lemma}
\end{subequations}

\begin{subequations}

Dividing each side of~\eqref{e:SA-y} by $\sqrt{ \alpha_{k+1} }$ gives,
\[
z_{k+1}^{(n)}    = \sqrt{\tfrac{\alpha_k}{ \alpha_{k+1} }}z_k^{(n)} + \alpha_{k+1}[ \sqrt{\tfrac{\alpha_k}{ \alpha_{k+1} }} \barA^{(n)}_k z_k^{(n)}] + \sqrt{\alpha_{k+1}}[ \clE_k^T + \clE_k^{\clD} + \Delta_{k+1} ].
\]
This combined with~\Cref{t:sqrtRatioBdd} provides the approximation:
\begin{lemma}
\label[lemma]{t:MainZapprox}
For $k\ge n$,  
\begin{align}
	z_{k+1}^{(n)}  
	& =   z_k^{(n)} + \alpha_{k+1}[  \tfrac{\diffalpha}{2} I+  \barA^{(n)}_k ] z_k^{(n) }  + \sqrt{\alpha_{k+1}} \MD_{k+1}  
	+  \sqrt{\alpha_{k+1}}\bdds{e:SA-z}_k, 
	\label{e:SA-z}
	\\[1em]
	& \quad \text{where} \ \  \bdds{e:SA-z}_k  \eqdef  
	- \clT_{k+1} + \clT_k  + \clE_k^T + \clE_k^{\clD} -  \alpha_{k+1}\Oops_{k+1}    + \sqrt{\alpha_{k+1}} \clE_k^\alpha,
	\label{e:SA-z-errors}
\end{align}
 with   $  \clE_k^\alpha$,   satisfying $ \|  \clE_k^\alpha \| = O(\alpha_k)  \|z_k^{(n)}\|$, is the error due to the approximation of $\tfrac{\alpha_k}{ \alpha_{k+1} }$:
 \[
  \clE_k^\alpha  =  \tfrac{1}{\alpha_{k+1}} \Bigl(  \sqrt{\tfrac{\alpha_k}{ \alpha_{k+1} }}  -\{ 1 + \tfrac{\diffalpha}{2} \alpha_k  \}   \Bigr) z_k^{(n)}
  +    \Bigl(  \sqrt{\tfrac{\alpha_k}{ \alpha_{k+1} }}     -1  \Bigr)\barA^{(n)}_k z_k^{(n)}.
 \]
 \end{lemma}

  \label{e:SA-z+z-errors} 
\end{subequations}


The following summarizes the bounds obtained in this subsection.   
It will be seen that these bounds combined with standard arguments complete the proof of~\Cref{t:FCLT}.

\begin{proposition}
\label[proposition]{t:tight}
The following hold under the assumptions of~\Cref{t:FCLT}, for any $T>0$:
There exists a constant $\bdd{t:tight}<\infty$ such that:
\[
\begin{aligned}
\Expect\bigl[\|z_k^{(n)} - z_l^{(n)}\|^4 \bigr] & \leq \bdd{t:tight} |\SAtime_k - \SAtime_{l}|^2,
\\
\Expect\bigl[\|z_k^{(n)}\|^4 \bigr] &\leq \bdd{t:tight}, &&  \text{ for all }  n\ge  n_g \text{ and } n\leq l < k\leq w_n.
 \end{aligned}
\]
\end{proposition}
 
We next explain why~(A5b) combined with the conclusions of~\Cref{t:FCLT} 
and~\Cref{t:tight} imply the CLT.  
Let $\{m_n\}$ be any increasing sequence satisfying the assumptions in the convergence theory:   $\SAtime_{m_{n+1} } \ge \SAtime_{m_n} +T$ for each $n$,   and $\lim_{n\to\infty}  [ \SAtime_{m_{n+1} } - \SAtime_{m_n}] =T$.     The FCLT implies the following limit for any continuous and bounded function $g\colon\Re^d\to\Re$:
\begin{equation}
\lim_{n\to\infty}  \Expect[ g( z^{(m_n)}_{m_{n+1}}) ]   =  \Expect[g(X_T)];
\label{e:preFCLTgivesCLT}
\end{equation}
here $\sclim_T$ is Gaussian: The solution to~\eqref{e:FCLT} with initial condition $X_0=0$.  
\Cref{t:tight}  implies that we can go beyond bounded functions.  The following uniform bound holds,
\[
\max_{m_n\le k\le m_{n+1}}  \|z_k^{(m_n)} \|_4  \le b_z,
\]
where $b_z$ grows exponentially with $T$, but independent of $n$.    It follows that~\eqref{e:preFCLTgivesCLT} holds for any function  satisfying $g=o(\upomega)$, with $\upomega$ the quartic function introduced in~\Cref{t:CLT}.
The extension to unbounded functions is via uniform integrability.

The challenge then is to replace  $z^{(m_n)}_{m_{n+1}}$ by $z_n \eqdef (\theta_n - \theta^*)/  \sqrt{\alpha_n} $ in this limit,  which amounts to bounding bias:
\[
z^{(m_n)}_{m_{n+1}}  =  \frac{1}{\sqrt{\alpha_{m_{n+1}}}} \{  \theta_{m_{n+1}}  - \odestate^{(m_n)}_{\SAtime_{m_{n+1}}} \}   
=
z_{m_{n+1}}  +  \frac{1}{\sqrt{\alpha_{ m_{n+1} } }}   \{   \theta^*    - \odestate^{(m_n)}_{\SAtime_{m_{n+1}}} \}.
\] 
It is here that we require  the stability condition~\eqref{e:EAS}, which provides the     bound,
\[
\begin{aligned} 
\|   \theta^*    - \odestate^{(m_n)}_{\SAtime_{m_{n+1}}} \|
&= 
\|   \theta^* -   \phi( \SAtime_{m_{n+1}} \!\! - \SAtime_{m_n} ;   \theta_{m_n}) \|     \le  \sqrt{\alpha_{ m_{n} } }   \bexp e^{ -\rhoexp T } \| z_{m_n} \|,
\end{aligned} 
\]
and hence from the prior identity,
\begin{equation}
\begin{aligned} 
\| z_{m_{n+1}}   \|  &  \le  \|  z^{(m_n)}_{m_{n+1}} \|    +  \xi_{m_n}(T)   \|  z_{m_n}  \| 
\\
\| z_{m_{n+1}}   -  z^{(m_n)}_{m_{n+1}} \|   &  \le     \xi_{m_n}(T)   \|  z_{m_n}  \|   \,,
\qquad  
	 \text{where}
\quad
\xi_{m_n}(T) =   \sqrt{ \tfrac{\alpha_{ m_n} }{  \alpha_{ m_{n+1} }}} \bexp e^{ -\rhoexp T }.
\end{aligned} 
\label{e:c-contract}
\end{equation}
It will be seen that $\xi_{m_n}(T) $ is vanishing with $T$, uniformly in $n$.    

It is convenient here to flip the window backwards in time:   
For each $k\ge w_1$ denote $w^-_k = \max\{ n  :  w_n \le k \}$,   and  
let,
\begin{equation}
\xi^-_k(T) =  \Bigl( \sqrt{  \alpha_{ w_k^-}  /  \alpha_k }  \Bigr)  \bexp e^{ -\rhoexp T } .
\label{e:xi-k}
\end{equation} 
\Cref{t:EAScontraction} tells us  that $\{\xi^-_k(T)  \}$ vanishes with $T$,  uniformly in $k$, 
from which we obtain the following:

\begin{lemma}
\label[lemma]{t:zCLTbdd} 
Suppose that the conclusions of~\Cref{t:FCLT} and~\Cref{t:tight} hold.   If in addition~(A5b) is satisfied, then:
\whamem{(i)}
 $\displaystyle \lim_{k\to\infty}  \Expect[ g( z^{(w^-_k)}_k ) ]   =  \Expect[g(X_T)]$
  for any continuous function satisfying $g=o( \upomega)$.

\whamem{(ii)}
  The error $L_4$ norm    $  \|  z_k   \|_4 $  is uniformly bounded in $k$, and moreover,
\[
  \limsup_{k\to\infty}  \|  z_k  - z^{(w^-_k)}_k\|_4   \le \bdds{t:zCLTbdd}_T \to 0\,,\quad \text{as $T\to\infty$.}
\]
\end{lemma}

\begin{proof}[Proof of~\Cref{t:CLT}]
The proof is presented here with the understanding 
that~\Cref{t:FCLT}, \Cref{t:tight} and~\Cref{t:zCLTbdd} have been established.  
Full details of the proofs are found below. 
 
We first establish~\eqref{e:CLT} for bounded and continuous functions, and for this it is sufficient to restrict to characteristic functions.     That is, it is sufficient to establish  the family of limits,
\[
\lim_{k\to\infty} \Expect[ \phi_v( z_k ) ]  =\Expect[\phi_v(\sclim)]  \,, \qquad v\in\Re^d,
\]
where $\phi_v(z) = \exp(j v^\transpose z)$ for $z\in\Re^d$ and $j=\sqrt{-1}$~\cite[Th.~29.4]{bil95}.
Letting $L_{\phi_v}$ denote a Lipschitz constant for $\phi_v$,~\Cref{t:zCLTbdd}~(ii) and Jensen's inequality give,
\[
 \limsup_{k\to\infty}  | \Expect[\phi_v(z_k)]  -  \Expect[ \phi_v( z^{(w^-_k)}_k ) ]  | 
 \le L_{\phi_v}  \limsup_{k\to\infty}   \Expect [ \| z_k -  z^{(w^-_k)}_k \| ]  
 \le L_{\phi_v} \{    \bdds{t:zCLTbdd}_T \}^{1/4}.
\]
Combining this with~\Cref{t:zCLTbdd}~(i) gives,
 \[
\limsup_{k\to\infty}  | \Expect[\phi_v(z_k)]  -  \Expect[\phi_v(X)] |  
	 \le  L_{\phi_v} \{    \bdds{t:zCLTbdd}_T \}^{1/4}   +  |  \Expect[\phi_v(X_T)] -  \Expect[\phi_v(X)] |.
\]
The right-hand side vanishes as $T\to\infty$ for any $v$, from which we conclude that~\eqref{e:CLT} holds for bounded and continuous functions.    

\Cref{t:tight}
and~\Cref{t:zCLTbdd}   imply that $\{ g(z_k)\}$ are uniformly integrable when  $g=o(\upomega)$,  which justifies extension of~\eqref{e:CLT} to  unbounded and 
continuous functions satisfying this bound.  
 \end{proof}

\wham{Tightness and the FCLT}   

We begin with a proof of tightness of the distributions of the two families of stochastic processes $\{ \scerrorpull{t}{n}\,, t\in[0,T]\}_{n=1}^\infty$ and $\{ \odestatepull{t}{n} \,, t\in[0,T]\}_{n=1}^\infty$, along with $L_p$ bounds required for refinements of the FCLT.

Using a well-known criterion for tightness of probability measure over $C([0, T];\Re^d)$ in~\cite[Th.~12.3]{bil68}, we establish a form   of uniform continuity: There exists a constant $\bdd{t:tight}<\infty$ such that for all $n\ge n_g$ and $n \leq   l<  k \leq w_n$,
\begin{subequations}
\label{e:conti-zvt}
\begin{align}
\Expect\bigl[\|\odestate_{\SAtime_k}^{(n)} - \odestate_{\SAtime_{l}}^{(n)}\|^4 \bigr] \leq \bdd{t:tight} |\SAtime_k - \SAtime_{l}|^2,  \label{e:conti-vt} \\
\Expect\bigl[\| \scerrorpull{\SAtime_k}{n} - \scerrorpull{\SAtime_l}{n}\|^4 \bigr] \leq \bdd{t:tight} |\SAtime_k - \SAtime_{l}|^2.
 \label{e:conti-zt} 
\end{align}
\end{subequations}
The bound~\eqref{e:conti-vt} follows directly from~\Cref{t:BigBounds}.   
The bound~\eqref{e:conti-zt} is the major part of~\Cref{t:tight} that is the main outcome of this subsection, whose proof requires the proof 
of~\Cref{t:bound-taylor-reminder} and another  simple lemma.

 \begin{proof}[Proof of~\Cref{t:bound-taylor-reminder}]
Part~(ii)  follows from Lipschitz continuity of $\barf$,   
and part~(iii) follows from~\Cref{t:noise-decomp} and~\eqref{e:bdds+lip-H} 
of~\Cref{t:bounds-H}.

For part~(i)  we begin with the definition~\eqref{e:clET}, written in the form,
\[
\clE_k^T = R( \odestate^{(n)}_{\SAtime_k}  ,  \theta_k )  \,,  \qquad R(\theta^\circ, \theta) = \barf(\theta) - \barf(\theta^\circ) - A(\theta^\circ)(\theta - \theta^\circ)  \quad \text{ for $\theta,\theta^\circ\in\Re^d$,}
\]
and
with $A(\theta^\circ) =  \partial \barf\,  ( \theta^\circ )$. 
The bound $\|R(\theta^\circ, \theta) \| = O(\|\theta^\circ- \theta\|^2)$ follows from~(A5a).  
Also, $R(\theta^\circ, \theta)$  is Lipschitz continuous in $\theta$  with  Lipschitz constant $2\barL$, giving,
\[
	\|R(\theta^\circ, \theta) \| \leq b_0 \min(\|\theta^\circ- \theta\|, \|\theta^\circ- \theta\|^2).
\] 
The desired bound in~(i) then follows from the definition $\clE_k^T = R( \odestate^{(n)}_{\SAtime_k}  ,  \theta_k ) $ and the definition   $y_k^{(n)} = \theta_k - \odestate_{\SAtime_k}^{(n)}   $.
\end{proof}
  
Recalling the notation $\barA^{(n)}_k= A(\odestate^{(n)}_{\SAtime_k} )$ introduced before   \eqref{e:SA-y},   consider the collection of $d\times d$ matrices $\{ \STy_{ l, k}  :  0\le l\le k\}$ defined by induction:    for each $l$,  
\[
\STy_{ l, l} = I \,, \ \ \textit{and} \quad
  \STy_{ l, k} =    [I+\alpha_{k}\barA^{(n)}_{k-1}]   \STy_{ l, k-1}\,, \quad k>l 
\]
We also consider the scaled matrices $\{  \ScaledST_{ l, k} \eqdef  \sqrt{\alpha_l/\alpha_k}  \STy_{ l, k}  :  0\le l\le k \}$.
These are written informally as  
 \begin{equation} 
\begin{aligned}
\STy_{ l, k}   &  \eqdef 	\prod_{i=l}^{k-1}[I+\alpha_{i+1}\barA^{(n)}_i] 
  \,,  \ \  &&   0\le  l < k\,,    \quad 
						&&	\STy_{ l, l}  \eqdef 	I  \, , 
\\
\ScaledST_{ l, k}  &  \eqdef 	\prod_{i=l}^{k-1}[I+\alpha_{i+1}\barA^{(n)}_i]  \frac{\sqrt{\alpha_{i}}}{\sqrt{\alpha_{i+1}}} \,,  \ \ 
&&  0\le  l < k\,,    \quad 
						&&	\ScaledST_{ l, l}  \eqdef 	I  \,  .
\end{aligned}
\label{e:STy}
\end{equation} 
With $n>0$ and $l \geq n$ fixed, iterating~\eqref{e:SA-y} for $k > l$,  gives,
\begin{equation}
y_k^{(n)} =   \STy_{l,k} y_l^{(n)}  + \sum_{j=l+1}^{k}  \alpha_{j}  
		\STy_{j,k}  [\clE_{j-1}^T + \clE_{j-1}^{\clD} +  \Delta_{j}].
\label{e:yAlmostLinear}
\end{equation}
We consider next a similar recursion for $\{ z_{k}^{(n)}  \}$ obtained by iterating~\eqref{e:SA-z}.

\begin{lemma}
	\label[lemma]{t:z-formula}
	For each $n$, and $n\leq l < k \leq w_n$,
\begin{equation}
		\label{e:zn-acc-sum}
	\begin{aligned}
			z_k^{(n)} = \ScaledST_{l, k}  z_{l}^{(n)}  &+  \sum_{j=l+1}^{k}\sqrt{\alpha_{j}} 
					 \ScaledST_{j, k} [\clE_{j-1}^T  + \clE_{j-1}^{\clD} +   \Delta_{j}] \, .
\end{aligned}
\end{equation} 
Moreover, there exists a deterministic constant $b_{\ScaledST}<\infty$ such that for all $n>0$,
\whamem{(i)} $\displaystyle
	\max_{n< l < k < w_n} \{  \| \STy_{l, k} \| +  \| \ScaledST_{l, k} \| \} \leq b_{\ScaledST} $,   \ \ 
	\textbf{\emph{(ii)}}
  $\|\ScaledST_{l, k} -I\| \leq b_\ScaledST (\SAtime_k - \SAtime_{l})$.
\end{lemma}

\begin{proof}
Dividing both sides of~\eqref{e:yAlmostLinear}
 by $\sqrt{\alpha_k}$ gives \eqref{e:zn-acc-sum} since 
 $
 \ScaledST_{j, k}    =    \sqrt{\alpha_l/\alpha_k}    \STy_{j, k}  
 $.

\begin{subequations}

Since $\|A(\theta)\|\leq \barL$ for all $\theta\in\Re^d$, we apply~\Cref{t:sqrtRatioBdd} to obtain the pair of bounds:
\begin{align}
\|   \STy_{j, k}  \|
		& \leq \prod_{i=j+1}^{k}(1+\alpha_{i}\barL) \leq \exp(\barL T),
   \\
   	\prod_{i= l}^{k-1}\frac{\sqrt{\alpha_{i}}}{\sqrt{\alpha_{i+1}}}         &  \le 1 + b  [\SAtime_k - \SAtime_l  ]  \le  1+ b  T 
            \,, \qquad n\leq l < k \leq w_n.  
        \label{e:prod-sqrt-alpha-rats}
\end{align} 
 Therefore, $\STy_{j, k} ,\ScaledST_{j, k} $ are uniformly bounded for all $n\leq j <k < w_n$,  and all $n$.
This establishes~(i),  and  writing 
$\STy_{l, k} -I  =    [I+\alpha_{k}\barA^{(n)}_{k-1}]  [  \STy_{ l, k-1}  - I ]  +   \alpha_{k}\barA^{(n)}_{k-1}$ brings us 
 one step towards~(ii):   For a fixed constant $b_y$,
\begin{equation}
\|\STy_{l, k} -I\| \leq b_y (\SAtime_k - \SAtime_{l})  \,,\quad k > l.
\label{e:StyBdd}
\end{equation}

\label{e:RecallLemmaProdBdds}
\end{subequations}

The remainder of the proof of~(ii) uses the identity,  
\[
\begin{aligned}
		\ScaledST_{l, k}   
		&= [ \ScaledST_{l, k}  -\STy_{l,k} ] + \STy_{l,k}
		= \bigl[\prod_{i=l}^{k-1} \frac{\sqrt{\alpha_{i}}}{\sqrt{\alpha_{i+1}}}-1\bigr]
		\STy_{l,k}+ \STy_{l,k}  \, .
\end{aligned}
\]
Applying~\eqref{e:prod-sqrt-alpha-rats} and \eqref{e:StyBdd} establishes~(ii):         
\[
\begin{aligned}
\| \ScaledST_{l, k}   - I \|  
&\le  \| \bigl[\prod_{i=l}^{k-1} \frac{\sqrt{\alpha_{i}}}{\sqrt{\alpha_{i+1}}}-1\bigr]  \|  
		\|\STy_{l,k}\| + \| I-\STy_{l,k}\|
		\\
	&\le b   [\SAtime_k - \SAtime_l  ]  \exp(\barL T)   +   \| I-\STy_{l,k}\| 
	\\
	&  \le   b_\ScaledST [\SAtime_k - \SAtime_l  ]   \,,\qquad \textit{with} \ \ b_\ScaledST \eqdef b \exp(\barL T)   + b_y 
\end{aligned}
\]
\end{proof}

Next we establish bound for each disturbance term on the right-hand side of~\eqref{e:zn-acc-sum}. 
The bound~(i) in~\Cref{t:tight-noise-l4} is far from tight.   Once we obtain $L_4$ bounds on $\{z_k^{(n)}  :  n\le k\le w_n \}$,  we can expect that the left-hand side of~(i) will vanish as $n\to\infty$.   Our immediate goal   is only to establish   $L_4$ bounds.

\begin{lemma}
\label[lemma]{t:tight-noise-l4}
The following hold under the assumptions of~\Cref{t:FCLT}:   There exist a constant $\bdd{t:tight-noise-l4}<\infty$ 
and a vanishing sequence $\{  \bdds{t:tight-noise-l4}_n : n\ge 1\}$
such that for all $n>0$ and $n\leq l < k \leq w_n$:
\whamem{(i)} 
$\displaystyle  \bigl\| \sum_{j=l+1}^{k}  \sqrt{\alpha_{j}} \,  \ScaledST_{j,k} \,  \clE_{j-1}^T \bigr\|_4 \leq \bdd{t:tight-noise-l4} \sum_{j=l}^{k-1} \alpha_{j} \|z_j^{(n)}\|_4$.

\whamem{(ii)} $ \displaystyle  \bigl\|  \sum_{j=l+1}^{k}  \sqrt{\alpha_{j}} \,  \ScaledST_{j,k} \,  [\clE_{j-1}^{\clD}  - \alpha_{j}\Oops_{j}] \bigr\|_4 \leq \bdds{t:tight-noise-l4}_n \bigl(\SAtime_k - \SAtime_l \bigr)^\half$.

\whamem{(iii)} $\displaystyle  \bigl\|
		\sum_{j=l+1}^{k}  \sqrt{\alpha_{j}} \,  \ScaledST_{j,k} \,   [\clT_{j}-\clT_{j-1}]
		\bigr\|_4 \leq    \min\bigl \{  \bdds{t:tight-noise-l4}_n  ,   \bdd{t:tight-noise-l4}  \bigl(\SAtime_k - \SAtime_l \bigr)^\half \bigr\}$.

\whamem{(iv)} $\displaystyle \bigl\|
\sum_{j=l+1}^{k}  \sqrt{\alpha_{j}} \,  \ScaledST_{j,k} \,  
 \MD_{j} \bigr\|_4 \leq \bdd{t:tight-noise-l4} \bigl(\SAtime_k - \SAtime_l \bigr)^\half $.
 \end{lemma}

The following bounds are required in~(i)--(iii): For a finite constant $\bdd{e:bounds-O}$:
\begin{equation}
\begin{aligned}
\| \clE_j^T \|_4, &\le  \bdd{e:bounds-O} \sqrt{\alpha_j}  \|z_j^{(n)}\|_4    & \qquad		\|\clE_j^{\clD}\|_4  &\leq \alpha_j \bdd{e:bounds-O},
		\\
			\| \Oops_{j}\|_4& \leq \bdd{e:bounds-O},  \quad & \|\clT_{j}\|_4& \leq \bdd{e:bounds-O}.
\end{aligned} 
\label{e:bounds-O}
\end{equation}
The first bound follows from $\clE_j^T = O(\|y_j^{(n)}\|^2 \wedge \|y_j^{(n)}\|)  \le  O(\sqrt{\alpha_j} \|z_j^{(n)}\|)$.
The remaining bounds follow from~\Cref{t:bound-taylor-reminder} combined 
with~\Cref{t:BigBounds} and~\Cref{t:noise-bound-l4}.

\begin{proof}[Proof of~\Cref{t:tight-noise-l4}]
Applying \Cref{t:z-formula}   and \eqref{e:bounds-O},
\[
	\bigl\| \sum_{j=l+1}^{k}\sqrt{\alpha_{j}}\ScaledST_{j, k} \clE_{j-1}^T \bigr\|
		 \leq \bdd{e:bounds-O}b_{\ScaledST}  \sum_{j=l+1}^{k}
\sqrt{\alpha_{j}   \alpha_{j-1}}\|z_{j-1}^{(n)}\|.
\]
The bound in~(i)  then follows.    The proof of (ii) is similar.  
 	
	Applying summation by parts to the objective in~(iii) gives,
\[
\begin{aligned}
&\sum_{j=l+1}^{k}  \sqrt{\alpha_{j}}  \ScaledST_{j, k} [\clT_{j}-\clT_{j-1}] 
\\
	&\	 = \sqrt{\alpha_k} \clT_k - \sqrt{\alpha_{l+1}} \ScaledST_{l+1, k}  \clT_{l} 
			+ \sum_{j=l+1}^{k}\sqrt{\alpha_{j}}\alpha_{j-1}\Bigl\{
		\diffalpha_{j-1} I +A(\odestate^{(n)}_{\SAtime_{j-1}})\Bigr\}\ScaledST_{j, k} \clT_{j},
\end{aligned}
\]
which leads to the bound,
\begin{align*}
&\Bigl\| \sum_{j=l+1}^{k}    \sqrt{\alpha_j} \ScaledST_{j, k} [\clT_{j}-\clT_{j-1}] \Bigr\|\\
&\leq
\sqrt{  \alpha_k} \|\clT_k\| 
	+ \sqrt{\alpha_{l+1}} \|\ScaledST_{l+1, k}  
	\clT_{l}\|  
	+ \sum_{j=l+1}^{k}  \sqrt{ \alpha_j } \alpha_{j-1} \Bigl\|\bigl\{
	\diffalpha_j I +A(\odestate^{(n)}_{\SAtime_j})\bigr\}
	\ScaledST_{j, k} \clT_{j}\Bigr\| .
\end{align*}
Using~\eqref{e:bounds-O} and the inequality $(a+b+c)^4\leq 16 \{ a^4 + b^4+c^4\}$, we may increase the constant $\bdd{e:bounds-O}$   to obtain,
\begin{align*}
&\Expect\Bigl[\Bigl\|     \sum_{j=l+1}^{k}     
\sqrt{\alpha_{j}}\ScaledST_{j, k}   
[\clT_{j}- \clT_{j-1}] \Bigr\|^4\Bigr]\\
&\leq
 \bdd{e:bounds-O}\alpha_k^2 \Expect[\|\clT_k\|^4] + \bdd{e:bounds-O}\alpha_{l+1}^2 \Expect[\|\clT_{l}\|^4] 	+\bdd{e:bounds-O} \Expect\Bigr[ \Bigr( \sum_{j=l+1}^{k}\sqrt{\alpha_{j+1}}\alpha_j \bigl\| \clT_{j}\bigr\|  \Bigr)^4 \Bigr].
\end{align*}
The first two terms on the right-hand side are bounded as follows:
\begin{equation*} 
		\alpha_k^2 \Expect[\|\clT_k\|^4] + \alpha_{l+1}^2 \Expect[\| \clT_{l}\|^4]  \leq \bdd{e:bounds-O}^4 \sum_{j=l+1}^{k} \alpha_{j+1}^2 \leq \bdd{e:bounds-O}^4 \bigl(\sum_{j=l+1}^{k} \alpha_{j+1}\bigr)^2.
\end{equation*} 
The right-hand side vanishes with $n$, as it is bounded by a constant times    $(\SAtime_k - \SAtime_l )^2$.
Next,
\[
\begin{aligned}
		\Expect\Bigr[ \Bigr( \sum_{j=l+1}^{k}\sqrt{\alpha_{j+1}}\alpha_j \bigl\|\clT_{j}\bigr\| \Bigr)^4 \Bigr]
		& =  \bigl(\sum_{j=l+1}^{k}\alpha_j\bigr)^4 \Expect\Bigr[\Bigr(\frac{1}{\sum_{j=l+1}^{k}\alpha_j}  \sum_{j=l+1}^{k}\sqrt{\alpha_{j+1}}\alpha_j  \|\clT_{j}\| \Bigr)^4 \Bigr] \\
		&  \leq  \bigl(\sum_{j=l+1}^{k}\alpha_j\bigr)^3 \sum_{j=l+1}^{k}\alpha_{j+1}^2\alpha_j \Expect\bigl[\|\clT_{j}\|^4\bigr] \\
		&  \leq  T^3 \bdd{e:bounds-O}^4\sum_{j=l+1}^{k}\alpha_{j+1}^2\alpha_j   \le   \bigl(\SAtime_k - \SAtime_l \bigr)^2\clE_n,
\end{aligned}
\]
where the first inequality follows by Jensen's inequality,   and the final inequality holds with $\displaystyle \clE_n \eqdef  T^3 \bdd{e:bounds-O}^4 \max \{ \alpha_{j+1}^2/ \alpha_j   :  j\ge n \}$, which vanishes as $n\to\infty$.
  Combining these bounds   gives~(iii). 
  
For each $n$, $l \geq n$ define $\Gamma_{l,l}^{(n)} = 0$
 and   $ 
	\Gamma_{l,k}^{(n)} \eqdef \sum_{j=l+1}^k\sqrt{\alpha_{j}}\ScaledST_{j, k} \MD_{j}$
	 for $k > l$.   
We have the recursive representation,
\[
\begin{aligned}
		\Gamma_{l,k}^{(n)}
		& =  \sqrt{ \tfrac{\alpha_{k-1}}{\alpha_{k}} } [I+\alpha_k\barA^{(n)}_{k-1}]\Gamma_{l,k-1}^{(n)}  + \sqrt{\alpha_k} \MD_k \\
		& =\Gamma_{l,k-1}^{(n)} +  \bigl[( \sqrt{\tfrac{\alpha_{k-1}}{\alpha_{k}} } - 1)I+\sqrt{\alpha_{k-1}}\sqrt{\alpha_k}\barA^{(n)}_{k-1}\bigr]\Gamma_{l,k-1}^{(n)}  + \sqrt{\alpha_k} \MD_k \,,
\end{aligned}
\]
and summing each side then gives,
\[
\begin{aligned}
		\Gamma_{l,k}^{(n)} 
		 = \sum_{j=l+1}^{k}\bigl[ (  \sqrt{ \tfrac{\alpha_{j-1}}{\alpha_{j} }} - 1)I+\sqrt{\alpha_j\alpha_{j-1}}A(\odestate^{(n)}_{\SAtime_j})\bigr]\Gamma_{l,j}^{(n)} 
			+  \sum_{j=l+1}^{k} \sqrt{\alpha_{j}} \MD_{j} .
\end{aligned}
\]
Applying~\Cref{t:sqrtRatioBdd}
once more we obtain, with a possibly larger constant $\bdd{e:bounds-O}<\infty$,
\[
	\bigl\| ( \sqrt{\tfrac{\alpha_{j-1}}{\alpha_{j}} } - 1)I+\sqrt{\alpha_j}\sqrt{\alpha_{j+1}}A(\odestate^{(n)}_{\SAtime_j})\bigr\| \leq \bdd{e:bounds-O}\alpha_{j}  \,, \qquad j>0.
\]
Then, by the triangle inequality,
\[
\begin{aligned}
		\| \Gamma_{l,k}^{(n)} \|_4  
		&\leq \bdd{e:bounds-O} \sum_{j=l+1}^{k} \alpha_{j} \| \Gamma_{l,j}^{(n)}\|_4 
			+ 
			b_{l,k}\,,  \qquad b_{l,k}  \eqdef 
			 \bigl\|	\sum_{j=l+1}^{k} \sqrt{\alpha_{j}} \MD_{j} \bigr\|_4 .
\end{aligned}
\]
Denote   $\barb_{l,k} = \max_{ l < j \le k} b_{l,j}$. 
 The discrete Gronwall's inequality then gives, 
\[
\| \Gamma_{l,k}^{(n)} \|_4  \leq 
\exp(\bdd{e:bounds-O}\sum_{j=l+1}^{k}\alpha_{j}) \barb_{l,k}    \le  \exp(\bdd{e:bounds-O} T)  \barb_{l,k}  \,,
\quad n\leq l < k\leq w_n \,,
\]
and we obtain via \Cref{t:bound-md},
\begin{align*}
\barb_{l,k} ^4\le 
\Expect\bigl[ \max_{l \leq j \leq k} \bigl\| \sum_{i = l}^{j} 
\sqrt{\alpha_{i}} \MD_{i} \bigl\|^4  \bigr] 
		 & \leq \bdd{t:bound-md} 
\max_{j} \Expect[\|\MD_{j}\|^4 ]\bigl(\sum_{i = l+1}^{k} 
	\alpha_{i}\bigr)^2
	  \leq \bdd{t:bound-md} \max_{j} \Expect[\|\MD_{j}\|^4 ] \bigl(\SAtime_k - \SAtime_l \bigr)^2.
\end{align*}
This and~\Cref{t:noise-bound-l4}  establishes
$ \barb_{l,k}  \le  \bdd{t:noise-bound-l4}   \bdd{t:bound-md}^{1/4}  \bigl(\SAtime_k - \SAtime_l \bigr)^\half $, and concludes the proof of~(iv). 
 \end{proof}
 
\begin{proof}[Proof of~\Cref{t:tight}]
Applying the triangle inequality to~\eqref{e:zn-acc-sum} in~\Cref{t:z-formula},

\[
	\begin{aligned}
		\|z_k^{(n)}\|_4  \le   &  \big\|   \sum_{j=n+1}^{k}\sqrt{\alpha_{j}}  \ScaledST_{j, k} \clE_{j-1}^T  \|_4   
		+
		  \big\|   \sum_{j=n+1}^{k}\sqrt{\alpha_{j}}  \ScaledST_{j, k} \MD_{j}^T  \|_4   
		\\
	&	+  \big\|   \sum_{j=n+1}^{k}\sqrt{\alpha_{j}}  
					 \ScaledST_{j, k} [  \clE_{j-1}^{\clD}  -\alpha_j \Oops_j -    \clT_{j-1}  +   \clT_{j}  ]   \big\| 
					  \,, \qquad n<k \leq w_n \, .
\end{aligned}
\]
Applying~\Cref{t:tight-noise-l4} to the above inequality implies that for $n$ sufficiently large,
\[
\|z_k^{(n)}\|_4  \leq   \bdd{t:tight-noise-l4} \sum_{j=n}^{k-1} \alpha_{j} \|z_j^{(n)}\|_4 
	+   3\bdd{t:tight-noise-l4}[\SAtime_k - \SAtime_n]^\half \leq \bdd{t:tight-noise-l4} \sum_{j=n}^{k} \alpha_{j+1} \|z_j^{(n)}\|_4 
	+  3\bdd{t:tight-noise-l4} T^\half .
\]
The discrete Gronwall's inequality gives: 
\begin{equation}
		\label{e:z-bounded-l4}
		\sup_n \sup_{n\leq k\leq w_n} \|z_k^{(n)}\|_4 <\infty.
\end{equation}
An application of the triangle inequality to~\eqref{e:zn-acc-sum}, for arbitrary $n\leq l < k \leq w_n$, gives,
\[
\begin{aligned}
		\| z_k^{(n)}  - z_{l}^{(n)}\|_4
		  \leq    \| & \ScaledST_{l, k}  -I\|       \|z_{l}^{(n)}\| _4
 +  \Bigl\| \sum_{j=l+1}^{k}\sqrt{\alpha_{j}}\ScaledST_{j, k} \clE_{j-1}^T \Bigr\|_4 + \Bigl\| \sum_{j=l+1}^{k}\sqrt{\alpha_{j}}\ScaledST_{j, k} \MD_{j} \Bigr\|_4 \\
		&  + \Bigl\|  \sum_{j=l+1}^{k}\sqrt{\alpha_{j}}\ScaledST_{j, k} [\clE_{j-1}^{\clD}  - \alpha_{j}\Oops_{j} - \clT_{j} + \clT_{j-1}] \Bigr\|_4.
\end{aligned}
\]
The proof is completed on combining this with~\eqref{e:z-bounded-l4},~\Cref{t:z-formula}~(ii) and~\Cref{t:tight-noise-l4}.
\end{proof}

Having established tightness of the sequence of stochastic processes $\{  \scerrorpull{\varble}{n}   :   n\ge 1\}$, the next step in the proof of the FCLT is to characterize any sub-sequential limit.  The following variant of~\Cref{t:tight-noise-l4} will be used in this step.

\begin{lemma}
\label[lemma]{t:tight-noise-l2}
The following hold under the assumptions of~\Cref{t:FCLT}:   
There exists a vanishing sequence $\{  \bdds{t:tight-noise-l2}_n : n\ge 1\}$
such that for all $n>0$ and $n\leq l < k \leq w_n$,
\whamem{(i)}
 $\displaystyle  
\Expect\Bigl[  
	\max_{l \le \ell < k}
\bigl\| \sum_{j=l}^{\ell}
\sqrt{\alpha_{j+1}}   \clE_j^T \bigr\|^2 \Bigr] \leq \bdds{t:tight-noise-l2}_n$.
\quad
 \textbf{\emph{(ii)}} $ \displaystyle 
\Expect\Bigl[  
	\max_{l \le \ell < k}
\bigl\| \sum_{j=l}^{\ell}
\sqrt{\alpha_{j+1}}  [\clE_j^{\clD}  - \alpha_{j+1}\Oops_{j+1}]  \bigr\|^2 \Bigr]  
			 \leq \bdds{t:tight-noise-l2}_n  $.

\whamem{(iii)} $ \displaystyle 
\Expect\Bigl[  
	\max_{l \le \ell < k}
\bigl\| \sum_{j=l}^{\ell}
\sqrt{\alpha_{j+1}}  [\clT_{j+1}-\clT_{j}] \bigr\|^2 \Bigr]  
			 \leq \bdds{t:tight-noise-l2}_n  $.
 \end{lemma}

\begin{proof}
The proofs of~(ii) and~(iii) are identical to the corresponding bounds in~\Cref{t:tight-noise-l4}.
Part~(i) requires the tighter bound
$\clE_j^T = O(\|y_j^{(n)}\|^2 \wedge \|y_j^{(n)}\|) $, so that,
$$\|\clE_j^T  \|^2  \le b_0^2
 \min\{  \alpha_j^2 \| z_j^{(n)}\|^4, \alpha_j  \| z_j^{(n)}\|^2\},$$
for some constant $b_0$,  and also,
\[
 \Expect[  \|\clE_j^T  \|^2 ]  \le b_0^2
 \min\bigl\{  \alpha_j^2 \Expect[   \| z_j^{(n)}\|^4]  , \alpha_j  \Expect[   \| z_j^{(n)}\|^2] \bigr\} .
\]
Applying the triangle inequality in $L_2$ gives,
\[
\begin{aligned}
\Expect\Bigl[  
	\max_{l \le \ell < k}
\bigl\| \sum_{j=l}^{\ell}
\sqrt{\alpha_{j+1}}  \clE_j^T  \bigr\|^2 \Bigr]^{\half}   
            & \le 
\sum_{j=n}^{w_n-1}
\sqrt{\alpha_{j+1}} \Expect\bigl[  
\bigl\|  \clE_j^T  \bigr\|^2 \bigr]^{\half}  
\\
&\le 
b_0
\sum_{j=n}^{w_n-1}
\sqrt{\alpha_{j+1}}
 \min\{  \alpha_j \| z_j^{(n)}\|^2_4, \sqrt{ \alpha_j}  \| z_j^{(n)}\|_2\} .
\end{aligned} 
\]
Choose $n_0\ge 1$ so that 
$  \alpha_j^2 \Expect [ \| z_j^{(n)}\|^4] \le  \alpha_j  \Expect[ \| z_j^{(n)}\|^2  ] $ for $j\ge n_0$.   
 The preceding then gives, for all $n\ge n_0$, $ l \ge n$ and $l\le k \le w_n$,
\[
\begin{aligned}
\Expect\Bigl[  
	\max_{l \le \ell < k}
\bigl\| \sum_{j=l}^{\ell}
\sqrt{\alpha_{j+1}}  \clE_j^T  \bigr\|^2 \Bigr]^{\half}   
&\le 
  b_0
\sum_{j=n}^{w_n-1}
\sqrt{\alpha_{j+1}} 
  \alpha_j 
 \| z_j^{(n)}\|_4^2    
 \le b_0 T \{ \max_{j\ge n} \sqrt{\alpha_{j+1}}  \| z_j^{(n)}\|_4^2 .
 \} 
\end{aligned} 
\]
The right-hand side vanishes as $n\to\infty$.  
\end{proof}

We next place the problem in the setting of~\cite[Ch.~7]{ethkur05},  which requires a particular decomposition.   As motivation, first write~\eqref{e:FCLT}  (with $F = \tfrac{\diffalpha}{2} I+A^*$) as,     
\[
 \sclim_{\SAtime} =   \intDrift_\SAtime  +   \Mart_\SAtime\,,   \qquad 0\le \SAtime\le T \,,\quad \text{where} \ \   \intDrift_\SAtime   
 					= \int_0^\SAtime  [\tfrac{\diffalpha}{2} I+A^*]  \sclim_t \,dt   \,,
 \]
 and $\{ \Mart_\SAtime \}$ is Brownian motion.  The required decomposition of $ \scerrorpull{\SAtime}{n}$ (defined in \eqref{e:scerror_norm}) is,
\begin{equation}
 \scerrorpull{\SAtime}{n} =   \scerrorpullDintDrift{\SAtime}{n} + \scerrorpullMart{\SAtime}{n}   \,,\qquad 0\le \SAtime\le T\,, \  n\ge 1,
\label{e:ethkur-decomposition}
\end{equation}
in which $\{\scerrorpullMart{\varble}{n} \}$ is a martingale for each $n$,   and the first term approximates $ \int_0^\SAtime  [\tfrac{\diffalpha}{2} I+A^*]   \scerrorpull{t}{n}  \,dt  $. 
The representation~\eqref{e:SA-z} suggests the choice of martingale,
\[
 \scerrorpullMart{\SAtime}{n}  =   \sum_{j=n+1}^{k}   \sqrt{\alpha_{j}}\MD_{j}   \,, \qquad
\SAtime_k \le \SAtime  <  \SAtime_{k+1} \,,
\]
valid for $\SAtime\ge \SAtime_{n+2}$,  with $ \scerrorpullMart{\SAtime}{n} \eqdef 0$ for  $\SAtime< \SAtime_{n+2}$.
Its covariance   is denoted $\Sigma^{(n)}_\SAtime \eqdef  \Expect[  \scerrorpullMart{\SAtime}{n} (\scerrorpullMart{\SAtime}{n})^\transpose]$.
The representation~\eqref{e:SA-z+z-errors} combined 
with~\Cref{t:tight-noise-l2} then gives us~\eqref{e:ethkur-decomposition} 
with $    \scerrorpullDintDrift{\SAtime}{n} \eqdef  
 \scerrorpull{\SAtime}{n} - \scerrorpullMart{\SAtime}{n}  $,
and a useful bound:
\begin{equation}
\begin{aligned}
	\scerrorpullDintDrift{\SAtime}{n}  
	&=   \int_0^{\SAtime }[\tfrac{\diffalpha}{2} I   + A ( \odestatepull{t}{n} )]\scerrorpull{t}{n}    \, dt 
	 + \clE_{\SAtime}^{\scerrorSymbol}(n)   	 	 \,, \qquad
\SAtime_k \le \SAtime  <  \SAtime_{k+1},
\\
\text{in which }\qquad
& \lim_{n \to \infty}  \Expect\bigl[  \sup_{    0 \le \SAtime \le T }  
\|  \clE_{\SAtime}^{\scerrorSymbol}(n)  \|^2 \bigr]  =0.
\end{aligned} 
\label{e:SA-approx-z} 
\end{equation}

\begin{proof}[Proof of~\Cref{t:FCLT}] 
We apply Theorem~4.1 of~\cite[Ch.~7]{ethkur05} to obtain the FCLT.  It is sufficient to establish the following bounds:
\begin{subequations}
\begin{align}
0  &=   \lim_{n\to\infty}   \Expect\bigl[  \sup_{0\le \SAtime\le T} \bigl\|   	\scerrorpull{\SAtime}{n}  -	\scerrorpull{\SAtime-}{n}  
\bigr\|^2\bigr]
\label{e:ethkur43}
   \\
0  &=   \lim_{n\to\infty}   \Expect\bigl[  \sup_{0\le \SAtime\le T} \bigl\|   	\scerrorpullDintDrift{\SAtime}{n}  -	\scerrorpullDintDrift{\SAtime-}{n}  
\bigr\|^2\bigr]
\label{e:ethkur44}
   \\
  0  &=   \lim_{n\to\infty}   \Expect\bigl[  \sup_{0\le \SAtime\le T} \bigl\|  \Sigma^{(n)}_\SAtime  - \Sigma^{(n)}_{\SAtime-}
  \bigr\|  \bigr]
\label{e:ethkur45}
   \\
0  &=   \lim_{n\to\infty}   \Prob \Bigl\{  \sup_{0\le \SAtime\le T} \bigl \| \scerrorpullDintDrift{\SAtime}{n} -
\int_0^\SAtime  [\tfrac{\diffalpha}{2} I+A^*]   \scerrorpull{t}{n}  \,dt  \bigr\| \ge \epsy \Bigr\},  &&  \epsy>0
\label{e:ethkur46}
   \\
\SAtime  \Sigma_MD &=   \lim_{n\to\infty} \Sigma^{(n)}_\SAtime 
\,, 
				\quad   &&  0\le\SAtime\le T.
\label{e:ethkur47}
\end{align} 
These five equations correspond to equations (4.3)--(4.7) of~\cite[Ch.~7]{ethkur05}.
\end{subequations}

Observe that~\eqref{e:ethkur43} is vacuous since 
$\scerrorpull{\varble}{n}  $ is continuous for each $n$.  
Proofs of the remaining  limits are established in order:

\wham{\eqref{e:ethkur44}:}
 We have $ 	\scerrorpullDintDrift{\SAtime}{n}  -	\scerrorpullDintDrift{\SAtime-}{n}  =    \sqrt{\alpha_{j-1}}\MD_{j-1}  $ if $\SAtime =\SAtime_j$ for some $j$, and zero otherwise.   Consequently,  for each $n$,
\[ 
 \Expect\bigl[  \sup_{0\le \SAtime\le T} \bigl\|   	\scerrorpullDintDrift{\SAtime}{n}  -	\scerrorpullDintDrift{\SAtime-}{n}  
\bigr\|^2\bigr]
\le   
  \Expect\bigl[   \max_{n \le j \le w_n}   \alpha_{j-1} \| \MD_{j-1}\| ^2 \bigr]    
  \le   
  \sqrt{
  \Expect\bigl[   \max_{n \le j \le w_n}   \alpha_{j-1}^2  \| \MD_{j-1}\|^4 \bigr] }.
\]
The right-hand side is bounded as follows:
\[
  \Expect\bigl[   \max_{n \le j \le w_n}   \alpha_{j-1}^2  \| \MD_{j-1}\|^4 \bigr] 
  \le
\sum_{j=n}^\infty 
  \Expect\bigl[     \alpha_{j-1}^2  \| \MD_{j-1}\|^4 \bigr]     \le  \bigl( \max_{j\ge n}   \Expect\bigl[     \| \MD_{j-1}\|^4 \bigr]  \bigr)
\sum_{j=n}^\infty   \alpha_{j-1}^2.
\]
The fourth moment is uniformly bounded  by~\Cref{t:noise-bound-l4},  and Assumption~(A1) then implies that the right-hand side vanishes as  $n\to\infty$.

\wham{\eqref{e:ethkur45}:}
  $ \Sigma^{(n)}_\SAtime  - \Sigma^{(n)}_{\SAtime-} = \alpha_{j-1} \Expect[\MD_{j-1} \MD_{j-1}^\transpose]  $ if $\SAtime =\SAtime_j$ for some $j$, and zero otherwise.
 Applying~\Cref{t:noise-bound-l4} once more gives,
$\displaystyle
\lim_{n\to\infty}  \sup_{\SAtime\ge 0}  \| \Sigma^{(n)}_\SAtime  - \Sigma^{(n)}_{\SAtime-} \|= 0$.

\wham{\eqref{e:ethkur46}:}
   Applying the definitions,
\[
 \scerrorpullDintDrift{\SAtime}{n} -
\int_0^\SAtime  [\tfrac{\diffalpha}{2} I+A^*]   \scerrorpull{t}{n}  \,dt   
=   \clE_{\SAtime}^{\scerrorSymbol}(n)   +   \int_0^{\SAtime } \bigl [ A ( \odestatepull{t}{n} )  - A^* \bigr]\scerrorpull{t}{n}    \, dt  ,
\]
and hence for each $n$ and $0\le \SAtime\le T$,
\[
\bigl\| 
 \scerrorpullDintDrift{\SAtime}{n} -
\int_0^\SAtime  [\tfrac{\diffalpha}{2} I+A^*]   \scerrorpull{t}{n}  \,dt     \bigr\| 
\le   \bigl \|  \clE_{\SAtime}^{\scerrorSymbol}(n)    \bigr\|  +   \int_0^{T }\bigl\|   \bigl [ A ( \odestatepull{t}{n} )  - A^* \bigr]\scerrorpull{t}{n}   \bigr\|  \, dt .
\]
The limit in~\eqref{e:SA-approx-z} 
tells us that
  $\sup_{0\le \SAtime\le T}  \bigl \|  \clE_{\SAtime}^{\scerrorSymbol}(n)    \bigr\|  $ converges to zero as $n\to\infty$, where the convergence is in $L_2$.  We next show that the second term vanishes in $L_1$.  

Thanks to Assumption~(A5a) we can apply Lipschitz continuity of $A$ to obtain the following bound, for some fixed constant $L_A<\infty$,
\[
\begin{aligned} 
  \Expect\Bigl[ \int_0^{T }\bigl\|   \bigl [ A ( \odestatepull{t}{n} )  - A^* \bigr]\scerrorpull{t}{n}   \bigr\|  \, dt  \Bigr] 
&  \le   L_A  \int_0^{T }  \Expect \bigl[ \bigl\|     \odestatepull{t}{n}  -\theta^* \bigr\|  \ \bigr\| \scerrorpull{t}{n}   \bigr\|  \bigr] \, dt  
\\
&
  \le   L_A  \sqrt{  \int_0^{T }  \Expect \bigl[ \bigl\|     \odestatepull{t}{n}  -\theta^* \bigr\| ^2 \bigr] \, dt  
   \int_0^{T }  \Expect \bigl[  \bigr\| \scerrorpull{t}{n}   \bigr\|^2 \bigr] \, dt    }.
  \end{aligned} 
\]
\Cref{t:BigConvergence} combined with~\Cref{t:BigBounds} implies that the first integral converges to zero.   
 \Cref{t:noise-bound-l4}  tells us that $ \Expect \bigl[  \bigr\| \scerrorpull{t}{n}   \bigr\|^2 \bigr] \, dt $ is bounded in $n$ and $t\ge 0$.
 
The limit~\eqref{e:ethkur47} is immediate from the definitions.
\end{proof}

\smallskip

\noindent 
{\bf Technical results for the CLT} 
Recall that the proof of~\Cref{t:CLT} rests on~\Cref{t:FCLT},~\Cref{t:tight} 
and~\Cref{t:zCLTbdd}.   It remains to prove the lemma,  which requires
the following bounds  whose proof is omitted.

\begin{lemma}
\label[lemma]{t:EAScontraction}
The following bound holds under~{\em (A5b)}: For a constant $b$ independent of $k$ or $T$:
\[
  \alpha_{ w_k^-}  /  \alpha_k 
  \le
 	\begin{cases}
	b\exp(T),  &   \rho=1
	\\
	b (1+ T)^{\rho/(1-\rho)},  &  \rho<1.
	\end{cases}
\]
Consequently,  with $\{\xi^-_k(T)  :  k\ge w_1\}$ defined in~\eqref{e:xi-k}, $\displaystyle
\lim_{T\to\infty}  \sup_{k\ge w_1}  \xi^-_k(T)   =0
$.
\end{lemma}

%
%
%

\begin{proof}[Proof of~\Cref{t:zCLTbdd}]
Part~(i) is immediate from~\Cref{t:FCLT} and~\Cref{t:tight}.
For~(ii) we first establish boundedness of  $  \|  z_k   \|_4 $, which is based on~\eqref{e:c-contract}
written in ``backwards form'':
\begin{equation} 
\begin{aligned} 
\| z_k   \|  &  \le  \|  z^{(w^-_k)}_k \|    +  \xi_{w^-_k}(T)   \|  z_{w^-_k}  \| 
\quad
\text{and}
\quad
\| z_k   -  z^{(w^-_k)}_k \|   &  \le     \xi^-_k(T)    \|  z_{w^-_k}  \|  .
\end{aligned} 
\label{e:c-contract-}
\end{equation}
 For given $n \ge 1$,   suppose that the integer $k$ satisfies $m_n < k \le m_{n+1}$,  with $\{m_n\}$ used in~\eqref{e:preFCLTgivesCLT}.   We then have  $m_{n-1} < w_k^- \le m_n$, and from the first bound in~\eqref{e:c-contract-},
\[
\begin{aligned}
b^z_{n+1}
\eqdef
\max_{m_n < k \le m_{n+1}}  \| z_k   \|_4   & \le   \max_{m_{n-1} < j \le m_{n}}  \|  z^{(w^-_j)}_j \|_4   +  \bar\xi^-(T)     \max_{m_{n-1} < j \le m_{n}}   \|  z_j  \|_4 
\\
 & \le  B^z   +  \bar\xi^-(T)    b^z_n,
\end{aligned} 
\]
where $B^z$ is a finite constant bounding each of $  \|  z^{(w^-_j)}_j \|_4 $,   and $ \bar\xi^-(T) = \sup_{k\ge w_1}  \xi^-_k(T)$.
\Cref{t:EAScontraction} tells us that we can choose $T$ so that $ \bar\xi^-(T) < 1$, which implies that $\{ \| z_k   \|_4 : k\ge 1 \}$ is   bounded.
 The second bound in~\eqref{e:c-contract-} gives the final conclusion:
\[
\limsup_{k\to\infty} \| z_k   -  z^{(w^-_k)}_k \|_4  \le    \bdds{t:zCLTbdd}_T \eqdef  \bar\xi^-(T)  \limsup_{k\to\infty}   \|  z_k \|_4 .
 \] 
\end{proof}

\subsection{Optimizing the Rate of Convergence}
\label{s:PR}

The covariance matrix~\eqref{e:SigmaMD} appears in  consideration 
of the pair of martingales,
\begin{equation}
\MartPR_K   \eqdef \sum_{k=1}^K    \MD_k\,,  
\qquad
\MartPR^*_K   \eqdef \sum_{k=1}^K    \MDstar_k\,,  
 \qquad 1\le K <\infty \,,
\label{e:MartPR}
\end{equation}
in which  $\MDstar_{n+1} \eqdef \haf(\theta^*, \Phi^*_{n+1}) - \Expect[\haf(\theta^*, \Phi^*_{n+1})  \mid \clF_{n}]  $ for any $n\ge 0$ (in analogy with~\eqref{e:noise-decomp}).   We have:
 \[
  \Sigma_{\Delta}^*
  =
   \lim_{n\to\infty}\frac{1}{n}  \Cov(\MartPR_{n} ) =  \lim_{n\to\infty}\frac{1}{n}  \Cov(\MartPR^*_{n} ).
 \] 
The following result implies~\Cref{t:PR}:

\begin{proposition}
\label[proposition]{t:PRbig}
Under the assumptions of~\Cref{t:CLT} there is a constant $\bdd{t:PRbig}$, depending on the initial condition $(\theta_0,\Phi_0)$,  such that~\eqref{e:PRbig} holds.
\end{proposition}

\begin{subequations}

\begin{proof}
We begin with a version of~\eqref{e:SA-y}:
\[
\tiltheta_{k+1} = \tiltheta_k + \alpha_{k+1}[ A^* \tiltheta_k    + \clE_k^T   + \Delta_{k+1} ] \,,\quad k\ge 0,
\]
where $ \clE_k^T = \barf(\theta_k) - A^* \tiltheta_k $, which satisfies $\|  \clE_k^T  \| \le L_A \| \tiltheta_k \|^2$,  with $L_A$ the Lipschitz constant for $A(\theta)$.    Dividing each side by $ \alpha_{k+1}$  and summing from $k=1$ to $n$ gives,
\[
\frac{1}{  \alpha_{n+1} } \tiltheta_{n+1} -\frac{1}{  \alpha_1 }  \tiltheta_1   = A^*  \sum_{k=1}^n   \tiltheta_k    
+ \sum_{k=1}^n  [  \clE_k^T   + \Delta_{k+1}  +  \diffalpha_k   \tiltheta_k   ]       \,, 
\]
with $\{ \diffalpha_k\}$ defined in~(A1),  which for the choice of step-size assumed here gives $\diffalpha_k \le  \rho k^{\rho - 1} $.   Rearranging terms and dividing by $n$,  
\[
  A^*  \tilthetaPR_n  +   n^{-1}\MartPR^*_{n+2} 
=
  n^{-1}\MartPR^*_{n+1}  +
\frac{1}{  n \alpha_{n+1} } \tiltheta_{n+1} -\frac{1}{  n \alpha_1 }  \tiltheta_1   -  \frac{1}{n}    \sum_{k=1}^n  [  \clE_k^T   + \Delta_{k+1}  +  \diffalpha_k   \tiltheta_k   ].
\]
where   $\MartPR^*_{n+1}   $ is defined in~\eqref{e:MartPR}.
Applying~\Cref{t:noise-decomp}  gives:
\begin{align}
 \| A^*  \tilthetaPR_n  +  n^{-1}\MartPR^*_{n}  \| \le    \frac{1}{n}  &   \sum_{k=1}^n \bigl [  L_A \| \tiltheta_k \|^2
 				 +   \alpha_{k+1}  \|  \Oops_{n+1} \|    \bigr] 
 \label{e:PRerror1}
 \\
 &
 +
   \frac{1}{n}   \|
\frac{1}{    \alpha_{n+1} } \tiltheta_{n+1} -\frac{1}{    \alpha_1 }  \tiltheta_1 \|   
+     \frac{1}{n}     \sum_{k=1}^n \diffalpha_k  \| \tiltheta_k \|  
 \label{e:PRerror2}
 \\
 &+    \frac{1}{n}  \| \MartPR_{n+1}  -  \MartPR^*_{n+1}\|
  \label{e:PRerror3}
 \\
 &+  \frac{1}{n}\|\clT_{n+1} - \clT_1 \|  +  \frac{1}{n}  \| \MartPR^*_{n+1} -  \MartPR^*_{n}  \| 
 \label{e:PRerror4}
\end{align}  
 
On taking   $L_2$ norms  we obtain,  
 \[
\Expect[  \| A^*  \tilthetaPR_n  +  n^{-1}\MartPR^*_{n+1}  \|^2] ^{1/2}
\le 
  \epsy_n^a     +  \epsy_n^b     +  \epsy_n^c +  \epsy_n^d,
  \]
  in which $ \epsy_n^a$ is the $L_2$ norm of the right-hand side of~\eqref{e:PRerror1},  
   $ \epsy_n^b$ is the $L_2$ norm of the right-hand side of~\eqref{e:PRerror2},     $ \epsy_n^c$ is the $L_2$ norm of the right-hand side of~\eqref{e:PRerror3},  and   $ \epsy_n^d$ is the $L_2$ norm of the right-hand side of~\eqref{e:PRerror4}. 
 Repeated application of the bound $  \Expect[   \|  \tiltheta_n \|^r] \le O(\alpha_n^{r/2} ) $ for $1\le r\le 4$, gives  
   $ \epsy_n^a  =    O\bigl( \alpha_n  \bigr)$,
     $ \epsy_n^b  =    O\bigl( n^{-1 + \rho/2}  \bigr)$,
     $ \epsy_n^c  =    O\bigl( n^{-(1 + \rho)/2}  \bigr)$
  and    $ \epsy_n^d  =    O\bigl( n^{-1 }  \bigr)$.  
  
  The error term $ \epsy_n^a$   dominates, which completes the proof.
\end{proof}

\end{subequations}

\subsection{Counterexample}

The proof of~\Cref{t:unbounded-mom} begins with a representation of~\eqref{e:sa-mm1}:   
\begin{equation}
	\theta_{n} =\theta_0 \prod_{k=1}^{n} [1 + \alpha_k(Q_k - \eta - 1)] + \sum_{k=1}^{n} \alpha_k \CEdist_k \prod_{l=k+1}^{n}[1 + \alpha_l(Q_l - \eta -1)] 
\label{e:unbounded_step1}.
\end{equation}
The two elementary lemmas that follow are used to obtain lower bounds.
The first is obtained based on comparison of the sum with the integral $ \int_{n/2+1}^{n+1}\frac{1}{x}\, dx$.

\begin{lemma}
	\label[lemma]{t:half-harm-series}
	For $n\geq 1$,  $\displaystyle	\sum_{k  = n/2+1}^n \frac{1}{k } \geq \log 2 - 1/n $.
\end{lemma}

 Denote  $\sigma = \log(1+2\delta)/(2\delta)$  with  $\delta =\mu - \alpha $.     
A first and second order Taylor series approximations of the logarithm leads to the following:
\begin{lemma}
\label[lemma]{t:exp-bound-delta}
If $\delta>0$ then   
$\sigma   \geq 1 - \delta $  and $
		\exp(\sigma x) \leq 1 + x \  $   for $  x\in[0\,, 2\delta] $.    Moreover,  $\sigma   >  \delta/(4\alpha  \log 2)$ if $\alpha > 1/3$.     
\end{lemma}

%
%
%
%

%

The two lemmas provide bounds on~\eqref{e:unbounded_step1}, subject to constraints on the sample paths of the queue.
The next step in the proof is to present a useful constraint, and bound the probability that it is satisfied.  For this we turn to the functional LDP for a reflected random walk.  

Define for any $n\ge 1$  the piecewise constant function of time $q^n_t = \frac{1}{n} Q_{\lfloor n t \rfloor}$.
For each $\epsy\in [0,1/2)$ we define a constraint on the scaled process as follows:
Let $\clX$ denote the set of all c\`adl\`ag functions $\phi\colon[0,1]\to\Re_+$,   and  let $\clR_\epsy \subset \clX$ consist of those functions satisfying the following strict bounds:  
\begin{equation} 
  \delta \epsy + \min \{ \delta t ,   \delta (1-t) \}    \le q_t  \le 2\delta t  \qquad\text{for all $t\in [\epsy, 1]$ when $q\in\clR_\epsy$.}
\label{e:clR}
\end{equation}
If $q^n \in\clR_\epsy$ this implies a bound on the term within the product in~\eqref{e:unbounded_step1}: 
\begin{lemma}
\label[lemma]{t:LDP_Q_Bdd}
For given $\epsy\in (0,1/2)$,  suppose that $q^n   \in\clR_\epsy$.
Then, there is $ n_\epsy$ such that the following bounds hold for each $n\ge n_\epsy$,   and each $\ell$ satisfying $\epsy n \le \ell  \le n$:
\begin{equation}
\min \{ \delta    ,   \delta ( n \alpha_\ell  -1 ) \}   \le 
\alpha_\ell [Q_\ell  - \eta - 1  ] \le 2\delta   .
\label{e:clRQ}
\end{equation}
\end{lemma}

\begin{proof}
Fix $\ell\geq \epsy n$ and denote $t=\ell/n$.  Under the 
assumption that $\alpha_\ell =1/\ell$,
\[
\frac{1}{t} q^n_t =  \frac{1}{t}  \frac{1}{n} Q(  \ell   )  =     \alpha_\ell Q(  \ell   )  .
\] 
 The upper bound is then immediate:  If $q^n\in  \clR_\epsy$ then,
\[
\alpha_\ell [Q_\ell  - \eta - 1  ] \le     \alpha_\ell Q(  \ell   )  =  \frac{1}{t}  q^n_t    \le  2\delta  .
\] 
The lower bound proceeds similarly, where we let  $n_\epsy =  (\eta +1)/(\epsy^2\delta)$.
For $\ell\geq\epsy n$,
\[
\begin{aligned} 
\alpha_\ell [Q_\ell  - \eta - 1  ]    = \frac{1}{t}   q^n_t    - \frac{1}{\ell} (\eta+1)   
&\ge  \frac{1}{t}   q^n_t    - \frac{1}{\epsy n} (\eta+1)
\\
&\ge  \frac{1}{t}  \Bigl(  \delta \epsy + \min \{ \delta t ,   \delta (1-t) \}  \Bigr)  - \frac{1}{\epsy n} (\eta+1)
\\
&\ge   \frac{1}{t}   \min \{ \delta t ,   \delta (1-t) \},  \qquad \text{when $n\ge n_\epsy$},
\end{aligned} 
\]
where the final inequality uses the bound $  \frac{1}{t}  \delta \epsy \ge  \frac{1}{\epsy n} (\eta+1)$ for $n\ge n_\epsy$ and $t\le 1$.
The proof is completed on substituting $1/t = n/\ell$.
\end{proof}

We now have motivation to bound the probability that  $q^n   \in\clR_\epsy$.

The log-moment generating function for the distribution of $D_{k+1}$ 
appearing in~\eqref{e:RRW}
is  $\Lambda(\vartheta) = \log (\alpha e^\vartheta + \mu e^{-\vartheta}) $.   Its convex dual is finite only for $|v|\le 1$, with
\begin{equation}
I(v) = \half (1+v) \log\Bigl( \frac{1+v}{2\alpha} \Bigr)  	
	+
	\half (1- v) \log\Bigl( \frac{1-v}{2\mu} \Bigr)  \,,\qquad   -1\le v\le 1 \, .
\label{e:MM1rate}
\end{equation}

\begin{lemma}
\label[lemma]{t:MM1rate}
For each $\epsy\in (0,1/2)$ the following holds:  
\begin{equation}
\lim_{n\to\infty} \frac{1}{n}\log \bigl(  \Prob \{ q^n \in \clR_\epsy \}   \bigr)  
	= - \{ \epsy I(\delta (1+\epsy) ) + (\tfrac{1}{2} -\epsy) I(\delta) \} 
	\geq  -   \half \frac{1}{\alpha}   \delta^2   +  O(\epsy^2) .
\label{e:clR_limit}
\end{equation}
\end{lemma} 

\begin{proof}
The functional LDP for the sequence $\{q^n\}\subset \clX$ is obtained from the LDP for  $\{ D_k : k\ge1 \}$  via the
contraction principle~\cite{ganoco02}, which for~\eqref{e:clR_limit} gives
\[
\lim_{n\to\infty} \frac{1}{n}\log \bigl(  \Prob \{ q^n \in \clR_\epsy \}   \bigr)  = - \inf_{\phi \in \clR_\epsy}  \int_0^1 I( \ddt \phi_t) \, dt.
\]
Convexity of $I$ implies that the optimizer $\phi^*$ is piecewise linear as illustrated in~\Cref{f:tent-path}, with slope $\delta (1+\epsy)$ on the interval $[0,\epsy)$,   $\delta$  on the interval $(\epsy,\half)$,  and $-\delta$ on the remainder of the interval.    The proof of the limit in~\eqref{e:clR_limit} is completed on recognizing that $I(-\delta) =0$.

The inequality in~\eqref{e:clR_limit} is established in two steps. First, $ \epsy I(\delta (1+\epsy) )  =   \epsy   I(\delta  )  + O(\epsy^2) $.  
The proof is completed on substituting $v=\delta$ in~\eqref{e:MM1rate} to obtain $I(\delta) = \delta \log( 1 +  {\delta}/{\alpha}) \leq \delta^2 /\alpha$.
\end{proof}

\begin{proof}[Proof of~\Cref{t:unbounded-mom}]
The almost sure convergence in~(i) follows from~\Cref{t:BigConvergence}.

For~(ii), let $\epsy \in (0,\half)$ be fixed and denote $n_0 =\lceil\epsy n\rceil$ (the least integer greater than or equal to $\epsy n$). Given $\{\CEdist_n\}$ is i.i.d.\ with zero mean and unit variance, we have for each $n \ge  n_0$,
\[
\begin{aligned}
		\Expect[\theta_n^2] 
		&\geq\Expect\Bigl[  \alpha_{n_0-1}^2 \CEdist_{n_0-1}^2 \prod_{k=n_0}^{n}[1 + \alpha_k(Q_k - \eta - 1)]^2  \Bigr] 
				= \alpha_{n_0-1}^2 \Expect [\clP^Q_n],
		 \\
		\clP^Q_n &\eqdef \prod_{k=n_0}^{n}[1 + \alpha_k(Q_k - (\eta +1) )]^2.
\end{aligned}
\]
The remainder of the proof consists of establishing the following bound:
\begin{equation}
	\label{e:prod-prob}
\begin{aligned} 
		\log
		\Expect\bigl[ \clP^Q_n  \bigr] 
		\geq  	\log	\Expect\bigl[\ind \{ q^n\in\clR_\epsy \} \clP^Q_n  \bigr] 
\geq  \log  \Prob\{ q^n\in\clR_\epsy \}  +    [ 2 n ( \log 2 -\epsy) ]\delta \sigma .
\end{aligned} 
\end{equation}
Since $\epsy\in(0,\half)$ can be arbitrarily close to zero,~\eqref{e:clR_limit} combined with~\Cref{t:exp-bound-delta} leads to the conclusion that the right-hand side of~\eqref{e:prod-prob} is unbounded as $n\to \infty$.
  
\Cref{t:LDP_Q_Bdd} gives    
$\alpha_k[Q_k - (\eta +1)]\in [0, 2\delta]$ 
for $n_0\le k \le n$ if $q^n\in\clR_\epsy$, and so 
by~\Cref{t:exp-bound-delta},  
\[
	\clP^Q_n \geq \exp\Bigl(\sigma \bigl[ \sum_{k=n_0}^n 2\alpha_k[Q_k  - (\eta + 1)]\bigr]\Bigr) \,,  \quad q^n\in\clR_\epsy \, .
\]
A second application of~\Cref{t:LDP_Q_Bdd} provides bounds on each term in the sum:  
\[
 \alpha_k[Q_k - (\eta +1)]
 \ge 
\begin{cases}
\delta,  &  n_0\le k\le n/2,
\\
\delta [ n \alpha_k -1 ],   & k>n/2,
\end{cases}
\qquad\qquad
 \text{whenever $q^n\in\clR_\epsy$.}
\]
Putting these bounds together gives,
\[
\begin{aligned}
	\clP^Q_n
		& \ge \exp\Bigl(\sigma   \sum_{k=n_0}^{n/2} 2\alpha_k[Q_k  - (\eta + 1)] +\sigma \sum_{k=n/2 + 1}^{n} 2\alpha_k[Q_k  - (\eta + 1)] \Bigr) 
\\
		& \ge \exp\bigl( (n-2n_0 + 2)\delta\sigma \bigr) \exp\Bigl( -\delta\sigma n + 2\sigma\delta n\sum_{k = n/2 + 1}^{n} \alpha_k \Bigr)
		\ge \exp\bigl( [2n\log 2   -2n\epsy]\delta \sigma\bigr) \, ,
\end{aligned}
\]
where the last inequality follows from $n\sum_{k = n/2 + 1}^{n} \alpha_k \geq n \log 2 -1 $ (see~\Cref{t:half-harm-series}),
and the substitution $n_0=n\epsy$.
This establishes~\eqref{e:prod-prob}
and completes the proof.
\end{proof}

\end{document}